\documentclass{amsart}
\usepackage{amsmath}%
\usepackage{amsfonts}%
\usepackage{amssymb}%
\usepackage{graphicx}
\usepackage{mathtools, tikz, float,hyperref}

\usepackage{color}

\newcommand{\C}{\mathcal{C}}

\newcommand{\PC}{\mathbb{P}\mathcal{C}}

\newcommand{\G}{\mathcal{G}}

\newcommand{\h}{\mathbb{H}^2}
\newcommand{\M}{\mathcal{M}}
\newcommand{\N}{\mathbb{N}}

\newcommand{\p}{\mathcal{P}}

\newcommand{\R}{\mathbb{R}}
\renewcommand{\S}{\mathcal{S}}
\newcommand{\T}{\mathcal{T}}

\newcommand{\Z}{\mathbb{Z}}

\DeclareMathOperator{\sys}{sys}
\DeclareMathOperator{\Col}{Col}

\newtheorem{theorem}{Theorem}[section]
\theoremstyle{plain}
\newtheorem{lem}[theorem]{Lemma}

\newtheorem{cor}[theorem]{Corollary}
\newtheorem{rem}[theorem]{Remark}
\newtheorem{prop}[theorem]{Proposition}

\newtheorem{ques}{Question}

\theoremstyle{definition}
\newtheorem{defi}[theorem]{Definition}

\newtheorem*{MainThm}{Theorem \ref{thm:Main}}
\newtheorem*{ThmShortCurve}{Theorem \ref{thm:Short_curves}} 
\newtheorem*{ThmThickPart}{Theorem \ref{thm:Thick_part_guarantee}}

\newcounter{jenyacomments}

\title{A length comparison theorem for geodesic currents}
\author{Jenya Sapir}

\begin{document}

\begin{abstract}
 We work with the space $\C(S)$ of geodesic currents on a closed surface $S$ of negative Euler characteristic. By prior work of the author with Sebastian Hensel, each filling geodesic current $\mu$ has a unique length-minimizing metric $X$ in Teichm\"uller space. In this paper, we show that, on so-called thick components of $X$, the geometries of $\mu$ and $X$ are comparable, up to a scalar depending only on $\mu$ and the topology of $S$. We also characterize thick components of the projection using only the length function of $\mu$.
\end{abstract}

\maketitle

\section{Introduction}
Let $S$ be a closed, oriented, finite type surface with negative Euler characteristic. The space of geodesic currents, $\C(S)$, contains many of the structures one might wish to study on $S$. For example, it contains the set of closed curves up to homotopy, as well as an embedded copy of Teichm\"uller space, $\T(S)$. These sets are united by an intersection pairing $i(\cdot, \cdot)$ on $\C(S)$. If $\mu, \nu \in \C(S)$ represent two closed curves, then $i(\mu, \nu)$ is just their geometric intersection number. And if $\mu$ represents a metric, while $\nu$ represents a closed curve, then $i(\mu, \nu)$ is the length of the geodesic representative of $\nu$ in the metric $\mu$.

We let $\C_{fill}(S)$ denote the set of \textit{filling} currents, that is, those currents that have positive intersection with all other geodesic currents. Then all currents representing metrics are examples of filling currents. We extend the notion of length function from currents representing metrics to all of $\C_{fill}(S)$. If $\mu \in \C_{fill}(S)$, we define its length function $\ell_\mu : \C(\S) \to \R $  so that
\[
 \ell_\mu(\nu) : = i(\mu, \nu)
\]
for all $\nu \in \C(S)$.

In \cite{HS21}, Hensel and the author show that there is a continuous projection 
\[
 \pi: \C_{fill}(S) \to \T(S)
\]
that minimizes the length of $\mu$ in the sense that $\ell_{\pi(\mu)}(\mu) < \ell_X(\mu)$ for all $X \neq \pi(\mu) \in \T(S)$. We call $\pi(\mu)$ the \textbf{length minimizer} of $\mu$. The goal of this paper is to compare the length function of a filling current $\mu$ to the length function of its length minimizer $\pi(\mu)$. 

To state our result precisely, let $c_b$ be the Bers constant. Then any two curves of length less than $c_b$ with respect to $\pi(\mu)$ are disjoint. Cut $\pi(\mu)$ along all such curves. Then the connected components of the result are the \textbf{thick components} of $\pi(\mu)$. We describe how to identify these components in Theorems \ref{thm:Short_curves} and \ref{thm:Thick_part_guarantee} below.

In \cite{BIPP}, they show that if $\mu \in \C_{fill}(S)$, then its \textbf{systolic length} $\sys(\mu) = \inf_{\alpha \subset S} i(\mu, \alpha)$ is positive, where the infimum is taken over all simple closed curves on $S$. For any subsurface $Y$ of $S$, they also define the $Y$-systolic length to be 
\[
 \sys_Y(\mu)= \inf_{\alpha \subset Y} i(\mu, \alpha)
\]
where the infimum is taken over all essential, non-peripheral simple closed curves in $Y$. Then we show the following:

Then we show that the geometries of $\mu$ and $\pi(\mu)$ are comparable on thick components:
\begin{theorem}
\label{thm:Main}
 Let $\mu \in \C_{fill}(S)$, and let $\pi(\mu)$ be its length minimizer. Let $Y$ be a thick component of $\pi(\mu)$. Then for all essential, non-peripheral simple closed curves $\alpha$ in $Y$,
 \[
\frac{\ell_\mu(\alpha)}{\ell_{\pi(\mu)}( \alpha)} \asymp \sys_Y(\mu)
 \]
 where $\sys_Y(\mu)$ is the $Y$-systolic length of $\mu$, and the constants depend only on the Euler characteristic $\chi(S)$.
\end{theorem}

\subsection{Notation}
Given quantities $A, B$, we use say $A\prec B$ with constants depending only on $C$ if there is a constant $c$ depending only on $C$ so that $A \leq c B$. Likewise, we say $A \succ B$ if $A \geq c' B$ for some $c'>0$ depending only on $C$, and $A \asymp B$ if $A \prec B$ and $B \prec A$.

\subsection{Identifying thick components of $\pi(\mu)$}
We also characterize when curves in $\pi(\mu)$ are short, and when subsurfaces are thick, in terms of the length function of $\mu$.

First, it turns out that a simple closed curve $\alpha$ is short in $\pi(\mu)$ if all simple closed curves $\beta$ crossing $\alpha$ are \textit{relatively} long with respect to $\mu$.
\begin{theorem}
\label{thm:Short_curves}
 For every $\epsilon > 0$, there are constants $N_1, N_2$ so that for any $\mu \in PC_{fill}(S)$, any simple closed curve $\alpha$, and any simple closed curve $\beta$ with $i(\alpha, \beta) \geq 1$,
 \begin{enumerate}
  \item If $\ell_{\pi(\mu)}(\alpha) < \epsilon$, then 
  \[
   i(\mu, \beta) > N_1 i(\mu, \alpha)
  \]
\item If
  \[
   i(\mu, \beta) > N_2 i(\mu, \alpha)
  \]
  then $\ell_{\pi(\mu)}(\alpha) < \epsilon$.
 \end{enumerate}
\end{theorem}
Note that this theorem is a coarse biconditional. The constant $N_1$ grows coarsely like the width of the collar about $\alpha$, while $N_2$ comes from the constant in Theorem \ref{thm:Main}, and is more mysterious.

%

Once we know which curves are short, the following theorem gives a more practical characterization of thick components of $\pi(\mu)$.
\begin{theorem}
\label{thm:Thick_part_guarantee}
 Let $Y \subset \pi(\mu)$ be a subsurface so that $\ell_{\pi(\mu)}(\beta) < c_b$ for each boundary component $\beta$ of $Y$, where $c_b$ is the Bers constant. Then for each essential simple closed curve $\alpha$ in $Y$,
\[
 \ell_{\pi(\mu)}(\alpha) \succ 1
\]
if and only if there exists a marking $\Gamma$ of $Y$ so that
 \[
  i(\mu, \gamma) \asymp \sys_Y(\mu)
 \]
for all $\gamma \in \Gamma$, where all constants depend only on $S$.
\end{theorem}

\subsection{Outline and idea of proof}
The paper is organized as follows. 
\begin{enumerate}
\item In Section \ref{sec:First_Thick_Thin}, we prove Proposition \ref{prop:Thick-thin_without_collar_lemma}. This is a preliminary version of Theorem \ref{thm:Main}. It shows that, if $\mu \in \C_{fill}(S)$ and $Y$ is a thick component of $\pi(\mu)$, then 
  \[
   \frac{\sys_Y(\mu)^2} {i(\mu, \Gamma)} \prec \frac{i(\mu, \alpha)}{i(\pi(\mu), \alpha)} \prec i(\mu, \Gamma)
  \]
  where $\Gamma$ is a shortest marking for $Y$ with respect to $\pi(\mu)$ (see Section \ref{sec:First_Thick_Thin} for a definition). We spend the rest of the paper to show that $i(\mu, \Gamma) \asymp \sys_Y(\mu)$.
  
  This result is completely independent from the rest of the paper. In fact, it relies only on analogues of results in \cite{RafiTT}, which we prove for geodesic currents rather than flat structures. 
 \item In Sections \ref{sec:Mixed_collar_idea} - \ref{sec:Mixed_collar_theorem} we prove Theorem \ref{thm:Full_Collar}, giving a mixed collar lemma for geodesic currents and their length minimizers. Roughly, this theorem says the following. Let $\mu \in \C_{fill}(S)$, and let $\alpha$ and $\beta$ be two simple closed curves with $i(\alpha, \beta) \geq 1$. If $\alpha$ is short and $\beta$ is long in $\pi(\mu)$, then $\mu$ must intersect $\beta$ much more than $\alpha$. An idea of the proof is given in Section \ref{sec:Mixed_collar_idea}.
 \item In Sections \ref{sec:Mixed_length_comparisons} and \ref{sec:Another_collar}, we prove a few more mixed collar theorems that follow from Theorem \ref{thm:Full_Collar}.
 \item In Section \ref{sec:Full_Thick_Thin}, we prove Proposition \ref{lem:Hyp_short_marking_bounded_by_mu_systole}, which says that the hyperbolically short marking $\Gamma$ of a thick subsurface $Y$ is actually $\mu$-short:
  \[
   i(\mu, \Gamma) \asymp \sys_Y(\mu)
  \]
The theorem then follows directly from Propositions \ref{prop:Thick-thin_without_collar_lemma} and \ref{lem:Hyp_short_marking_bounded_by_mu_systole}.
\item In Section \ref{sec:When_curves_short}, we prove Theorems \ref{thm:Short_curves} and \ref{thm:Thick_part_guarantee} that show how to identify short curves and thick subsurfaces of $\pi(\mu)$. 
\item Section \ref{sec:Appendix} is an appendix proving a few identities about the collars of geodesics in hyperbolic surfaces.

\end{enumerate}

\subsection{Acknowledgements}
The author would like to thank D\'idac Mart\'inez-Granado, Giuseppe Martone and Sebastian Hensel for many helpful discussions.
\section{Background and connection to related results}

\subsection{Geodesic currents}
Geodesic currents on $S$ were first defined by Bonahon in \cite{Bonahon85} as follows. Fix a complete hyperbolic metric $X$ for $S$. Identify its universal cover with $\h$. Let $\G$ be the set of all unparameterized, unoriented geodesics in $\h$. Each geodesic in $\h$ is uniquely determined by its endpoints on the circle $S^1$ at infinity. Thus $\G$ can be identified with $S^1 \times S^1 \setminus \Delta/ \sim$, where $\Delta$ is the diagonal in $S^1 \times S^1$, and $\sim$ is the relation so that $(a,b) \sim(b,a)$. Since $\pi_1(S)$ acts on $\h$ by isometries, it also acts on $\G$. A \textbf{geodesic current} is a $\pi_1(S)$-invariant, Borel measure on $\G$. We let $\C(S)$ denote the space of currents and endow it with the weak* topology. It turns out that $\C(S)$ is independent of the choice of metric $X$. In fact, the space of currents defined using any other metric $Y$ is H\"{o}lder equivalent to $\C(S)$ \cite{Bonahon85}.

The set of closed geodesics on $X$ embeds into the space of currents as follows. Given any closed geodesic $\gamma$ on $X$, we can lift it to a subset $\tilde \gamma \in \G$. Then the current corresponding to $\gamma$ will be the Dirac measure on $\tilde \gamma$. Abusing notation, we will still let $\gamma$ refer to this current. There is a natural action of $\R^+$ on $\C(S)$ by scaling. It turns out that the $\R^+$ orbit of the set of closed geodesics is dense in $\C(S)$ \cite{Bonahon85}. Moreover, Bonahon also shows that the geometric intersection function $i(\cdot, \cdot)$ extends continuously from pairs of closed geodesics to a bilinear, symmetric intersection form on pairs of currents. 

By work of Bonahon, Teichm\"uller space also embeds in $\C(S)$. Moreover, suppose $\gamma$ is a closed geodesic and $Y \in \T(S)$. We abuse notation slightly, and let $\gamma$ and $Y$ still denote the corresponding currents. Then,
\[
i(\gamma, Y) = \ell_Y(\gamma)
\]
In fact, many spaces of metrics on $S$ embed into $\C(S)$: spaces of singular flat metrics \cite{DLR10}, metrics of variable negative curvature \cite{Otal90}, and others. The embeddings are all characterized by the fact that intersecting a metric with a closed curve $\gamma$ gives the length of the geodesic representative of $\gamma$ with respect to that metric.

\subsection{Relationship to the thick-thin decomposition for quadratic differentials}
This theorem is inspired by an analogous result of Rafi in \cite{RafiTT}. Rafi shows the following. Each holomorphic quadratic differential on $S$ induces a singular flat metric $q$. By \cite{DLR10}, this set of singular flat metrics can be viewed as filling currents in the sense that if $q$ is a flat structure and $\gamma$ is a curve, then $i(q,\gamma)$ is the length of the $q$-geodesic representative of $\gamma$ in $q$. Moreover, $q$ has a unique hyperbolic metric $X$ in its conformal class.  So, we get a projection from singular flat structures (which we will think of as a subset of $\C_{fill}(S)$) to $\T(S)$.

Given a hyperbolic metric $X$ conformally equivalent to a flat structure $q$, Rafi considers the thick components of $X$, defined the same way as above. He shows that if $Y$ is a thick component of $X$, and $\alpha$ is any essential, non-peripheral simple closed curve in $Y$, then, in the language of this paper,
\[
 \frac{\ell_q(\alpha)}{\ell_X(\alpha)} \asymp \sys_Y(q)
\]
where the constants depend only on $\chi(S)$ \cite[Theorem 1.3]{RafiShortCurves}. Thus, the statement of the theorem of Rafi is analogous to ours, although the projections to Teichm\"uller space are different. It would be interesting to see how different the two projections are.
\begin{ques}
 Given a singular flat metric $q$ coming from a holomorphic quadratic differential, let $X$ be the hyperbolic metric in its conformal class, and let $\pi(q)$ be its length minimizing metric. Are $X$ and $\pi(q)$ at a uniformly bounded distance, with respect to some natural metric on $\T(S)$?
\end{ques}

The method of proof of Theorem \ref{thm:Main} is also inspired by the proof in \cite{RafiTT}. Rafi's proof relies crucially on a mixed collar lemma for quadratic differentials, which he proves in \cite{RafiShortCurves}. His theorem says that if $X$ is the hyperbolic metric in the conformal class of a holomorphic quadratic differential $q$, then for any two intersecting simple closed curves $\alpha$ and $\beta$, 
\[
\ell_q(\alpha) \geq d \cdot \ell_q(\beta)
\]
where $d = d(\ell_X(\beta))$ depends only on the length of $\beta$ with respect to $X$, and on $\chi(S)$ \cite[Theorem 1.3]{RafiShortCurves}.

The first half of this paper involves proving an analogous result (Theorem \ref{thm:Full_Collar}). If $\alpha$ and $\beta$ are intersecting simple closed curves, and $\pi(\mu)$ is the length minimizer of some $\mu \in \C_{fill}(S)$, we show that
\[
 \ell_\mu(\alpha) \geq \frac{D}{i(\alpha, \beta)} \ell_\mu(\beta)
\]
where $D = D(\ell_{\pi(\mu)}(\beta))$ depends only on the length of $\beta$ with respect to $\pi(\mu)$.
The factor of $i(\alpha, \beta)$ in the denominator is due to the differences in the methods of proof. We use a careful geometric argument to show how the length of $\mu$ will change if we pinch a metric $X$ along $\alpha$. If $\mu$ intersects $\alpha$ much less than $\beta$, then pinching $\alpha$ increases the length of $\mu$ coming from crossing $\alpha$, while decreasing the contribution to length coming from intersecting $\beta$. Balancing these two effects gives our mixed collar lemma.

We can prove a version of the length comparison theorem (Proposition \ref{prop:Thick-thin_without_collar_lemma}) using techniques similar to those of \cite[Theorem 1]{RafiTT}. We need this result to prove the main theorem, however its statement is rather unsatisfying on its own. To get the full version of Theorem \ref{thm:Main}, we have to work specifically with length minimizing metrics, and the techniques are rather different.

\subsection{Connection to Higher Teichm\"uller Spaces} Giuseppe Martone recently made us aware of a similar length comparison theorem in Higher Teichm\"uller theory. Take the space of representations $\rho: \pi_1(S) \to PSL(3,\R)$. Consider its \textit{Hitchin component} $\T_H$, which is a connected component of discrete, faithful, orientation-preserving representations. There is a map from $\T_H$ to $\C_{fill}(S)$. This map is natural in the sense that, if $\rho \in \C_{fill}(S)$ represents an element of $\T_H$, then for any closed curve $\gamma$, $i(\rho, \gamma)$ gives the so-called Hilbert length of $\gamma$ with respect to $\rho$ \cite{BCL,MZ}.

Labourie \cite{Labourie_projection} and Loftin \cite{Loftin} showed independently that there is a mapping class group-invariant projection from $\T_H$ to Teichm\"uller space. This is part of a much larger, quite active research program: see \cite[Conjecture 14]{Wienhard} for an overview of the broader context. 

As noted in \cite[Lemma 5.1]{DM20}, work of Tholozan \cite[Theorem 3.9, Corollary 3.10]{Tholozan} implies a length comparison result for all but a bounded set of representations $\rho \in \T_H$ with projection $X \in \T(S)$. For all (not necessarily simple) closed curves $\gamma$ on $S$,
\[
\frac{\ell_\rho(\gamma)}{\ell_X(\gamma)} \asymp \frac{1}{h(\rho)}
\]
 where $h(\rho)$ is the topological entropy of the Hilbert length of $\rho$, and the length functions are the ones defined above for the associated currents. In fact, it follows from \cite[Theorem 1.4, Corollary 1.5]{MZ} that $\sys(\rho) \prec 1/h(\rho) \prec \sys_Y(\rho)$ for a thick subsurface $Y$ of $\pi(\rho)$. The upper bound is not stated this way in their paper, but it can be deduced using Theorem \ref{thm:Full_Collar} that the function $K_\rho$ in \cite{MZ} is the $Y$-systole of $\rho$ on a certain subsurface $Y$. Thus, we get the inequality
 \[
  \sys(\rho) \prec \frac{\ell_\rho(\gamma)}{\ell_X(\gamma)} \prec  \sys_Y(\rho)
 \]
where the lower bound is the systole of $\rho$ on all of $S$, and the upper bound is the $Y$-systole for a certain $Y \subset S$.
 
It should be noted that the constants in this inequality depend on the projection $X$ of $\rho$. It would be interesting to see if one could state this theorem with constants that depend only on the topology of $S$, and simultaneously tighten the two bounds to get a statement similar to Theorem \ref{thm:Main}. Moreover, it would again be interesting to know if this projection of a higher Teichm\"uller space to $\T(S)$ is bounded distance to the length minimizing projection, with respect to some metric on $\T(S)$.

\section{A first length comparison}
\label{sec:First_Thick_Thin}
This section is independent of the rest of the paper. Given a filling current $\mu$ and a closed curve $\alpha$, we re-interpret $i(\mu, \alpha)$ as the $\mu$-length of $\alpha$. Suppose the length minimizer of $\mu$ is $\pi(\mu)$. Then even without using the mixed collar lemma, we can compare the $\mu$- and $\pi(\mu)$- lengths of simple closed curves that lie in \textit{thick components} of $\pi(\mu)$, which we define now.

We can choose a shortest marking $\Gamma = \{\gamma_1, \dots, \gamma_n\}$ for $\pi(\mu)$ as follows: First, we choose a shortest pants decomposition for $\pi(\mu)$ using the greedy algorithm. That is, we let $\gamma_1$ be the shortest simple closed curve on $\pi(\mu)$, then given $\gamma_1, \dots, \gamma_i$, we let $\gamma_{i+1}$ be the shortest curve in the complement of $\gamma_1, \dots, \gamma_i$. This gives a pants decomposition $\gamma_1, \dots, \gamma_m$ for $S$. Then we choose the shortest transverse curves $\gamma_{m+1}, \dots, \gamma_{2m}$ so that if $i = j$ then $i(\gamma_m+i, \gamma_j) = 1$ or 2, and if $i \neq j$, then $i(\gamma_m+i, \gamma_j) = 0$, and so that $\tau_{\gamma_i}(\gamma_{m+i})$ is minimal.

Let $c_b$ be the Bers constant for $S$. That is, all closed curves shorter than $c_b$ are simple and disjoint. The curves that are shorter than $c_b$ in $\pi(\mu)$ will be called the \textbf{Bers curves}. Cut $\pi(\mu)$ along the Bers curves, and let $Y$ be a connected component of the result. We call $Y$ a \textbf{thick component} of $X$. Note that by taking the curves in $\Gamma$ that are essential and non-peripheral in $Y$, we get the shortest marking on $Y$. We call this the shortest marking on $\pi(\mu)$ restricted to $Y$. Note that this marking may be empty if $Y$ is a pair of pants.

Following the notation of \cite{BIPP}, we can define the systolic length of $\mu$ in any subsurface $Y$ of $S$. We let
\[
 \sys_Y(\mu) = \inf_{\alpha \subset Y} i(\mu, \alpha)
\]
where the infimum is taken over all essential, non-peripheral simple closed curves in $Y$.

\begin{prop}
\label{prop:Thick-thin_without_collar_lemma}
Let $\mu \in \C_{fill}(S)$, and let $\pi(\mu)$ be its length minimizer. Let $Y$ be a thick component of $\pi(\mu)$, and let $\Gamma$ be the shortest marking on $Y$. Then, for any essential, non-peripheral simple closed curve $\alpha$ in $Y$, we have 
 \[
 \frac{\sys_Y(\mu)^2} {i(\mu, \Gamma)} \prec \frac{i(\mu, \alpha)}{i(\pi(\mu), \alpha)} \prec i(\mu, \Gamma)
 \]
 where the constants depend only on $\chi(S)$.
\end{prop}
 This proposition follows directly from the lemmas in this section. We will first prove this proposition assuming the lemmas, then give the proofs of the lemmas after. 
\begin{proof}
To show the upper bound, we will show in Lemma \ref{lem:Lower_bound_Gamma_intersection} that 
 \[
  i(\alpha, \mu) \leq 2i(\mu, \Gamma) i(\alpha, \Gamma) 
 \]
 Then, just as in \cite{RafiTT}, we use the upper bound in \cite[Equation 4.1]{Minsky93}, which says $i(\alpha, \Gamma) \prec \ell_{\pi(\mu)}(\alpha)$, where the constant depends on the $\pi(\mu)$-lengths of the curves in $\Gamma$. As $\Gamma$ is a marking for a thick component of $\pi(X)$, the lengths of the curves are all bounded in terms of $\chi(S)$. Thus, we get 
 \[
  i(\alpha, \mu) \prec i(\mu, \Gamma) i(\pi(\mu), \alpha)
 \]
where the constants depend only on $\chi(S)$.  We then rearrange this inequality into the upper bound we need.

For the lower bound, we show in Lemma \ref{lem:Upper_bound_two_curve_comparison} that if $\gamma_i$ is a curve in $\Gamma$, then 
\[
  \frac{1}{12} \sys_Y(\mu)^2 i(\alpha, \gamma_i) \leq i(\mu, \alpha) i(\mu, \gamma_i) 
\]
Summing this over all curves in $\Gamma$ gives 
\[
  \frac{1}{12} \sys_Y(\mu)^2 i(\alpha, \Gamma) \leq i(\mu, \alpha) i(\mu, \Gamma) 
\]
This time we use the lower bound in \cite[Equation 4.1]{Minsky93}, which says $\ell_{\pi(\mu)}(\alpha)\prec i(\alpha, \Gamma)$, where again the fact that $Y$ is a thick component of $\pi(\mu)$ mean the constants depend only on $\chi(S)$. This gives us
\[
 \sys_Y(\mu)^2 i(\alpha, \pi(\mu)) \prec i(\mu, \alpha) i(\mu, \Gamma) 
\]
Again, rearranging this inequality gives us the lower bound we need.
\end{proof}

\subsection{Upper bound}
We will use the following lemma for the upper bound in Proposition \ref{prop:Thick-thin_without_collar_lemma}:

\begin{lem}
\label{lem:Lower_bound_Gamma_intersection}
 If $\Gamma$ fills a subsurface $Y$ of $\pi(\mu)$, and $\alpha$ is an essential simple closed curve in $Y$, then
 \[
  i(\alpha, \mu) \leq 2  i(\mu, \Gamma) i(\alpha, \Gamma)
 \]
\end{lem}
\begin{proof}
Let $\alpha$ be a simple closed curve in $Y$. We will homotope $\alpha$ into a closed curve $\alpha'$ that lies entirely in the graph of $\Gamma$, so that $\alpha'$ passes through each point of $\Gamma$ at most $2 i(\alpha, \Gamma)$ times. This will imply that $i(\mu, \alpha') < 2 i(\alpha, \Gamma)  i(\mu, \Gamma)$. But by the same proof as in \cite[Lemma 4.4]{MZ}, $i(\mu, \alpha) < i(\mu, \alpha')$, so this will complete the proof.

The marking $\Gamma$ cuts $Y$ into simply connected regions, and boundary-parallel annuli. Now, $\Gamma$ cuts $\alpha$ into $i(\alpha, \Gamma)$ arcs, where each arc lies in one of the regions of $Y \setminus \Gamma$. Let $\bar \alpha$ be one such arc lying in a region $D$. 

Then $\bar \alpha$ is homotopic relative its endpoints to a concatenation of the arcs in $\partial D$ which we denote $\bar \alpha'$. Note that $\bar \alpha'$ passes through each point in $\partial D$ at most once. In the case where $D$ is simply connected, this is always the case. In the case where $D$ is an annulus, this follows from the fact that $\alpha$ is a simple closed curve. In either case, this implies that $\bar \alpha'$ passes through each point in $\Gamma$ at most twice. 

Homotoping each segment of $\alpha \setminus \Gamma$ in this way, we get a closed curve $\alpha'$ lying entirely in the graph of $\Gamma$. As each arc passes through each point of $\Gamma$ at most twice, and we have $i(\alpha, \Gamma)$ arcs, we get that $\alpha'$ passes through each point of $\Gamma$ at most $2i(\alpha, \Gamma)$ times. Therefore, $i(\alpha, \mu) \leq 2  i(\mu, \Gamma) i(\alpha, \Gamma)$, as desired.
\end{proof}

\subsection{Lower bound}
The following lemma is equivalent to \cite[Lemma 5]{RafiTT}. The author would like to thank Sebastian Hensel for discussions that led to the proof. 
\begin{lem}
\label{lem:Upper_bound_two_curve_comparison}
 Let $\alpha, \beta$ be two essential, non-peripheral closed curves in $Y$, so that $\alpha$ is simple. Then, 
 \[
  \sys_Y(\mu)^2 i(\alpha, \beta) \leq 12 i(\mu, \alpha) i(\mu, \beta) 
 \]
\end{lem}
Note that $\beta$ is allowed to be non-simple.
\begin{proof}

 Let $m = \lceil 2 \frac{i(\alpha, \mu)}{\sys_Y(\mu)} \rceil$. We will show that there is a closed segment $\bar \alpha$ of $\alpha$ that contains at least $\frac 1m i(\alpha, \beta)$ of the intersections between $\alpha$ and $\beta$, and so that the $\mu$-length of its interior satisfies $i(\bar \alpha^\circ, \mu) \leq \frac 12 \sys_Y(\mu)$. 
 
 \begin{figure}[h!]
  \centering 
  \includegraphics{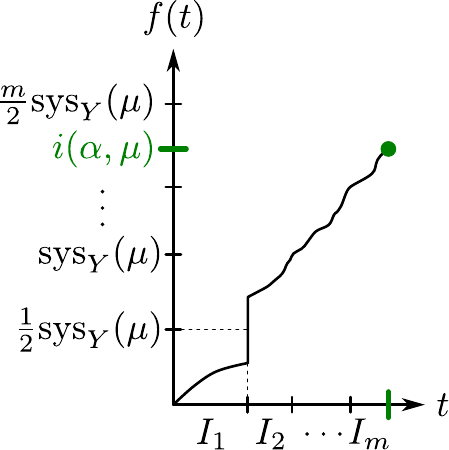}
  \caption{The graph of $f(t) = i(\mu, \alpha[0,t])$}
  \label{fig:Cutting_up_alpha}
 \end{figure}

 In fact, fix a parametrization $\alpha: [0,1] \to S$ for $\alpha$. Let $f(t) = i(\mu, \alpha[0,t])$ (Figure \ref{fig:Cutting_up_alpha}). Then $f$ is an increasing function with image $[0, i(\mu, \alpha)]$. Moreover, $f$ is continuous away from at most a countable set of points. For each $i = 1, \dots, m$, let $I_i$ be the closure of $f^{-1}[\frac{i-1}{2}\cdot \sys_Y(\mu), \frac i2\cdot \sys_Y(\mu)]$. Then 
 \[
  i(\mu,\alpha(I_i^\circ)) \leq \frac 12 \sys_Y(\mu))
 \]
That is, the interior of each arc $\alpha(I_i)$ has $\mu$-length at most $\frac 12 \sys_Y(\mu)$. Note that we have to take the interior of $I_i$ because $f$ might not be continuous at the endpoints of $I_i$.

We have cut $\alpha$ into $m$ arcs. Therefore, there is some $i$ so that the arc $\bar \alpha = \alpha(I_i)$ contains $n \geq \frac 1m i(\alpha, \beta)$ intersections with $\beta$. 

Thus at this stage, we have defined $m, n$ and $\bar \alpha$ so that
\begin{align*}
 m & = \left \lceil 2 \frac{i(\alpha, \mu)}{\sys_Y(\mu)} \right \rceil \\
 n & = i(\bar \alpha, \beta) \geq \frac 1m i(\alpha, \beta) \\
 i(\bar \alpha, \mu) & \leq \frac 12 \sys_Y(\mu)
\end{align*}

Fixing an orientation for $\bar \alpha$ and $\beta$ we can give a sign to each intersection between them. We will say that an intersection is positive if $\beta$ crosses $\bar \alpha$ from right to left, and negative otherwise (Figure \ref{fig:alphabeta}). Without loss of generality, there are $k \geq n/2$ positive intersections. The positive intersections divide $\beta$ into arcs $\beta_1, \dots, \beta_k$. Then since $\sum i(\beta_i, \mu) = i(\beta, \mu)$, there must be some $\beta_i$ so that $i(\beta_i, \mu) \leq i(\beta, \mu)/k$. (Here we use that no intersection between $\alpha$ and $\beta$ is an atom of $\mu$.) So if $\bar \beta = \beta_i$, we have 
\[
i(\bar \beta, \mu) \leq 2i(\beta, \mu)/n
\]
Note that there are no positive intersections of $\bar \alpha$ with $\beta$ along $\bar \beta$, but there may be many positive intersections of $\bar \alpha$ and $\beta$ along $\bar \alpha$.
\begin{figure}[h!]
 \centering 
 \includegraphics{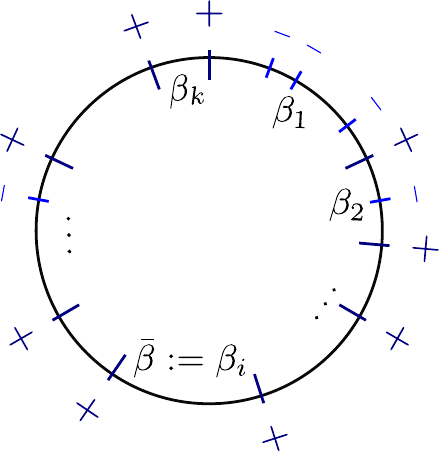}
 \caption{The arc $\bar \alpha$ (in blue) cuts $\beta$ (in black) into a collection of arcs.}
 \label{fig:alphabeta}
\end{figure}

We will now find an essential, non-peripheral closed curve $\sigma$ so that $i(\sigma, \mu)\leq i(\bar \alpha, \mu) + i(\bar \beta, \mu)$. (The reason for considering the signs of intersections between $\alpha$ and $\beta$ is precisely to ensure that $\sigma$ is non-peripheral.)

Once we find $\sigma$, the proof will proceed as follows. They show in \cite{BIPP} that $i(\mu, \nu) \geq \sys_Y(\mu)$ for any geodesic current $\nu$ whose support lies entirely in $Y$. Thus we will have
\begin{align*}
 \sys_Y(\mu) &\leq i(\sigma, \mu) \\
  & \leq i(\bar \alpha, \mu) + i(\bar \beta, \mu) \\
  & \leq \frac 12 \sys_Y(\mu)+ 2i(\beta, \mu)/n  \\
  & \leq \frac 12 \sys_Y(\mu)  + 2\frac{m}{i(\alpha, \beta)} \cdot i(\beta, \mu)
\end{align*}
Rearranging this will give 
\[
\sys_Y(\mu) i(\alpha, \beta)\leq 4m \cdot i(\beta, \mu)
\]
As $m = \left \lceil 2 \frac{i(\alpha, \mu)}{\sys_Y(\mu)} \right \rceil$ and $i(\alpha, \mu) \geq \sys_Y(\mu)$, we have $m \leq 3\frac{i(\alpha, \mu)}{\sys_Y(\mu)}$. And so we get 
\[
 \sys_Y(\mu)^2 i(\alpha, \beta) \leq 12i(\beta, \mu)i(\alpha, \mu)
\]

The construction of $\sigma$ involves reducing to the two situations in Figure \ref{fig:alpha1}.
\begin{figure}[h!]
 \centering 
 \includegraphics{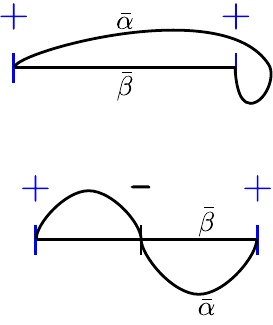}
 \caption{We show that without loss of generality, $\bar \alpha$ and $\bar\beta$ fall into one of these two cases.}
 \label{fig:alpha1}
\end{figure}
We first deal with the case where $\bar \alpha$ does not intersect $\bar \beta$ in its interior. Shrink $\bar \alpha$ to the arc whose endpoints both lie on $\bar \beta$. We still denote the new arc by $\bar \alpha$, as its intersection number with $\mu$ can only decrease when we do this. Then $\sigma = \bar \beta \circ \bar \alpha$ is a closed curve (Figure \ref{fig:OneInt}). This curve is also essential and non-peripheral because $i(\sigma,\alpha) \geq 1$. We show this as follows: Homotope $\sigma$ off of itself slightly so that all of its intersections with $\alpha$ are transverse. By construction, $\bar \alpha$ has no transverse intersections with $\alpha$. Thus, $\sigma$ and $\alpha$ can only intersect along $\bar \beta$. As $\alpha$ and $\beta$ are in minimal position, all of these intersections must be essential. Moreover, from the point of view of $\alpha$, the intersections at the endpoints of $\bar \beta$ have the same sign. So we have homotoped $\sigma$ to have at least one intersection with $\alpha$ coming from its endpoint. Therefore, $i(\sigma, \alpha) \geq 1$, and so $\sigma$ is essential and non-peripheral. Thus we have
\[
 i(\sigma, \mu) \leq i(\bar \beta, \mu) + i(\bar \alpha,\mu)
\]
as desired. 

\begin{figure}[h!]
 \centering 
 \includegraphics{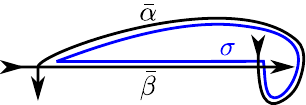}
 \caption{If $\bar \beta$ and $\bar \alpha$ have disjoint interiors, then $i(\sigma, \alpha) \geq 1$.}
 \label{fig:OneInt}
\end{figure}

\begin{figure}[h!]
 \centering 
 \includegraphics{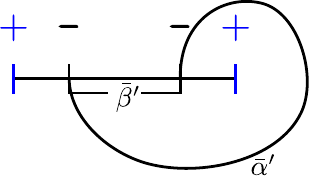}
 \caption{If $\bar \alpha$ intersects $\bar \beta$ more than once, we can find subarcs $\bar \beta'$ and $\bar \alpha'$ that are disjoint.}
 \label{fig:ExcessIntersections}
\end{figure}

In general, $\bar \alpha$ only intersects $\bar \beta$ in negative intersection points. If $\bar \alpha$ intersects $\bar \beta$ in at least two negative intersection points (Figure \ref{fig:ExcessIntersections}), then we can find subarcs $\bar \beta'\subset \bar \beta$ and $\bar \alpha' \subset \bar \alpha$ so that the endpoints of $\bar \alpha'$ are negative intersection points with $\beta$, and $\bar \beta'$ and $\bar \alpha'$ have disjoint interior. In this case, we let $\sigma = \bar \beta' \circ \bar \alpha'$. By the same logic as above, $\sigma$ is essential and non-peripheral. Since $i(\bar \beta', \mu) \leq i(\bar \beta, \mu)$ and $i(\bar \alpha', \mu) \leq i(\bar \alpha, \mu)$, we get $i(\sigma, \mu) \leq i(\bar \alpha, \mu) + i(\bar \beta, \mu)$.

\begin{figure}[h!]
 \centering 
 \includegraphics{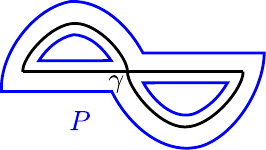}
 \caption{If $i(\bar \beta, \bar \alpha) = 1$, then a regular neighborhood of $\alpha$ is an embedded pair of pants $P$.}
 \label{fig:PantsBoundary}
\end{figure}

Suppose lastly that $\bar \alpha$ intersects $\bar \beta$ exactly once in its interior (Figure \ref{fig:alpha1}). Then $\sigma = \bar \beta \circ \bar \alpha$ is a curve with at least one self-intersection. (If $\sigma$ were simple, then $\beta$ and $\alpha$ would form a null-homotopic bigon, which is impossible.) As $\sigma$ is non-simple, it is essential and non-peripheral, and we again have that $i(\sigma, \mu) \leq i(\bar \alpha, \mu) + i(\bar \beta, \mu)$.

As we can find an essential, non-peripheral closed curve $\sigma$ with $i(\sigma, \mu) \leq i(\bar \alpha, \mu) + i(\bar \beta, \mu)$, we complete the proof.
\end{proof}

\section{Sketch of proof for Mixed Collar Theorem \ref{thm:Full_Collar}}
\label{sec:Mixed_collar_idea} 
Before embarking on a proof of Theorem \ref{thm:Full_Collar}, we give a short sketch of its proof. 
For simplicity, suppose $\mu$ is a filling closed geodesic, and $\alpha$ and $\beta$ are simple closed curves with $i(\alpha, \beta) = 1$. We wish to show that if $\alpha$ is short and $\beta$ is long in the length minimizer $\pi(\mu)$, then $\mu$ must intersect $\beta$ much more than $\alpha$. 

To see this, suppose we can shrink the length of $\alpha$ by contracting $S$ along an $\epsilon$-strip about $\beta$. Thus, for each intersection of $\mu$ with $\beta$, the length of $\mu$ will decrease by roughly $\epsilon$. However, we choose $\epsilon$ so that this construction will increase the width of the maximal embedded collar about $\alpha$ by length at most 1. Thus, for each intersection between $\mu$ and $\alpha$, the length of $\mu$ will increase by 1. In other words, the length of $\mu$ changes by roughly $i(\mu, \alpha) - \epsilon i(\mu, \beta)$.

If we do this construction to the length minimizer $\pi(\mu)$ of $\mu$, then the total length of $\mu$ can only increase. This gives us our result on the relationship between $i(\mu, \alpha), i(\mu, \beta)$ and $\epsilon$. Roughly, $i(\mu, \alpha) > \epsilon \cdot i(\mu, \beta)$. Our possible choices of $\epsilon$ depend on the length of $\beta$ in $\pi(\mu)$. If $\beta$ is longer, then the maximal width of an embedded strip about $\beta$ is smaller. So $\epsilon$ depends on $\ell_{\pi(\mu)}(\beta)$, and we get our mixed collar lemma.

To make this result precise, we first explain how to remove an $\epsilon$ strip about a geodesic arc in a surface with boundary in Section \ref{sec:Surgery_with_boundary}. The construction itself is a slight refinement of work of Papadopoulos and Th\'eret (\cite{PT09}), which is itself based on a construction of Thurston \cite{Thurston88}. However, we have to carefully control the geometry of the resulting surface. This way, in Section \ref{sec:Closed_surfaces}, we show that we can do this surgery on a closed hyperbolic surface $X$, by cutting $X$ along $\alpha$, doing the construction for surfaces with boundary, and carefully controlling the lengths of the resulting boundary components so that we can glue them back together. Moreover, we estimate precisely how the length of a current $\mu$ changes after removing an $\epsilon$-strip from a hyperbolic surface $X$ (Lemma \ref{lem:Length_change_closed}). Using this result, we use the fact that $\pi(\mu)$ is a length minimizer to prove the mixed collar theorem in Section \ref{sec:Mixed_collar_theorem}.

\section{Surgery on surfaces with boundary}
\label{sec:Surgery_with_boundary}

Let $X$ be a hyperbolic surface with geodesic boundary. Let $\beta$ be a simple geodesic arc whose endpoints are orthogonal to $\partial X$. For all $\epsilon > 0$ small enough, we will describe how to get a new hyperbolic surface with boundary by \textbf{removing an $\epsilon$-strip about $\beta$.} This is the exact construction used by Papadopoulos and Th\'eret in \cite{PT09}. However, we will need to carefully control the geometry of the resulting surface, so we will re-describe the construction here.

Let $\hat X$ be the Nielsen completion of $X$. That is, we attach an infinite hyperbolic cuff to each boundary component of $X$. Alternatively, if we embed the universal cover $\tilde X$ of $X$ into $\h$, the the action $\pi_1(X)$ on $\tilde X$ extends to an action on $\h$. Then $\hat X$ is the quotient of $\h$ by this action. 

Take a lift $\tilde \beta$ of $\beta$ to $\h$, and let $\tilde x_0$ be its midpoint. Let $\tilde \beta_1$ and $\tilde \beta_2$ be two complete geodesics in $\h$ that are at distance $2\epsilon$ from each other and equidistant from $\tilde \beta$, such that their mutual perpendicular passes through $\tilde x_0$. We denote this mutual perpendicular by $\tilde \eta_0$. This condition was not imposed in the original construction in \cite{PT09}, but is necessary to be able to extend this construction to closed surfaces in Section \ref{sec:Closed_surfaces}.

Let $\tilde S_\epsilon$ be the bi-infinite strip bounded by $\tilde \beta_1$ and $\tilde \beta_2$ (Figure \ref{fig:S_epsilon}). As pointed out in \cite{PT09}, for all $\epsilon$ small enough, $\tilde S_\epsilon$ embeds in $\hat X$. We give an upper bound $\epsilon_{\max}$ on $\epsilon$ that depends only on $\ell_X(\beta)$ in equation \ref{eq:epsilon_max} below (proven in Lemma  \ref{lem:Geometry_embedded}).
\begin{figure}[h!]
 \centering 
 \includegraphics{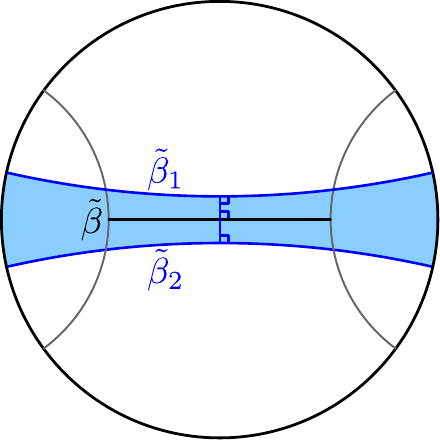}
 \caption{The blue strip is $S_\epsilon$.}
 \label{fig:S_epsilon}
\end{figure}

Let $S_\epsilon$ be the projection of $\tilde S_\epsilon$ to $\hat X$. Then $\beta_1$ and $\beta_2$ are the projections of its boundary curves $\tilde \beta_1$ and $\tilde \beta_2$, respectively, $\eta_0$ is the projection of the mutual orthogonal $\tilde \eta_0$, and $x_0$ is the midpoint of $\beta$. Extend $\beta$ to a bi-infinite geodesic parameterized by arclength, $\beta: \R \to \hat X$, so that $\beta(0) = x_0$. Then for each $t$, let $\eta_t$ be the arc intersecting $\beta$ at $\beta(t)$, joining $\beta_1$ to $\beta_2$ and staying a constant distance $|t|$ away from $\eta_0$. We collapse each $\eta_t$ to a point to get a new hyperbolic surface $\hat Y$. We call its convex core, $Y$, the result of removing an $\epsilon$ strip about $\beta$ from $X$ (Figure \ref{fig:Surface_with_boundary_construction}).

\begin{figure}[h!]
 \includegraphics{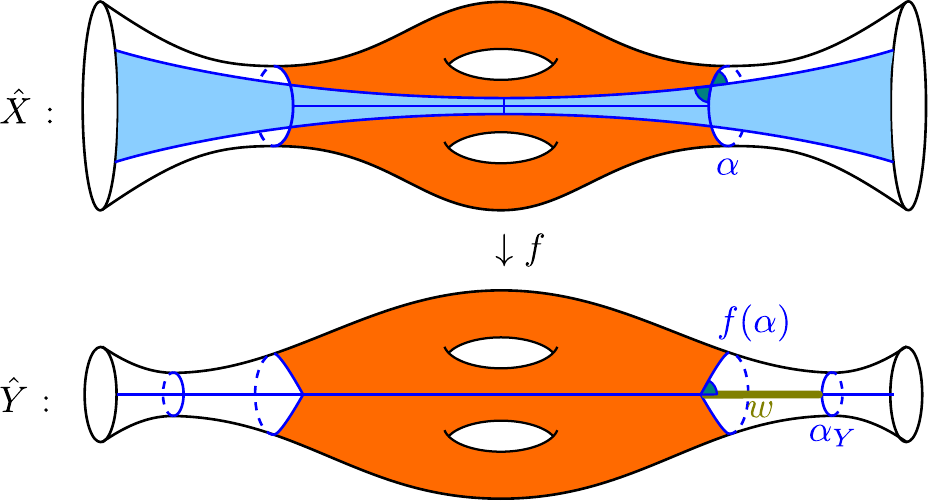}
 \caption{We glue together the two boundary curves of the strip $S_\epsilon$ (in blue). The image of $X$ (orange) is inside the convex core of $\hat Y$.}
 \label{fig:Surface_with_boundary_construction}
\end{figure}

\subsection{Geometry of the new surface}

\label{subsec:Geometry}

Let $f: \hat X \to \hat Y$ be the quotient map coming from this construction (Figure \ref{fig:Surface_with_boundary_construction}). Note that $f$ sends $X$ into $Y$ (see the proof of Lemma \ref{lem:Geometry_embedded} for details.) Moreover, suppose one of the endpoints of the original arc $\beta$ is on a boundary curve $\alpha$ of $X$. Then $f(\alpha)$ is a closed curve that is a concatenation of either one or two geodesic segments. (The number of geodesic segments depends on whether $\beta$ has one or both endpoints on $\alpha$.) Let $\alpha_Y$ be the geodesic representative of $\alpha$ in $Y$. Then $f(\alpha)$ and $\alpha_Y$ bound a cylinder $C$. We say the \textbf{width} $w$ of $C$ is the Hausdorff distance between its boundaries. We now show that we can choose $\epsilon$ depending only $\ell_X(\beta)$ so that we can remove an $\epsilon$ strip from $X$ (Lemma \ref{lem:Geometry_embedded}), control the width $w$ of $C$ (Lemma \ref{lem:Geometry_collar_width}), and achieve $\ell_Y(\alpha)$ in a range of values that depend only on $\ell_X(\beta)$ (Lemma \ref{lem:Geometry_alpha_length}).

First we show a collar lemma for arcs that is analogous to the one for simple closed curves. It is an easy corollary of the usual collar lemma, but does not seem to be explicitly written anywhere in the literature.
\begin{lem}
\label{lem:Collar_for_arcs}
Let $X$ be a hyperbolic surface with geodesic boundary. Let $\beta$ be a simple geodesic arc whose endpoints are orthogonal to the boundary of $X$. For all $\epsilon$ with 
 \[
  \epsilon \leq \sinh^{-1}\left (\frac{1}{\sinh(\ell_X(\beta))} \right)
 \]
the regular $\epsilon$-neighborhood of $\beta$ is embedded in $X$.
\end{lem}
\begin{proof}
 Double $X$ along its boundary to get a closed hyperbolic surface $X'$. We can likewise double $\beta$ along the boundary of $X$ to get a new curve $\beta'$ on $X'$. Since the $\beta$ is orthogonal to $\partial X$ at its endpoints, $\beta'$ is a closed geodesic on $X'$, whose length is $2 \ell_X(\beta)$. Moreover, an $\epsilon$-neighborhood of $\beta'$ on $X'$ restricts to a $\epsilon$-neighborhood of $\beta$ on $X$. By the collar lemma, $N_\epsilon(\beta')$ is embedded for every $\epsilon$ with 
 \[
  \epsilon <  \sinh^{-1}\left (\frac{1}{\sinh(\frac 12 \ell_{X'}(\beta'))}\right )
 \]
As $\frac 12 \ell_{X'}(\beta')) = \ell_X(\beta)$, we are done.
\end{proof}
From now on, for any geodesic arc $\beta$, we let 
\[
 \Col_X(\beta) =  \sinh^{-1}\left (\frac{1}{\sinh(\ell_X(\beta))} \right)
\]
be the \textbf{collar width} of $\beta$.

Now, we show that the range of $\epsilon$ for which we can remove an $\epsilon$ strip from $X$ depends only on $\ell_X(\beta)$. In fact, set $\epsilon_{\max} = \epsilon_{\max}(\ell_X(\beta))$ so that 
\begin{equation}
\label{eq:epsilon_max}
\sinh \epsilon_{\max} = \tanh 1 \frac{\tanh(\frac 12 \Col_X(\beta))}{\cosh( \frac12\ell_X(\beta))}
\end{equation}
\begin{rem}
\label{rem:decay_of_epsilon_max}
In fact, $\epsilon_{\max}$ is a decreasing function with $0 < \epsilon_{\max} < \tanh 1$. Moreover, as $x$ tends to infinity, $\epsilon_{\max}(x)$ decays like $e^{-\frac 32 x}$. That is, the limit of $e^{\frac 32 x}\epsilon_{\max}(x)$ tends to a non-zero constant as $x$ tends to infinity.
\end{rem}

Then we have:
\begin{lem}
 \label{lem:Geometry_embedded}
 Let $\epsilon_{\max} = \epsilon_{\max}(\ell_X(\beta))$ be as in equation \ref{eq:epsilon_max}. If $\epsilon \leq \epsilon_{\max}$, then the $\epsilon$-strip $S_\epsilon$ about $\beta$ will be embedded in $\hat X$.
\end{lem}

\begin{proof}
First, we show that if $\epsilon < \epsilon_{\max}$, then $\tilde S_\epsilon$ embeds in $X$, and then we will show that it embeds in all of $\hat X$. Refer to Figure \ref{fig:Lombardi_quadrilaterals} for what follows. 

\begin{figure}[h!]
 \includegraphics{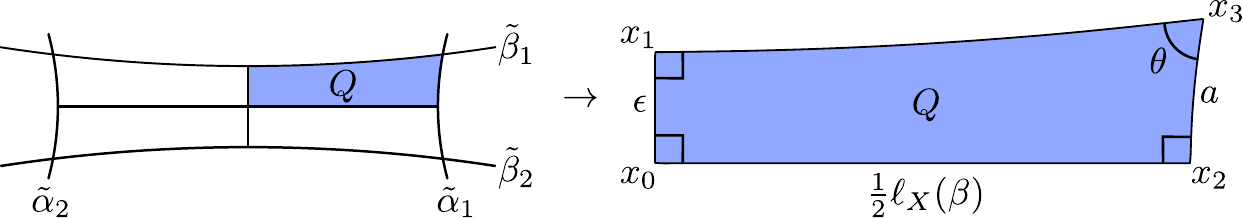}
 \caption{Between $\tilde \alpha_1$ and $\tilde \alpha_2$, $\tilde S_\epsilon$ looks like 4 congruent Lombardi quadrilaterals}
 \label{fig:Lombardi_quadrilaterals}
\end{figure}

We fix some notation: Recall that $\tilde \eta_0$ is the arc orthogonal to $\tilde \beta_1, \tilde \beta$ and $\tilde \beta_2$. Let $\tilde X$ be the connected component of the lift of $X$ to $\h$ that contains $\tilde \beta$. Then $S_\epsilon \cap \tilde X$ is a quadrilateral with 4 congruent angles. In fact, $\tilde \eta_0$ and $\tilde \beta$ cut this quadrilateral into 4 congruent quadrilaterals. Since $\tilde \eta_0$ and $\tilde \beta$ bisect each other, all four quadrilaterals are congruent. Let $Q$ be a quadrilateral bounded by $\tilde \beta$ and a lift $\tilde \alpha$ of $\alpha$. 

Recall that $\tilde \eta_0$ passes through the midpoint $x_0$ of $\tilde \beta$. Let $x_1$ be the endpoint of $\tilde \eta_0$ on $\tilde \beta_1$, $x_2$ be the endpoint of $\tilde \beta$ on $\tilde \alpha$, and $x_3$ be the intersection point of $\tilde \beta_1$ and $\tilde \alpha$. We have: 
\[
 \ell (\overline{x_0 x_1}) = \epsilon, \text{ and } \ell (\overline{x_0 x_2}) = \frac 12 \ell_X(\beta)
\]

Note that $Q$ has right angles at $x_0, x_1$ and $x_2$, so it is a Lambert quadrilateral. Let $a = \ell(\overline{x_1 x_2})$. By, for example, \cite[Theorem 32.21]{Martin}, we have 
\begin{equation}
\label{eq:define_a}
  \tanh a = \tanh \epsilon \cosh \frac{\ell_X(\beta)}{2} 
\end{equation}
First we show that $\epsilon \leq \epsilon_{\max}$ implies 
\[
a < \frac 12 \Col_X(\beta)
\]
Indeed, we have $\tanh \epsilon < \sinh \epsilon$ for all $\epsilon$. Then using $\tanh 1 < 1$, we see that $\epsilon < \epsilon_{\max}$ implies $\tanh \epsilon  \cosh \frac{\ell_X(\beta)}{2}  < \tanh(\frac 12 \Col_X(\beta))$. This gives us the desired inequality. Thus, $\tilde S_\epsilon \cap \tilde X$ lies in a $\frac 12 \Col_X(\beta)$-neighborhood of $\tilde \beta$. In particular, $\tilde S_\epsilon$ is embedded in $X$.

\begin{figure}[h!]
 \includegraphics{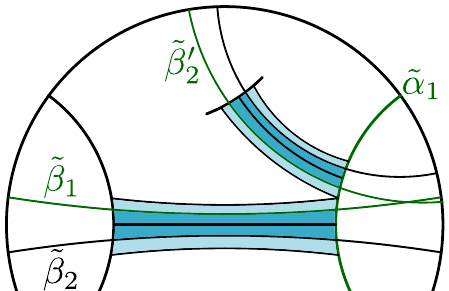}
 \caption{Two lifts of $\tilde S_\epsilon$. The light blue regions are $\Col_X(\beta)$-neighborhoods of $\tilde \beta$.}
 \label{fig:Nested_strip_and_max_angle}
\end{figure}

Next we show that $\tilde S_\epsilon$ is embedded in $\hat X$. Recall that $\beta_1, \beta_2$ are the projections of $\tilde \beta_1, \tilde \beta_2$ to $\hat X$. Since $\tilde S_\epsilon$ is embedded in $X$, if it is not embedded in $\hat X$, then $\beta_1$ and $\beta_2$ must intersect in $\hat X \setminus X$. That is, they intersect in a cuff. Without loss of generality, they intersect in the cuff bounded by $\alpha$. Thus, there is a lift $\tilde \beta_2'$ of $\beta_2$ that intersects $\tilde \alpha$ and $\tilde \beta_1$.

Let $x_2'$ be the intersection points between $\tilde \beta_2'$ and $\tilde \alpha$, and let $p$ be the intersection point of $\tilde \beta_1$ and $\tilde \beta_2'$. Then we have a triangle $\Delta$ with vertices $x_2, x_2'$ and $p$. Since $a < \frac 12 \Col_X(\beta)$, and the $\Col_X(\beta)$-neighborhood of $\beta$ is embedded in $X$, the distance between $x_2$ and $x_2'$ is at least $\Col_X(\alpha)$. Now let $\theta$ be its acute angle of $Q$. Then the angles of $\Delta$ at $x_2$ and $x_2'$ are both $\theta$ as well.  

Recall that for a distance $d$, then \textit{angle of parallelism} $\theta_{\max}$ of $d$ is the maximum angle for which a right hyperbolic triangle can have base $d$ and acute angle $\theta_{\max}$. Moreover, $\theta_{\max}$ decreases as $d$ increases. In particular, We can cut $\Delta$ into two congruent right triangles with base $\frac 12 \Col_X(\beta)$. Thus, $\theta$ must be smaller than the angle of parallelism for $\frac 12 \Col_X(\beta)$. So, for example, by \cite[Chapter 32]{Martin}, 
\[
 \cos(\theta) \geq \tanh(\frac 12 \Col_X(\beta))
\]
as $\cos(\theta)$ increases as $\theta$ decreases for $0 < \theta < \pi/2$. However, $\theta$ is the acute angle of $Q$, so we have 
\[
  \cos(\theta) = \sinh \epsilon \sinh (\ell_X(\beta)/2)
\]
Given our choice of $\epsilon_{\max}$, this is a contradiction. Indeed, $\sinh(\ell_X(\beta/2)) < \cosh(\ell_X(\beta/2))$. Thus, $\sinh \epsilon \sinh (\ell_X(\beta)/2) < \sinh \epsilon \cosh(\ell_X(\beta)/2)$. And by our choice of $\epsilon_{\max}$, this is smaller than $\tanh (\frac 12 \Col_X(\beta))$.

In particular, $\tilde S_\epsilon$ is embedded in $X$.
\end{proof}

\begin{lem}
 \label{lem:Geometry_collar_width}
 For $\epsilon_{\max}$ as in Equation (\ref{eq:epsilon_max}), if $\epsilon \leq \epsilon_{\max}$, then the width $w$ of the cylinder bounded by $f(\alpha)$ and $\alpha_Y$ will be at most 1.
 \end{lem}
 \begin{proof}
  Let $S_\epsilon$ be the (embedded) $\epsilon$-strip about $\beta$ in $\hat X$.  Suppose for what follows that $\beta$ has only one endpoint on a boundary curve $\alpha$ of $X$. The case where $\beta$ has both endpoints on $\alpha$ gives the same estimates. Following the notation from Lemma \ref{lem:Geometry_embedded}, let $x_3$ be an intersection point between boundary geodesic $\beta_1$ of $S_\epsilon$ and $\alpha$. Let $p = f(x_3)$. Since $\beta$ intersects $\alpha$ once, $f(\alpha)$ consists of a geodesic segment whose endpoints are joined at $p$ in $Y$. Drop a perpendicular from $p$ down to a point $q$ on $\alpha$. Then $w = \ell_Y(\overline{pq})$ is exactly the Hausdorff distance between $f(\alpha)$ and $\alpha_Y$. (If $\beta$ intersects $\alpha$ twice, then we would get two singular points on $f(\alpha)$, and two perpendiculars. But by symmetry of the construction, the lengths of these two perpendiculars will be the same.)
  
  Let $C$ be the cylinder bounded by $f(\alpha)$ and $\alpha$. Then the segment $\overline{pq}$ cuts $C$ into a Sacchieri quadrilateral: it has two right angles corresponding to $q$, and two congruent opposite sides corresponding to $\overline{pq}$.  By taking the mutual perpendicular between points $r \in f(\alpha)$ and $s \in \alpha_Y$, we cut the Sacchieri quadrilateral into two congruent Lambert quadrilaterals. Let $R$ be one of these quadrilateral. Abusing notation, we denote its vertices by $p,q,r,s$.
  
  First, note that the angle $\theta$ at $p$ is the same as the angle $\theta$ at $x_3$ inside $S_\epsilon \cap X$ (Figure \ref{fig:Surface_with_boundary_construction}). In fact, because the mutual perpendicular between boundary curves $\beta_1$ and $\beta_2$ of $S_\epsilon$ passes orthogonally through $\beta$, the point $x_3$ gets glued to a point $x_3'$ at the intersection of $\beta_2$ and $\alpha$. The interior angle at $x_3$ of $S_\epsilon \cap X$ is also the same as the interior angle at $x_3'$. So the interior angle of $C$ at $p = f(x_3)$ is exactly $2 \theta$. Dropping the perpendicular from $p$ to $\alpha_Y$ bisects this angle, again, by symmetry. From the proof of Lemma \ref{lem:Geometry_embedded},
  \[
   \cos(\theta) = \sinh \epsilon \sinh (\ell_X(\beta)/2)
  \]

  Next, we get a lower bound on the length of $\overline{pr}$. By the proof of Lemma \ref{lem:Geometry_embedded}, if $\epsilon < \epsilon_{\max}$, then $S_\epsilon \cap X$ lies in a $\frac 12 \Col_X(\beta)$-neighborhood of $\beta$. Now, $S_\epsilon$ cuts $\alpha$ into either one or two components, depending on whether $\beta$ intersects $\alpha$ othogonally once or twice. Since a $\Col_X(\beta)$-neighborhood of $\beta$ is embedded, this implies that these components both have length at least $\Col_X(\beta)$. In particular, each geodesic segment of $f(\alpha)$ has length at least $\Col_X(\beta)$. Since we bisected this segment to form the quadrilateral $R$, we have 
  \[
   \ell_Y(\overline{pr}) \geq \frac 12 \Col_X(\beta)
  \]

  By identities for Lambert quadrilaterals (see, eg, \cite[Theorem 32.21]{Martin}), for $w = \ell_Y(\overline{pq})$,
  \begin{align*}
   \cos \theta & = \tanh \ell_Y(\overline{pr}) \tanh w \\
    & \geq \tanh \left( \frac 12 \Col_X(\beta)\right) \tanh w
  \end{align*}

  Putting this together, we get 
  \[
   \tanh w \leq \frac{\sinh \epsilon \sinh (\ell_X(\beta)/2)}{\tanh \left( \frac 12 \Col_X(\beta)\right)}
  \]
 But we chose $\epsilon < \epsilon_{\max}$, so $\sinh \epsilon < \tanh 1 \frac{\tanh(\frac 12 \Col_X(\beta))}{\cosh \ell(\beta)/2}$. As $\sinh(x) < \cosh(x)$ for all $x$, this implies $\tanh w < \tanh 1$. That is, the Hausdorff distance between $f(\alpha)$ and $\alpha_Y$ is at most 1.
 \end{proof}

 \begin{lem}
 \label{lem:Geometry_alpha_length}
  Given $\ell_X(\alpha)$, there is a range of possible values for $\ell_Y(\alpha)$ that depends only on $\ell_X(\beta)$. More precisely, there is a constant $\ell_{\min} = \ell_{\min}(\ell(\beta))$ so that there is a one-to-one correspondence between the set of $\ell$ so that 
  \[
   \ell_{\min} \leq \ell \leq \ell_X(\alpha)
  \]
and the set of $\epsilon$ so that 
\[
 0 \leq \epsilon \leq \epsilon_{\max}(\ell_X(\beta))
\]
for which
  \[
   \ell_Y(\alpha) = \ell.
  \]
  where $Y$ is the metric obtained by removing an $\epsilon$-strip about $\beta$ from $X$. Moreover, the function $\ell_{\min}(x)$ is increasing in $x$.
 \end{lem}
 \begin{proof}
  In what follows, we will carefully compute $\ell_Y(\alpha)$ given $\ell_X(\alpha)$, $\ell_X(\beta)$ and $\epsilon$. But first, suppose $\beta$ has exactly one endpoint on $\alpha$. The case where $\beta$ has both endpoints on $\alpha$ is very similar, and we will point out the differences as they come up. Then $S_\epsilon$ intersects $\alpha$ in an arc of length $2a$, where, by Equation (\ref{eq:define_a}) in the proof of Lemma \ref{lem:Geometry_embedded},
  \[
   \tanh a = \tanh \epsilon \cosh \frac{\ell_X(\beta)}{2} 
  \]
  Let $a_{\max}$ be so that $\tanh a_{\max} = \tanh \epsilon_{\max} \cosh \frac{\ell_X(\beta)}{2}$. Then for any $a$ with $0 < a < a_{\max}$, we can find an $\epsilon$ with $0 < \epsilon < \epsilon_{\max}$ so that $S_\epsilon$ intersects $\alpha$ in an arc of length $2a$. Moreover, as $\epsilon$ or $\ell_X(\beta)$ increases, so does $a$.
  
  Next, we see that $\ell_Y(f(\alpha)) = \ell_X(\alpha) - 2a$. We form the Lambert quadrilateral $R$ in the same way as in the proof of Lemma \ref{lem:Geometry_collar_width}. In particular, we drop a perpendicular from the singular point of $f(\alpha)$ to $\alpha$, and then take the mutual perpendicular between $f(\alpha)$ and $\alpha$. In the case where $\beta$ intersects $\alpha$ once, this gives us two congruent Lambert quadrilaterals. In the case where $\beta$ intersects twice, this gives us four such quadrilaterals. We let $R$ be one such quadrilateral, with vertices $p,q,r,s$, where $p$ is the singular point of $f(\alpha)$, $r$ is the midpoint of the geodesic segment of $f(\alpha)$, and $\overline{pq}$ and $\overline{rs}$ are the sides of $R$ joining $f(\alpha)$ to $\alpha_Y$. Then, 
  \[
   \ell_Y(\alpha) = 2 \ell_Y(\overline{qs})
  \]
  when $\beta$ intersects $\alpha$ once. If $\beta$ intersects $\alpha$ twice, then the constant is 4 instead of 2. By, eg, \cite[Theorem 32.21]{Martin}, and using that $\ell_Y(\overline{qs}) = \frac 12 \ell_Y(\alpha)$ and $\ell_Y(\overline{pr}) = \ell_X(\alpha) - a$, we get
  \[
   \sin \theta = \frac{\cosh \left( \frac 12 \ell_Y(\alpha)\right )}{\cosh \left(\ell_X(\alpha) - a\right )}
  \]
 where $\theta$ is the acute angle of $R$. As we showed in the proof of Lemma \ref{lem:Geometry_collar_width}, $\theta$ is the same as any of the interior angles of $S_\epsilon \cap X$. In particular, 
 \[
  \cos (\theta) = \sinh \epsilon \sinh (\ell_X(\beta)/2)
 \]
Using that $\sin(\cos^{-1}(x)) = \sqrt{1 - x^2}$, we get that 
\begin{equation}
\label{eq:change_of_length_epsilon}
 \left(1 - \sinh^2 \epsilon \sinh^2 (\ell_X(\beta)/2)\right ) \cosh^2 \left(\ell_X(\alpha) - a\right )= \cosh^2 \left( \frac 12 \ell_Y(\alpha)\right )
\end{equation}
When $\epsilon = 0$, we get $\ell_Y(\alpha) = \ell_X(\alpha)$, as expected. For fixed $\ell_X(\alpha)$, as $\ell_X(\beta)$ increases, both $\epsilon_{\max}$ and the left hand side of this equation decrease. As $\epsilon$ increases, $a$ increases, so the left hand side decreases. Thus, the left hand side decreases monotonically to its minimum at $\epsilon = \epsilon_{\max}$.  In particular, for each $\ell_X(\beta)$, and $\epsilon_{\max}$ depending on $\ell_X(\beta)$, there is an $\ell_{\min}$ so that for all $\ell$ so that $\ell_{\min} < \ell < \ell_X(\alpha)$, there is an $\epsilon$ with $0 < \epsilon < \epsilon_{\max}$ for which 
\[
 \ell_Y(\alpha) = \ell
\]  
Conversely, for any $\epsilon < \epsilon_{\max}$, we have $ \ell_Y(\alpha) = \ell$ for $\ell$ between $\ell_{\min}$ and $\ell_X(\alpha)$.
 \end{proof}
 \begin{rem}
\label{rem:beta_decreases_epsilon_increases}
By Equation (\ref{eq:change_of_length_epsilon}), if $\ell_Y(\alpha)$ is fixed and $\ell_X(\beta)$ decreases, then $\epsilon$ must increase.
\end{rem}

\subsection{Change in length of arcs}
We now estimate how the length of an arc on $X$ changes when we remove an $\epsilon$ strip about $\beta$. Papadopoulos and Th\'eret prove a version of the lemma we need for simple closed curves on $X$ (Proposition 2.2 in \cite{PT09}). We give a slight extension of it here:   
\begin{lem}
\label{lem:Length_change_arcs}
Let $\gamma$ be any (not necessarily simple) geodesic arc on $X$ with endpoints on $\partial X$. Then there is a constant $C(\epsilon)$ so that 
 \[
 \ell_X(\gamma) - \ell_Y(f(\gamma)) >  C(\epsilon) i(\gamma, \beta) - C(\epsilon)
 \]
 for $C(\epsilon) = \min\{\epsilon, \log(1+e^{-\ell_X(\beta)/2}\epsilon^2)\}$.
\end{lem}

\begin{proof}
 Let $Y$ be the result of removing an $\epsilon$-strip about $\beta$ from $X$. The map $f: X \to Y$ defined above is an isometry outside of $S_\epsilon$. So we just need to see how the length of $\gamma$ changes inside $S_\epsilon$. Since $f$ is 1-Lipschitz, we have that $\ell_Y(f(\gamma \cap S_\epsilon)) \leq \ell_X(\gamma \cap S_\epsilon)$. We just wish to quantify how much length we lose.

 For each intersection between $\gamma$ and $\beta$, we get a subarc $g$ in $S_\epsilon \cap \gamma$. If both endpoints of $g$ are on $\alpha$, then $g$ must be all of $\gamma$. In particular, at most one intersection between $\gamma$ and $\beta$ contributes such a subarc $g$. So we assume that for the other $i(\gamma, \beta) -1$ intersections, the corresponding subarc $g$ crosses both $\beta$ and at least one of $\beta_1$ and $\beta_2$. Let $g$ be such a subarc. 
 
By the computation in the proof of Proposition 2.2 in \cite{PT09}, 
 \[
 \ell_X(g) - \ell_Y(f(g)) \geq \min\{\epsilon, \log(1+e^{-\ell(\beta)/2}\epsilon^2)\}
 \]
 
Let $C(\epsilon) = \min\{\epsilon, \log(1+e^{-\ell(\beta)/2}\epsilon^2)\}$. Summing this inequality over all $i(\gamma, \beta)-1$ subarcs, we get our result.
\end{proof}
\begin{rem}
  Note that $C(\epsilon)$ goes to 0 as $\epsilon$ goes to 0. In fact, for $\epsilon = \epsilon_{\max}(\ell_X(\beta))$, we have that $C(\epsilon_{\max})$ decays like $e^{-\frac 52 \ell_X(\beta)}$ as $\ell_X(\beta)$ goes to infinity. That is, their ratio tends to a non-zero constant. This follows from the fact that $\epsilon_{\max}$ decays like $e^{-\frac 32 \ell_X(\beta)}$ as $\ell_X(\beta)$ goes to infinity by Remark \ref{rem:decay_of_epsilon_max} and the fact that $\log(1 + x)$ is asymptotic to $x$ as $x$ goes to 0.
\end{rem}

 \section{Surgery on closed surfaces}
 \label{sec:Closed_surfaces} 
 
 \begin{figure}[h!]
 \includegraphics{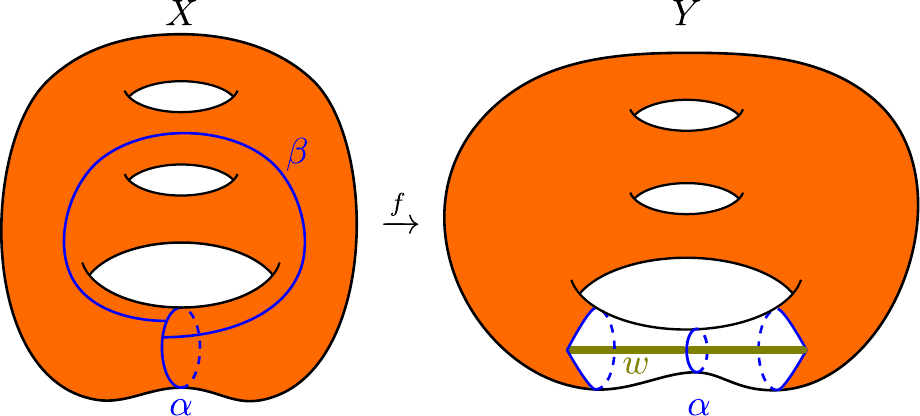}
 \caption{Given a closed surface $X$ with simple closed curve $\alpha$ and arc $\beta$, we obtain a new surface $Y$.}
 \label{fig:Closed_surface_construction}
\end{figure}
 Now let $X$ be a closed hyperbolic surface. For what follows, refer to Figure \ref{fig:Closed_surface_construction}. Let $\alpha$ be a simple closed geodesic on $X$. Given an arc $\beta$ with both endpoints on $\alpha$, we say that both endpoints of $\beta$ are on the \textbf{same side} of $\alpha$ if, given an orientation of $\alpha$, $\beta$ emerges from these two endpoints in the same direction. Otherwise, we say that the endpoints of $\beta$ are on \textbf{either side} of $\alpha$.
 \begin{defi}
 \label{def:Orthogonal_arc_system}
Let $\alpha$ be a simple closed geodesic. We say $\beta = \{\beta^+, \beta^-\}$ is an \textbf{orthogonal arc system} for $\alpha$ (Figure \ref{fig:Orthogonal_arc_system}) if $\beta^+, \beta^-$ are simple geodesic arcs that are orthogonal to $\alpha$ at their endpoints, are disjoint from $\alpha$ and each other in their interior, and either: 
\begin{itemize}
 \item The endpoints of $\beta^+$ are on either side of $\alpha$, and $\beta^-$ is empty, or 
 \item Both endpoints of $\beta^+$ are on one side of $\alpha$, and both endpoints of $\beta^-$ are on the other side of $\alpha$
\end{itemize}
 \end{defi}
  For any metric $X$, and orthogonal arc system $\beta = \{\beta^+, \beta^-\}$ for a simple closed geodesic $\alpha$, we let 
  \[
   \Col_X(\beta) = \Col_X(\beta^+) + \Col_X(\beta^-)
  \]
  to simplify notation later on.

 \begin{figure}[h!]
  \centering 
  \includegraphics{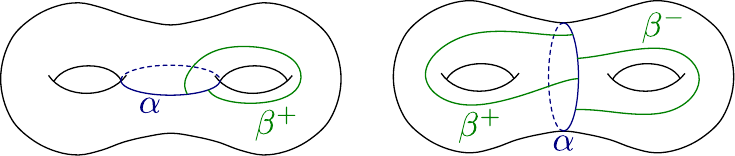}
  \caption{Two possibilities for an orthogonal arc system for $\alpha$}
  \label{fig:Orthogonal_arc_system}
 \end{figure}

 Let $\beta = \{\beta^+, \beta^-\}$ be an orthogonal arc system for some simple closed geodesic $\alpha$ on $X$. For all $\epsilon$ small enough, we describe how to \textbf{remove an $\epsilon$-strip about $\beta$ from $X$} to get a new closed surface $Y$. If $\beta$ consists of two arcs, then $\epsilon$ will refer to the width of the strip about $\beta^+$. The width of the strip about $\beta^-$ will depend on $\epsilon$, so we don't include it in the notation.
 
 \subsection{Constructing $Y_\alpha$} 
 Cut $X$ along $\alpha$ to get a surface $X_\alpha$ with two geodesic boundary components. We first wish to remove strips from $X_\alpha$ to get a new surface with boundary, $Y_\alpha$, whose two boundary components have the same length. There are two cases.
 
 \subsection*{Constructing $Y_\alpha$, case 1.} First suppose $\beta$ consists of just one arc. Then this arc, which we still refer to as $\beta$, joins the two different boundary components of $X_\alpha$. Let $\epsilon_{\max}$ be the constant from Equation (\ref{eq:epsilon_max}). For any $\epsilon < \epsilon_{\max}$, we simply remove an $\epsilon$-strip from $X_\alpha$ to get a new surface $Y_\alpha$.
 
  Suppose $\beta$ joins boundary components $\alpha^+$ and $\alpha^-$. Then we claim that $\ell_Y(\alpha^+) = \ell_Y(\alpha^-)$. We see this as follows. We chose the $\epsilon$-strip $S_\epsilon$ so that the mutual orthogonal between its two sides bisects $\beta$. Thus, the $X$-lengths of $S_\epsilon \cap \alpha^+$ and $S_\epsilon \cap \alpha^-$ are equal. So, if $f: X \to Y$ is the continuous map described in the construction above (Section \ref{subsec:Geometry}), then $\ell_Y(f(\alpha^+)) = \ell_Y(f(\alpha^-))$. Moreover, by symmetry, all four angles of $S_\epsilon \cap X$ are the same. So the acute angles of $f(\alpha^+)$ and $f(\alpha^-)$ are equal. Therefore, the $Y$- geodesic representatives of $\alpha^+$ and $\alpha^-$ must have the same length.

 \subsection*{Constructing $Y_\alpha$, case 2.} Now suppose $\beta$ consists of two (non-empty) arcs, $\beta^+$ and $\beta^-$. Then both endpoints of $\beta^+$ are on a boundary component $\alpha^+$ of $X_\alpha$, and both endpoints of $\beta^-$ are on a boundary component $\alpha^-$.

 We now show that we can remove strips about $\beta^+$ and $\beta^-$ to get a new surface with boundaries of equal length.
\begin{lem}
There is an $\epsilon$ depending only on $\ell_X(\beta^+)$ and $\ell_X(\beta^-)$ so that we can first remove an $\epsilon^+$ strip about $\beta^+$, then an $\epsilon^-$ strip about $\beta^-$, and get a new surface $Y_\alpha$ with
 \[
  \ell = \ell_{Y_\alpha}(\alpha^+) = \ell_{Y_\alpha}(\alpha^-)
 \]
 for any $0< \epsilon^+ < \epsilon$, and some $\epsilon^-$ depending only on $\epsilon^+$.
\end{lem}

\begin{proof}
Consider the function $\ell_{\min}$ defined in Lemma \ref{lem:Geometry_alpha_length}. Let
\[
L = \max\{\ell_{\min}(\ell_X(\beta^+)), \ell_{\min}(\ell_X(\beta^-))\}
\]
Then by Lemma \ref{lem:Geometry_alpha_length}, there is an $\epsilon_\beta$ so that for all $\epsilon^+ < \epsilon_\beta$, we can remove an $\epsilon^+$-strip about $\beta^+$, and get a surface $Z_\alpha$ with 
\[
 L \leq \ell_{Z_\alpha}(\alpha^+)  \leq \ell_X(\alpha)
\]
and in fact, we can achieve any value in that range. Choose any $\ell$ with $L < \ell < \ell_X(\alpha)$, and pick $\epsilon$ so that 
\[
\ell_{Z_\alpha}(\alpha^+) = \ell
\]

 Let $f: X_\alpha \to Z_\alpha$ be the induced 1-Lipschitz map. Since $\alpha^-$ is disjoint from $\beta$, $f$ sends $\alpha^-$ to a boundary component of $Z_\alpha$ with $\ell_{Z_\alpha}(\alpha^-) = \ell_X(\alpha)$. Moreover, since $\beta^-$ is disjoint from $\beta^+$, and both of its endpoints are on $\alpha^-$, the arc $\beta^-$ cannot have entered the $\epsilon^+$-strip about $\beta^+$. Thus,
 \[
  \ell_{Z_\alpha}(\beta^-) = \ell_{X}(\beta^-)
 \]
 where the length of $\beta^-$ in $X$ is the same as its length in $X_\alpha$.
  Thus, for the $\ell$ chosen above, we have 
 \[
 \ell_{\min}(\ell_{Z_\alpha}(\beta^-)) < \ell < \ell_{Z_{\alpha}}(\alpha^-)
 \]
 So by Lemma  \ref{lem:Geometry_alpha_length}, there is an $\epsilon^-$ so that when we remove an $\epsilon^-$ strip about $\beta'$, we get a new surface $Y_\alpha$ so that 
 \[
  \ell_{Y_\alpha}(\alpha^-) = \ell
 \]
Since $\beta^-$ is disjoint from $\alpha^+$, we still have $\ell_{Y_\alpha}(\alpha^+) = \ell$. In particular, the two boundary components of $Y_\alpha$ have the same length. 

\end{proof}

So given a curve $\alpha$ and an (ordered) orthogonal arc system $\beta = \{\beta^+, \beta^-\}$, we say we remove an $\epsilon^+$-strip about $\beta$ from $X_\alpha$ to get a new hyperbolic surface $Y_\alpha$. Since $\epsilon^-$ depends on $\epsilon^+$, we only give $\epsilon^+$, and not $\epsilon^-$.

\subsection{Gluing the boundaries}
\label{sec:Gluing_boundaries}
 Next we glue the boundary components of $Y_\alpha$ back together. This is a bit delicate, as we wish to do so without introducing extra partial twisting along $\alpha$.  Let $f: X_\alpha \to Y_\alpha$ be the 1-Lipschitz map induced by removing the strip(s) from $X_\alpha$. (If we had to remove two strips, then $f$ is the composition of the maps we get at each step.) Let $C^+, C^-$ be the two cylinder components of $Y_\alpha \setminus f(X_\alpha)$. If $\alpha^+_Y$ and $\alpha^-_Y$ denote the $Y_\alpha$-geodesic representative of $\alpha^+$ and $\alpha^-$, respectively, then we label the cylinders so that 
 \[
  \partial C^+ = f(\alpha^+) \cup \alpha^+_Y, \text{ and } \partial C^- = f(\alpha^-) \cup \alpha^-_Y
 \]
 
For each point $x \in f(\alpha^+)$, we will describe how to assign a point $p(x) \in \alpha^+_Y$ (Figure \ref{fig:Cylinder_point_assignment}). We assign points in this way in order to make computations easier in Lemma \ref{lem:Length_of_bridge_arcs} below.

If $x$ is a singular point of $f(\alpha^+)$, then we drop a perpendicular from $x$ to the point $p(x) \in \alpha^+_Y$. Cut $C^+$ along one of these perpendiculars. Then we get a hyperbolic polygon $P$, where $P$ is either a quadrilateral, or the union of two quadrilaterals, depending on whether $\beta$ intersects $\alpha^+$ once or twice. In either case, we isometrically identify the portion of $\partial P$ that lies on $f(\alpha^+)$ with the interval $[0,\ell]$ for $\ell = \ell_{Y_\alpha}(f(\alpha^+))$. We also identify the portion of $\partial P$ that lies on $\alpha^+_Y$ with the interval $[0,m]$ for $m = \ell_{Y_\alpha}(\alpha^+)$. If $x$ lies at time $t\ell \in [0, \ell]$ for $t \in [0,1]$, then we let $p(x)$ be the point at time $tm \in[0,m]$. We let $\sigma(x)$ be the geodesic arc in $P$ joining $x$ to $p(x)$.
\begin{figure}[h!]
 \centering 
 \includegraphics{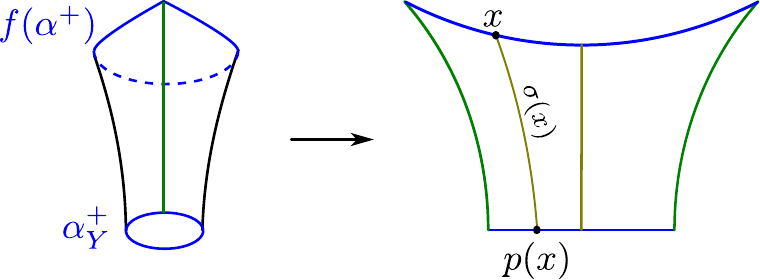}
 \caption{Each point $x$ on $f(\alpha^+)$ is assigned a point $p(x)$ on $\alpha_Y^+$. They are joined by the arc $\sigma(x)$.}
 \label{fig:Cylinder_point_assignment}
\end{figure}

We do the same procedure to $C^-$ to identify each point $x \in f(\alpha^-)$ with a point $p(x) \in \alpha^-_Y$. We then define the arc $\sigma(x)$ in the same way. 

Choose a fixed point $x_0$ on $\alpha$. This gives us two points, $x_0^+$ and $x_0^-$ on $\alpha^+$ and $\alpha^-$, respectively. Mapping these to $Y_\alpha$, we get points $f(x_0^+)$ and $f(x_0^-)$ on $f(\alpha^+)$ and $f(\alpha^-)$. Then we glue $\alpha^+_Y$ to $\alpha^-_Y$ so that the projections of these points, $p(f(x_0^+))$ and $p(f(x_0^-))$, are glued together. We choose the marking in the resulting cylinder $C^+ \cup C^-$ so that the arc $\sigma(f(x_0^+)) \cup \sigma(f(x_0^-))$ has zero twist. Suppressing the dependence on $\beta'$, if any, we call $Y$ the result of removing an $\epsilon$-strip about $\beta$ from $X$.

\subsection{Change in geodesic length}
We will also use the above construction to estimate how the lengths of geodesics change from $X$ to $Y$. For this, we need to work again in $X_\alpha$ and $Y_\alpha$. We use the same assignment as above between points $x \in f(\alpha^\pm)$ and $p(x) \in \alpha^\pm_Y$, and the preferred geodesic segments $\sigma(x)$ between them. We will use the following lemma:
\begin{lem}
\label{lem:Length_of_bridge_arcs}
 The length of the segment $\sigma(z)$ is at most 1 for each $z \in f(\alpha^\pm)$.
\end{lem}
\begin{proof}
We will work with $C^+$, as the proof for $C^-$ is identical. Let $x \in f(\alpha^+)$ be a singular point of $f(\alpha^+)$. Then the arc $\sigma(x)$ joining $x$ to $p(x)$ is orthogonal to $\alpha^+_Y$. Note that if $y$ is another singular point of $f(\alpha^+)$, then it lies on the midpoint of $f(\alpha^+) \setminus \{x\}$, so by construction $\sigma(y)$ is also orthogonal to $\alpha^+_Y$.

Cut $C^+$ along $\sigma(x)$ (or along $\sigma(x)$ and $\sigma(y)$ if $f(\alpha^+)$ has two singular points). This cuts $C^+$ into either one or two Sacchieri quadrilaterals, as in the proof of Lemma \ref{lem:Geometry_collar_width}. Recall that, as $x$ is the singular point, we showed in that lemma that the length of $\sigma(x)$ is at most 1 (and is the same as the length of $\sigma(y)$, if there are two singular points). Now, this quadrilateral has a side that lies on $f(\alpha^+)$ and a side that lies on $\alpha^+_Y$. The segment joining the midpoints of these two sides is orthogonal to both, and is $\sigma(x')$ for some point $x'$. Cutting along it, we get two Lambert quadrilaterals. Let $Q$ be one of these two quadrilaterals (Figure \ref{fig:Cylinder_sigma_length}).

\begin{figure}[h!]
 \centering 
 \includegraphics{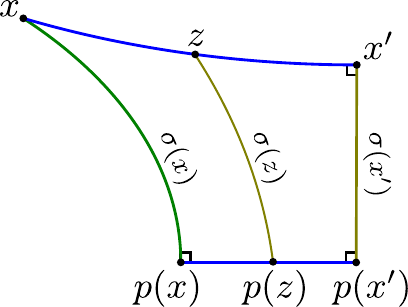}
 \caption{Since $\sigma(x)$ has length at most 1, where $x$ is a singular point, $\sigma(z)$ has length at most 1 for all $z$.}
 \label{fig:Cylinder_sigma_length}
\end{figure}

Suppose $Q$ has vertices $x, x' \in f(\alpha^+)$. Then it has vertices $p(x), p(x') \in \alpha^+_Y$. The arc $\overline{x x'}$ is isometric to the interval $[0,\ell]$ for some $\ell$, and the arc $\overline{p(x) p(x')}$ is isometric to the interval $[0,m]$ for some $m$. As we have cut $C^+$ along arcs of the form $\sigma(\cdot)$, if $z \in \overline{x x'}$ is a point at time $t\ell$, then $p(z)$ is a point at time $tm$. Let $Q'$ be the quadrilateral with vertices $x' z p(z) p(x')$. Then $Q'$ is a quadrilateral contained in $Q$, that shares the side $\overline{x' p(x')}$. Note that $Q'$ has two right angles, at $x'$ and $p(x')$, but the angle at $p(z)$ will not be a right angle unless $z = x$ or $z = x'$. 

We can use the hyperbolic laws of cosines and sines to show that
\[
 \cosh \ell(\overline{z p(z)}) = 2\sinh^2 \ell(\overline{x' p(x')}) \cosh(t\ell) \cosh(tm) + \cosh(t\ell - tm)
\]
and likewise,
\[
 \cosh \ell(\overline{x p(x)}) = 2\sinh^2 \ell(\overline{x' p(x')}) \cosh(\ell) \cosh(m) + \cosh(\ell - m)
\]
As $0 \leq t \leq 1$, we see that $\cosh \ell(\overline{z p(z)}) \leq \cosh \ell(\overline{x p(x)})$. In particular, $\ell(\sigma(z)) = \ell(\overline{z p(z)}) \leq 1$.
\end{proof}

Now we can estimate how much the length of a geodesic current changes from $X$ to $Y$. For this, we need to define the intersection between a current $\mu$ and an arc $\beta$. 
\begin{defi}
\label{def:Intersection_with_arcs}
 Let $\mu \in \C(S)$. Identify the universal cover of $X$ with $\h$. Let $\tilde \beta$ be any lift of an arc $\beta$ on $X$. Let $I(\tilde \beta)$ be the set of all complete geodesics in $\h$ that intersect $\tilde \beta$. Then the \textbf{intersection number of $\mu$ and $\beta$} is defined to be
 \[
  i(\mu, \beta) = \mu(I(\tilde \beta))
 \]
\end{defi}
\begin{figure}[h!]
 \centering 
 \includegraphics{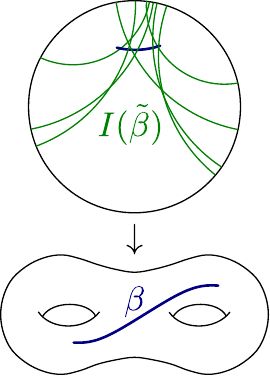}
 \caption{Defining the intersection number between an arc an a current}
 \label{fig:ArcIntersection}
\end{figure}

Since $\mu$ is $\pi_1(S)$-invariant, we see that $i(\mu, \beta)$ is independent of the lift we chose of $\beta$. Thus, the intersection number between currents and arcs is well-defined. We also see immediately that intersection number between a current and an arc is linear in $\mu$. Moreover, if $\mu_n$ is a sequence of currents converging to a current $\mu$, then if we take $\beta$ to be an open arc, so that $I(\beta)$ is open, then
\[
 \liminf i(\mu_n, \beta) \geq i(\mu, \beta)
\]

\begin{lem}
\label{lem:Length_change_closed}
 Let $\mu$ be a geodesic current. Let $X\in \T(S)$. Let $\alpha$ be a simple closed geodesic in $X$, and let $\beta = \{\beta^+, \beta^-\}$ be an orthogonal arc system for $\alpha$ with $i(\mu, \beta^+) \geq i(\mu, \beta^-)$. Let $Y$ be the result of removing an $\epsilon$-strip about $\beta$ from $X$. Then,
 \[
  i(X, \mu) - i(Y, \mu) \geq \frac 12 C(\epsilon)  i(\mu, \beta) - \Big (2+4\Col_X(\beta) + C(\epsilon)\Big ) i(\mu, \alpha)
 \]
where we define $\Col_X(\beta) = \Col_X(\beta^+) + \Col_X(\beta^-)$, and $C(\epsilon)$ is the constant from Lemma \ref{lem:Length_change_arcs}.
\end{lem}
Recall that when we remove an $\epsilon$-strip about $\beta = \{\beta^+, \beta^-\}$, the $\epsilon$ refers to the width of the strip about $\beta^+$. Thus, $C(\epsilon)$ is the constant from Lemma \ref{lem:Length_change_arcs} defined for $\beta^+$.
\begin{proof}
 First, we work with a closed geodesic $\gamma$ on $X$. We do the case where $\beta$ is just a single arc first, and then explain what needs to change when it consists of two arcs. 
 
 We first describe how to make a new curve $\gamma_Y$ in the free homotopy class of $\gamma$ on $Y$, and then we estimate its length. Let $X_\alpha$ be the result of cutting $X$ along $\alpha$. Cutting $\gamma$ along $\alpha$, we get a collection of geodesic arcs on $X_\alpha$. Abusing notation, we still call this collection of arcs $\gamma$. Let $Y_\alpha$ be the result of cutting $Y$ along $\alpha$. Let $f: X_\alpha \to Y_\alpha$ be the induced 1-Lipschitz map. 
 
 Refer to Figure \ref{fig:Closed_surface_mapping_curves} for what follows. Each point $x$ on $\gamma \cap \alpha$ in $X$ corresponds to endpoints $x^+$ on $\alpha^+$ and $x^-$ on $\alpha^-$ of $\gamma$ on $X_\alpha$. Then $f(\gamma)$ has endpoints $f(x^+)$ and $f(x^-)$ on $f(\alpha^+)$ and $f(\alpha^-)$, respectively. Concatenate $f(\gamma)$ at these points with arcs $\sigma(f(x^+))$ and $\sigma(f(x^-))$ in $Y_\alpha$ joining $f(x^\pm)$ to points $p(f(x^\pm))$ on $\alpha_Y^\pm$, respectively. See Section \ref{sec:Gluing_boundaries} and Figure \ref{fig:Cylinder_point_assignment} for a definition of the arcs $\sigma$. Once we glue $\alpha^+_Y$ to $\alpha^-_Y$ to get $Y$, the points $p(f(x^+))$ and $p(f(x^-))$ might not coincide. We join them with an arc $\delta(x)$ in $\alpha_Y$ that does not intersect our distinguished point $p(f(x^+_0))$ that we used to glue $\alpha^+_Y$ to $\alpha^-_Y$.
 
 By Lemma \ref{lem:Length_of_bridge_arcs}, the length of $\sigma(f(x))$ is at most 1 for all $x$. Moreover, we will show that $\ell_Y(\delta(x))$ is at most $4 \Col_X(\beta)$. To see this, note that the map from $X_\alpha$ to $Y_\alpha$ contracts $\alpha^+$ and $\alpha^-$ along subarcs of length $2a$, where $a$ is the constant from Equation (\ref{eq:define_a}). By construction, and the computations in the proof of Lemma \ref{lem:Geometry_embedded}, 
 \[
  a \leq \frac 12 \Col_X(\beta)
 \]
But the assignment of points from $f(\alpha^\pm)$ to $\alpha^\pm_Y$ shrinks distances. Thus, $p(f(x^+))$ and $p(f(x^-))$ are at most distance $4 \Col_X(\beta)$ apart. Since the assignment of points on $f(\alpha^\pm)$ to points on $\alpha^\pm_Y$ preserves the order of points, and $p(f(x^+_0))$ is glued to $p(f(x^-_0))$, the shortest arc from $p(f(x^+))$ to $p(f(x^-))$ does not intersect $p(f(x^+_0))$. In particular, this shortest arc is exactly $\delta(x)$. Thus, we join points $f(x^+)$ and $f(x^-)$ with a piecewise geodesic arc of length at most $2 + 4\Col_X(\beta)$.

\begin{figure}[h!]
 \centering 
 \includegraphics{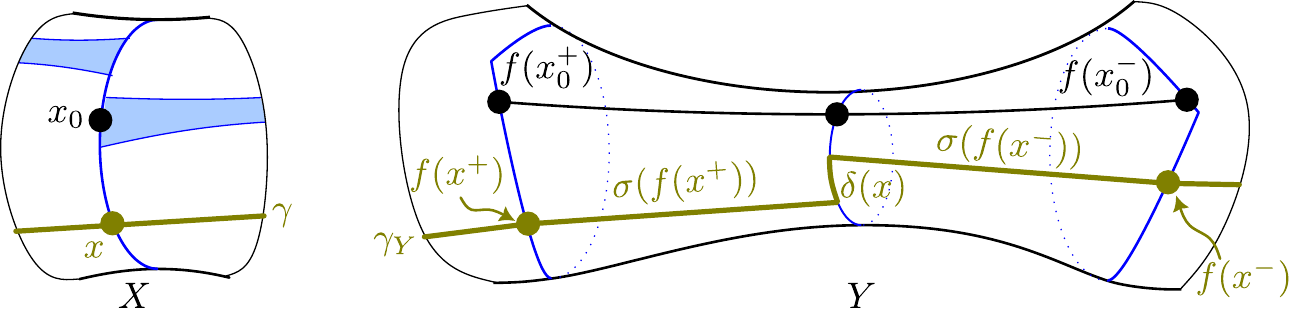}
 \caption{A curve $\gamma$ and $\gamma_Y$ in neighborhoods of $\alpha$ and $\alpha_Y$, respectively. Note that $\delta(x)$ does not intersect the arc joining $f(x_0^+)$ and $f(x_0^-)$.}
 \label{fig:Closed_surface_mapping_curves}
\end{figure}

Concatenating $f(\gamma)$ with all of the above arcs gives a curve $\gamma_Y$ on $Y$ that is in the free homotopy class of $\gamma$. In Lemma \ref{lem:Length_change_arcs}, we estimated how the length of each arc changes from $X_\alpha$ to $Y_\alpha$. Summing this estimate over all the arcs in $\gamma$ gives us 
\[
 \ell_Y(f(\gamma)) < \ell_X(\gamma) - C(\epsilon) \cdot  i(\gamma, \beta) + C(\epsilon)i(\gamma, \alpha)
\]
as there are $i(\gamma, \alpha)$ arcs. Note that if some arc $\gamma_i$ does not intersect $\beta$, then $i(\gamma_i, \beta) = 0$, and $\ell_Y(f(\gamma_i)) < \ell_X(\gamma_i)$ because $f$ is 1-Lipschitz. In particular, we still have $\ell_Y(f(\gamma_i)) < \ell_X(\gamma_i) + C(\epsilon)$, so we can treat these arcs the same as arcs that do intersect $\beta$.

Then, the arcs we use to join the endpoints of $f(\gamma)$ to form $\gamma_Y$ each have length at most $2 + 4\Col_X(\beta)$. As there are exactly $i(\gamma, \alpha)$ such arcs, we have
\[
 \ell_Y(\gamma) < \ell_X(\gamma) - C(\epsilon) \cdot  i(\gamma, \beta) + \Big (2+4\Col_X(\beta)+ C(\epsilon)\Big )i(\gamma, \alpha)
\]
which is exactly what we want.

Now, suppose $\beta$ consists of (non-empty) arcs $\beta^+$ and $\beta^-$. Suppose $i(\gamma, \beta^+) > i(\gamma, \beta^-)$. Then we first remove an $\epsilon = \epsilon^+$-strip about $\beta^+$ from $X_\alpha$ to get a surface $Z_\alpha$, and then we remove an $\epsilon^-$-strip about $\beta^-$ from $Z_\alpha$ to get $Y_\alpha$. Let 
\begin{align*}
 f^+&: X_\alpha \to Z_\alpha\\
 f^-&: Z_\alpha \to Y_\alpha
\end{align*}
be the induced 1-Lipschitz maps.

Let $f = f^- \circ f^+$. Again, we let $\gamma$ denote the collection of arcs we get by cutting $\gamma$ along $\alpha$. Applying Lemma \ref{lem:Length_change_arcs}, we see 
\[
 \ell_{Z_{\alpha}}(f^+(\gamma))  < \ell_{X_\alpha}(\gamma) - C(\epsilon^+) \cdot  i(\gamma, \beta^+) + C(\epsilon^+)i(\gamma, \alpha^+)
\]
Since $f^-$ is a 1-Lipschitz map, so we have \[\ell_{Y_\alpha}(f(\gamma)) < \ell_{Z_{\alpha}}(f^+(\gamma))\]

Now we construct the curve $\gamma_Y$ on $Y$ exactly as above.  Note that the map $f^+$ contracts $\alpha^+$ by an arc of length $4a^+$, where $a^+$ is the constant in Equation \ref{eq:define_a} that is bounded above by $\Col_X(\beta^+)$. Likewise, since $\beta^+$ and $\beta^-$ are disjoint, $\beta^-$ cannot enter the $\epsilon^+$-strip about $\beta^+$. Thus, $\ell_Z(f^+(\beta^-)) = \ell_X(\beta^-)$. So, $f^-$ contracts $\alpha^-$ by an arc of length $4a^-$, for $a^- < \Col_X(\beta^-)$. Then $\ell_Y(\gamma_Y)$ is bounded above by $\ell_Y(f(\gamma))$ plus $2 +4\Col_X(\beta^+) + \Col_X(\beta^-)$ for each point of $\gamma$ on $\alpha$ in $X$. In other words, 
\[
 \ell_Y(\gamma_Y) = \ell_{Y_{\alpha}}(f(\gamma)) + (2 + 4\Col_X(\beta)) i(\gamma, \alpha)
\]
where we define $\Col_X(\beta) = \Col_X(\beta^+) + \Col_X(\beta^-)$ to simplify notation. 

Using our estimate of $\ell_Y(\gamma_Y)$ above and the fact that $i(\gamma, \beta^+) \geq i(\gamma, \beta^-)$, we get
\begin{equation}
\label{eq:Length_change_closed}
 \ell_Y(\gamma) < \ell_{X}(\gamma) - \frac 12 C(\epsilon^+) \cdot  i(\gamma, \beta) + \Big (2 + 4\Col_X(\beta) + C(\epsilon^+)\Big )i(\gamma, \alpha)
\end{equation}

Next, take a current $\mu$. Suppose $c_n \gamma_n \to \mu$ for $c_n > 0$ and $\gamma_n$ a sequence of closed geodesics. Then for each $n$, we can rewrite Equation (\ref{eq:Length_change_closed}) in terms of intersection number so that
\[
i(Y,c_n\gamma_n) < i(X,c_n\gamma_n) - \frac 12 C(\epsilon) \cdot  i(c_n\gamma_n, \beta) + \Big (2 + 4\Col_X(\beta) + C(\epsilon^+)\Big)i(c_n\gamma_n, \alpha)
\]
As the usual intersection number is continuous, and intersection between currents and open arcs is lower semicontinuous, we can take limits to get 
\[
 i(Y, \mu) \leq i(X,\mu) - \frac 12 C(\epsilon) \cdot  i(\mu, \beta) + \Big (2 + 4\Col_X(\beta) + C(\epsilon^+)\Big )i(\mu, \alpha)
\]
as desired.
\end{proof}

\section{Mixed collar lemma for length minimizers}
\label{sec:Mixed_collar_theorem}
We are now ready to prove various mixed collar lemmas. First, we will prove a version for a simple closed curve and an orthogonal arc system. 
\begin{prop}
\label{prop:Mixed_Collar_Orthogonal_Arcs}
 Let $\mu$ be a filling current, and let $\pi(\mu)$ be its length minimizer. Let $\alpha$ be any simple closed geodesic, and let $\beta = \{\beta^+, \beta^-\}$ be an orthogonal arc system for $\alpha$. Then,
 \[
  i(\mu, \alpha) > \delta(\ell_{\pi(\mu)}(\beta)) \; i(\mu, \beta)
 \]
 where $\delta$ is a continuous, decreasing function $\delta : \R_+ \to \R_+$ depending only on the Euler characteristic $\chi(S)$. Moreover, $\delta$ is bounded above, and goes to zero as $x$ goes to infinity.
\end{prop}
\begin{proof}
Given $\beta = \{\beta^+,\beta^-\}$, there is a range of values $(0, \epsilon_\beta)$ so that for all $\epsilon \in (0, \epsilon_\beta)$, we can remove an $\epsilon$-strip about $\beta$. Note that $\epsilon_\beta$ depends only on $\ell_{\pi(\mu)}(\beta^+)$ and $\ell_{\pi(\mu)}(\beta^-)$. Choose $\epsilon = \frac 12 \epsilon_\beta$. Then by Lemma \ref{lem:Length_change_closed},
  \[
  \ell_{\pi(\mu)}(\mu) - \ell_Y(\mu) \geq \frac 12 C(\epsilon) \cdot  i(\mu, \beta) - \Big(2  + 4\Col_{\pi(\mu)}(\beta) + C(\epsilon)\Big)i(\mu, \alpha)
 \]
 where $\Col_{\pi(\mu)}(\beta) = \Col_{\pi(\mu)}(\beta^+)+\Col_{\pi(\mu)}(\beta^-)$, and $C(\epsilon)$ is the constant from Lemma \ref{lem:Length_change_arcs}.
 
 But $\pi(\mu)$ is the length minimizer of $\mu$. Thus, we must have $\ell_{\pi(\mu)}(\mu) - \ell_Y(\mu) < 0$, so
 \[
 \frac 12 C(\epsilon) \cdot  i(\mu, \beta) - (2  + 4\Col_{\pi(\mu)}(\beta) + C(\epsilon))i(\mu, \alpha) < 0
 \]
In particular, 
\[
i(\mu, \alpha) > \frac{C(\epsilon)}{4  + 8\Col_{\pi(\mu)}(\beta) + 2C(\epsilon)} i(\mu, \beta)
\]
By our choice of $\epsilon$, the constant
\[
d(\ell_{\pi(\mu)}(\beta)) = \frac{C(\epsilon)}{4  + 8\Col_{\pi(\mu)}(\beta) + C(\epsilon)}
\]
only depends on the lengths of $\ell(\beta^+), \ell(\beta^-)$ and $\chi(S)$.
Moreover, as $\ell_{\pi(\mu)}(\beta)$ goes to infinity, $\epsilon_\beta$ will go to 0, so $C(\epsilon)$ goes to zero. Thus, $d(\ell_{\pi(\mu)}(\beta))$ will go to 0. The fact that $d(x)$ is uniformly bounded above is also clear. Lastly, as $d(x)$ is a continuous, strictly positive function that goes to 0 as $x$ goes to infinity, we can bound it from below by a decreasing, strictly positive function $\delta(x)$ that also goes to 0 as $x$ tends to infinity. We use this function from now on.
\end{proof}

\subsection{Twisting numbers}
To prove the mixed collar theorem for pairs of simple closed curves, we need to define a twisting number for a simple arc or closed curve $\beta$ about a simple closed curve $\alpha$. First, let $X$ be a closed hyperbolic surface, and let $\alpha$ be a simple closed geodesic on $X$. Let $\beta$ be any simple arc with both endpoints on $\alpha$. Let $\beta'$ be the arc that is orthogonal to $\alpha$ at its endpoints, and that is homotopic to $\beta$ via a homotopy that keeps its endpoints on $\alpha$. 
\begin{figure}[h!]
 \centering 
 \includegraphics{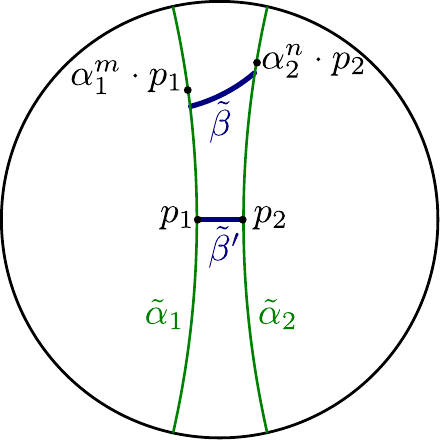}
 \caption{}
 \label{fig:Intersection_and_twisting}
\end{figure}
Lift $\beta'$ to an arc $\tilde \beta'$ in the universal cover so that it has endpoints $p_1, p_2$ on lifts $\tilde \alpha_1$ and $\tilde \alpha_2$ of $\alpha$, respectively  (Figure \ref{fig:Intersection_and_twisting}). Since $\beta'$ and $\beta$ are homotopic, there exists a lift $\tilde \beta$ of $\beta$ with endpoints on $\tilde \alpha_1$ and $\tilde \alpha_2$. Suppose $\tilde \beta$ has endpoints $q_1 \in \tilde \alpha_1$ and $q_2 \in \tilde \alpha_2$. Take elements $\alpha_1, \alpha_2 \in \pi_1(X)$ that act by translation on $\tilde \alpha_1$ and $\tilde \alpha_2$, respectively. Choose $\alpha_1$ so that $q_1$ is between $p_1$ and the attracting fixed point of $\alpha_1$, and do the same for $\alpha_2$ with respect to $p_2$.

If $x,y$ are points on $\tilde \alpha_2$, then the interval $(x,y]$ is the subarc of $\tilde \alpha_2$ of all points between $x$ and $y$, including $x$ and excluding $y$. Then we have
\[
 q_1 \in (\alpha_1^{m-1} \cdot p_1, \alpha_1^m \cdot p_1] \text{ and } q_2 \in (\alpha_2^{n-1} \cdot p_2, \alpha_2^n \cdot p_2]
\]
for $n, m \in \N$. We let 
\[
 \tau_\alpha(\beta) = n + m
\]
(If $p_1 = q_1$, we take $m = 0$, and likewise if $p_2 = q_2$). We say $\tau_\alpha(\beta)$ is the \textbf{twisting number} of the arc $\beta$ about $\alpha$. Note that this means $\tau_\alpha(\beta) \geq 2$ for all $\beta$.

Suppose now that $\beta$ is a simple closed geodesic that intersects $\alpha$. Then $\alpha$ cuts $\beta$ into arcs $\beta_1, \dots, \beta_k$ for $k = i(\alpha, \beta)$. We then say 
\[
 \tau_\alpha(\beta) = \sum_i \tau_\alpha(\beta_i)
\]

While the twisting number will allow us to go from the mixed collar lemma for a curve and an arc (Proposition \ref{prop:Mixed_Collar_Orthogonal_Arcs}) to the mixed collar lemma for pairs of curves (Theorem \ref{thm:Full_Collar}), it turns out we can avoid using it in the statement of the theorem. We will instead use the following inequality:

\begin{lem}
\label{lem:Bound_Twisting_by_Length}
Let $\alpha$ and $\beta$ be two simple closed geodesics on $X$ with $i(\alpha, \beta) \geq 1$. Then 
\[
 \tau_\alpha(\beta) \leq 2\frac{\ell_X(\beta)}{\Col_X(\beta)} + 4i(\alpha, \beta)
\]
\end{lem}
\begin{proof}
We have that $\alpha$ cuts $\beta$ into $k = i(\alpha, \beta)$ strands $\beta_1, \dots, \beta_k$. Let $\beta_i$ be one such strand. In fact, the number of twists of $\beta_i$ about $\alpha$ is bounded. As we show below, each time $\beta_i$ winds around $\alpha$, it gains roughly $\ell_{X}(\alpha)$ of length. As the length of $\beta$ is bounded above, and the length of $\alpha$ is bounded below by $\Col_X(\beta)$, this means $\beta_i$ has bounded twisting number about $\alpha$. Summing up the twisting of each strand, we get our bound.
\begin{figure}[h!]
 \centering 
 \includegraphics{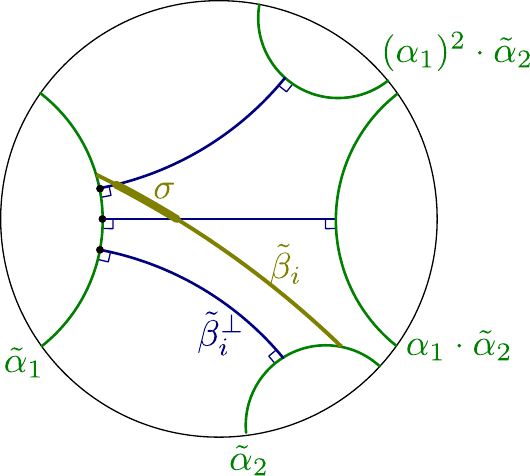}
 \caption{}
 \label{fig:BoundTwisting_of_Marking}
\end{figure}

To see that $\tau_{\alpha}(\beta_i)$ is bounded for each strand $\beta_i$ of $\beta$, take a lift $\tilde \beta_i$ of $\beta_i$ to the universal cover (Figure \ref{fig:BoundTwisting_of_Marking}). Take lifts $\tilde \alpha_1$ and $\tilde \alpha_2$ of $\tilde \alpha$ that pass through either endpoint of $\tilde \beta_i$. Let $\tilde \beta_i^\perp$ be the mutual orthogonal joining $\tilde \alpha_1$ and $\tilde \alpha_2$. Suppose it has endpoint $p$ on $\tilde \alpha_1$. 

Let $\alpha_1$ be the hyperbolic isometry corresponding to $\alpha$, with axis $\tilde \alpha_1$. Then $\tilde \alpha_2$ is disjoint from $\alpha_1 \cdot \tilde \alpha_2$ because they are both lifts of $\alpha$, which is a simple closed curve. Moreover, $\tilde \beta_i^\perp$ is a simple arc, because $\beta_i$ is simple. Thus, $\tilde \beta_i^\perp$ is disjoint from $\alpha_1 \cdot \tilde \beta_i^\perp$. 

Let $\alpha^+$ be the attracting fixed point of $\alpha_1$. Suppose one endpoint of $\tilde \beta_i$ is on the ray out of $\alpha_1\cdot p$ in the direction of $\alpha^+$. As the other endpoint of $\tilde \beta_i$ is on $\tilde \alpha_2$, $\tilde \beta_i$ must intersect $\alpha_1 \cdot \beta_i^\perp$. In fact, let $q$ be the endpoint of $\tilde \beta_i$ on $\tilde \alpha_1$. Suppose
\[
q \in ((\alpha_1)^{n-1} \cdot p, (\alpha_1)^n \cdot p]
\]
Then $\tilde \beta_i$ intersects $(\alpha_1)^i \cdot \tilde \beta_i^\perp$ for all $0< i < n-1$.

Now, for any $i$ with $1 < i < n-1$, let $\sigma$ be the segment of $\tilde \beta_i$ between $(\alpha_1)^{i-1} \cdot \tilde \beta_i^\perp$ and $(\alpha_1)^i \cdot \tilde \beta_i^\perp$. Both of these translates of $\tilde \beta_i^\perp$ are orthogonal to $\tilde \alpha_1$. Moreover, the distance between them is exactly $\ell_{X}(\alpha)$. Thus, the length of $\sigma$ is at least $\ell_{X}(\alpha)$. Therefore, 
\[
 \ell_{X}(\beta_i) \geq (n-2)\ell_{X}(\alpha)
\]
As the same is true for translates of $\tilde \beta_i$ along $\alpha_2$, we have that 
\[
 (\tau_{\alpha}(\beta_i)-4) \ell_{X}(\alpha) \leq 2 \ell_{X}(\beta_i)
\]
Summing over all segments $\beta_1, \dots, \beta_n$ and rearranging, we get 
\[
 \tau_{\alpha}(\beta) \leq 4i(\alpha, \beta)+ \frac{2}{\ell_{X}(\alpha)} \ell_{X}(\beta)
\]
By the collar lemma, we have that $\ell_X(\alpha) > \Col_X(\beta)$, where $\Col_X(\beta)$ is the width of the collar about $\beta$ coming from the collar lemma. (Of course, we can do better than this, but we won't need it for our application.) Thus, we get 
\[
  \tau_{\alpha}(\beta) \leq 4i(\alpha, \beta)+ 2\frac{\ell_{X}(\beta)}{\Col_X(\beta)} 
\]
as desired.
\end{proof}

\subsection{Full mixed collar lemma}
We are now ready to state the mixed collar lemma for pairs of simple closed curves.
\begin{theorem}
\label{thm:Full_Collar}
 Let $\mu$ be a filling current with length minimizer $\pi(\mu)$. Let $\alpha, \beta$ be simple closed curves with $i(\alpha, \beta) \geq 1$. Then there is a constant $D = D(\ell_{\pi(\mu)}(\beta))$ so that 
\[
i(\mu, \alpha) \geq \frac{D}{ i(\alpha, \beta)} i(\mu, \beta)
\]
where $D(x)$ is a continuous function depending only on $\chi(S)$, that is bounded above, and goes to 0 as $x$ goes to infinity.
\end{theorem}
\begin{proof}
 First, we show this result for closed geodesics. So let $\gamma$ be a filling closed geodesic with length minimizer ${\pi(\gamma)}$. Suppose $i(\alpha, \beta) = k$. Then $\alpha$ cuts $\beta$ into arcs $\beta_1, \dots, \beta_k$.

Let $\beta^\perp_1$ be the arc whose endpoints are orthogonal to $\alpha$, that is freely homotopic to $\beta_1$ via a homotopy that keeps its endpoints on $\alpha$. We will first show that 
\[
 i(\gamma,\beta_1) \leq i(\gamma,\beta^\perp_1) + \tau_\alpha(\beta) i(\gamma, \alpha)
\]

See Figure \ref{fig:Intersection_and_twisting} for what follows. Just as in the definition of twisting number, take a lift $\tilde \beta^\perp_1$ of $\beta_1$ whose endpoints $p_1, p_2$ lie on lifts $\tilde \alpha_1$ and $\tilde \alpha_2$ of $\alpha$, respectively. Then take a lift $\tilde \beta_1$ of $\beta_1$ whose endpoints $q_1, q_2$ also lie on $\tilde \alpha_1$ and $\tilde \alpha_2$, respectively. Let $\tilde \gamma$ be the full pre-image of $\gamma$ in the universal cover. Then, 
\[
 i(\gamma, \beta_1) = \# \tilde \gamma \cap \tilde \beta_1 \text{ and }  i(\gamma, \beta^\perp_1) = \# \tilde \gamma \cap \tilde \beta^\perp_1
\]
Moreover, if $\alpha_1, \alpha_2 \in \pi_1(S)$ are the deck transformations acting by translation along $\tilde \alpha_1, \tilde \alpha_2$, respectively, then for each $n \in \Z$,
\[
 \#\tilde \gamma \cap [\alpha_i^n p_i, \alpha_i^{n+1} p_i] = i(\gamma, \alpha)
\]

Suppose 
\[
 q_1 \in (\alpha_1^{m-1} \cdot p_1, \alpha_1^m \cdot p_1] \text{ and } q_2 \in (\alpha_2^{n-1} \cdot p_2, \alpha_2^n \cdot p_2]
\]
If a geodesic in $\tilde \gamma$ intersects $\tilde \beta_1$, then it intersects either $\tilde \beta^\perp_1$, $[p_1, \alpha_1^m \cdot p_1]$, or $[p_2, \alpha_1^n \cdot p_2]$. In particular, 
\[
 \# \tilde \gamma \cap \tilde \beta_1 \leq \# \tilde \gamma \cap [p_1, \alpha_1^m \cdot p_1] + \# \tilde \gamma \cap \tilde \beta^\perp_1 + \tilde \gamma \cap [p_2, \alpha_1^n \cdot p_2]
\]
That is, 
\begin{align*}
 i(\gamma, \beta_1) & \leq |m| i(\gamma, \alpha) + i(\tilde \beta^\perp_1, \gamma) + |n| i(\gamma, \alpha)\\
  & = \tau_\alpha(\beta_1)i(\gamma, \alpha) + i(\beta^\perp_1, \gamma)
\end{align*}

Let $\beta^\perp_2, \dots, \beta^\perp_k$ be the arcs that are orthogonal to $\alpha$ at their endpoints, and homotopic relative to $\alpha$ to $\beta_2, \dots, \beta_k$, respectively. As the same inequality holds for $\beta_2, \dots, \beta_k$, sum over all the arcs we get 
\begin{equation}
\label{eq:i(gamma,beta)}
 i(\gamma, \beta) \leq \tau_\alpha(\beta)i(\gamma, \alpha) + \sum_{i=1}^k i(\beta^\perp_i, \gamma)
\end{equation}
as the twisting number of $\beta$ is the sum of the twisting numbers of its arcs.

We will break the set $\beta^\perp_1, \dots, \beta^\perp_k$ up into orthogonal arc systems for $\alpha$. First, suppose $\beta^\perp_1, \dots, \beta^\perp_l$ each have an endpoint on either side of $\alpha$. Then they form orthogonal arc systems on their own. Now, for $i > l+1$, suppose $\beta_i$ has both endpoints on the same side of $\alpha$. Then note that the set of arcs with endpoints on one side of $\alpha$ are in one-to-one correspondence with the set of arcs that have endpoints on the other side of $\alpha$. This follows from the fact that, for each intersection point of $\beta$ and $\alpha$, there is one arc emerging to the left, and one arc emerging to the right. So suppose
\[
 \{\beta^\perp_1\}, \dots, \{\beta^\perp_l\}, \{\beta^\perp_{l+1}, \beta^\perp_{l + 2}\}, \dots, \{\beta^\perp_{k-1}, \beta^\perp_n\}
\]
is our collection of orthogonal arc systems of arcs for $\alpha$.

Now Proposition \ref{prop:Mixed_Collar_Orthogonal_Arcs} implies that for $i \leq l$,
\[
i(\gamma, \alpha) > \delta(\ell_{\pi(\gamma)}(\beta^\perp_i)) i(\gamma, \beta^\perp_i)
 \]
 and for $i = l+1, l+3, \dots, n-1$, 
 \[
  i(\gamma, \alpha) > \delta(\ell_{\pi(\gamma)}(\beta^\perp_i)) \big(i(\gamma, \beta^\perp_i) + i(\gamma, \beta_{i+1}) \big)
 \]
 where $\delta = \delta(\ell_{\pi(\gamma)}(\beta^\perp_i))$ depends only on $\ell_{\pi(\gamma)}(\beta^\perp_i)$, $\ell_{\pi(\gamma)}(\beta^\perp_{i+1})$ (if $i > l$) and the Euler characteristic $\chi(S)$. Now, $\delta$ is a decreasing function, so $\delta(\ell_{\pi(\gamma)}(\beta^\perp_i)) > \delta(\ell_{\pi(\gamma)}(\beta))$ for all $i$. Using this fact, and summing over all arcs $\beta^\perp_1, \dots, \beta^\perp_k$, we get
 \begin{align*}
  i(\gamma, \alpha)i(\alpha, \beta) & \geq \delta(\ell_{\pi(\gamma)}(\beta)) \sum_{i=1}^k  i(\gamma, \beta^\perp_1) \\
  & \geq \delta(\ell_{\pi(\gamma)}(\beta))(i(\gamma, \beta) - \tau_\alpha(\beta)i(\gamma, \alpha)) 
 \end{align*}
 where the second inequality is by Equation (\ref{eq:i(gamma,beta)}).
 
Rearranging, this gives $i(\gamma, \alpha) \geq \frac{\delta(\ell_{\pi(\gamma)}(\beta))}{i(\alpha, \beta) + \delta(\ell_{\pi(\gamma)}(\beta))\tau_\alpha(\beta)}i(\gamma, \beta)$. For what follows, we will divide through by $\delta = \delta(\ell_{\pi(\gamma)}(\beta))$ to get
\begin{equation}
\label{eq:Full_Collar_for_closed_curves}
 i(\gamma, \alpha) \geq \frac{1}{\frac{1}{\delta} i(\alpha, \beta)+ \tau_\alpha(\beta)}i(\gamma, \beta)
\end{equation}
By Lemma \ref{lem:Bound_Twisting_by_Length}, we have $\tau_\alpha(\beta) \leq 2 \frac{\ell_{\pi(\gamma)}(\beta)}{\Col_{\pi(\gamma)}(\beta)} + 4 i(\alpha, \beta)$. Thus, 
\[
 \frac{1}{\delta} i(\alpha, \beta)+ \tau_\alpha(\beta) \leq (\frac{1}{\delta}+4) i(\alpha, \beta)+2 \frac{\ell_{\pi(\gamma)}(\beta)}{\Col_{\pi(\gamma)}(\beta)}
\]
As $i(\alpha, \beta) \geq 1$, and the other terms on the right-hand side depend only on $\ell_{\pi(\gamma)}(\beta)$, we set
\[
D = D(\ell_{\pi(\gamma)}(\beta)) = \frac{1}{\frac{1}{\delta(\ell_{\pi(\gamma)}(\beta))}+4 + 2 \frac{\ell_{\pi(\gamma)}(\beta)}{\Col_{\pi(\gamma)}(\beta)}}
\]
Then we have
\[
 i(\gamma, \alpha) \geq \frac{D(\ell_{\pi(\gamma)}(\beta))}{ i(\alpha, \beta)} i(\gamma, \beta)
\]
where $D(x)$ is a continuous function that goes to 0 as $x$ goes to infinity. Moreover, it is bounded above by $\frac 14$.

Now take any filling current $\mu$ with length minimizer ${\pi(\mu)}$. Let $\gamma_n$ be a sequence of closed curves so that $c_n \gamma_n \to \mu$. If $\pi(\gamma_n)$ is the length minimizer of $\gamma_n$, then $\pi(\gamma_n) \to \pi(\mu)$. Thus, $\ell_{\pi(\gamma_n)}(\beta)$ converges to $\ell_{\pi(\mu)}(\beta)$, and since $D$ is continuous, we have $D(\ell_{\pi(\gamma_n)}(\beta))$ converges to $D(\ell_{\pi(\mu)}(\beta))$. Thus, by bilinearity of intersection number,
\[
 \lim_{n \to \infty}(c_n \gamma_n, \alpha) \geq \lim_{n \to \infty} \frac{D(\ell_{\pi(\gamma_n)}(\beta))}{ i(\alpha, \beta)} i(C_n\gamma_n, \beta)
\]
and so 
\[
 i(\mu, \alpha) \geq \frac{D(\ell_{\pi(\mu)}(\beta))}{ i(\alpha, \beta)} i(\mu, \beta)
\]
where $D = D(x)$ is a continuous function depending only on $\chi(S)$, as desired.
\end{proof}

\section{Short markings in the thick part}
\label{sec:Mixed_length_comparisons}
If $Y$ is a thick component of some hyperbolic surface $X$, then it has a marking where all the curves are uniformly short. That is, all the curves in the marking are comparable in length to the $Y$-systole of $X$. 

As a consequence of the mixed collar theorem in the previous section, we show that if $X = \pi(\mu)$ for some filling current $\mu$, then $Y$ also has a marking that is uniformly short with respect to $\ell_\mu$. In fact, we can find a marking on $Y$ where the $\mu$-length of each curve is comparable to the $Y$-systolic length of $\mu$. 

We need to show the following lemma, which states that any simple closed curve has an orthogonal arc system of bounded length.
\begin{lem}
  \label{lem:Shortest_arc_length_bound}
Let $\alpha$ be a simple closed geodesic on a closed hyperbolic surface $X$. Then $\alpha$ has an orthogonal arc system $\beta$ with 
\[
 \ell_X(\beta) \prec \Col_X(\alpha) + 1
\]
where the constant depends only on the topology of $X$.
 \end{lem}
 
\begin{proof}
Cut $X$ along $\alpha$ to get a hyperbolic surface with boundary $X_\alpha$. Then $X_\alpha$ has two boundary components: $\alpha^+$ and $\alpha^-$.  Finding an orthogonal arc system for $\alpha$ on $X$ is equivalent to finding either an arc $\beta$ joining $\alpha^+$ to $\alpha^-$, or two arcs, $\beta^+$ and $\beta^-$ where $\beta^+$ has both endpoints on $\alpha^+$, and $\beta^-$ has both endpoints on $\alpha^-$. Of course, we can homotope these arcs relative the boundary of $X_\alpha$ to be orthogonal to the boundary. Our goal is to find such arcs on $X_\alpha$ that are of bounded length.
 
Let $\p$ be the Bers shortest pants decomposition of $X$. When we cut $X$ to get $X_\alpha$, the curves in $P$ get cut into a collection of curves and arcs on $X_\alpha$. Abusing notation, we will still let $\p$ denote this collection of curves and arcs on $X_\alpha$. Thus, for each curve or arc $\gamma \in \p$ we have $\ell_{X_\alpha}(\gamma) \prec 1$, for constants depending only on $\chi(S)$. There are two cases.

Suppose there is an arc $\gamma \in \p$ on $X_\alpha$ that joins $\alpha^+$ to $\alpha^-$. We tighten $\gamma$ to an arc $\beta$ that is orthogonal at its endpoints to $\partial X_\alpha$. Then its length can only decrease. Thus, we have found an orthogonal arc system $\beta$ for $\alpha$ with 
\[
\ell_X(\beta) \prec 1
\]

So suppose there is no $\gamma \in \p$ that joins $\alpha^+$ to $\alpha^-$. Then again there are two cases: either $\alpha$ was part of the Bers pants decomposition of $X$, or it was not. 

Suppose first that $\alpha$ is not part of the Bers pants decomposition of $X$, as this is easier. In this case, there must be arcs $\gamma^+$ and $\gamma^-$ in $\p$ so that $\beta^+$ has both endpoints on $\alpha^+$, and $\beta^-$ has both endpoints on $\alpha^-$. To see this, note that if $\alpha$ is not part of the pants decomposition of $X$, then some curve $\gamma$ must intersect $\alpha$. Cutting $\gamma$ along $\alpha$ gives a collection of arcs on $X_\alpha$. But we assume that no arc joints $\alpha^+$ to $\alpha^-$. So all arcs must either join $\alpha^+$ to itself, or $\alpha^-$ to itself. Since the number of endpoints on $\alpha^+$ must equal the number of endpoints on $\alpha^-$, we find the arcs $\gamma^+$ and $\gamma^-$ we are looking for.

Again, since $\gamma^+$ and $\gamma^-$ are in $\p$, we have $\ell_{X_\alpha}(\gamma^+), \ell_{X_\alpha}(\gamma^+) \prec 1$. Tightening these to arcs $\beta^+$ and $\beta^-$ that are orthogonal to $\alpha^+$ and $\alpha^-$ at their endpoints, respectively, gives us the orthogonal arc system $\beta = \{\beta^+, \beta^-\}$ where 
\[
 \ell_X(\beta^+),\ell_X(\beta^-) \prec 1
\]

Lastly, suppose $\alpha$ is part of the Bers short pants decomposition of $X$. Then there is a pair of pants $P$ on $X_\alpha$ that has $\alpha^+$ as a boundary component. In this case, we will find an arc $\beta^+$ inside $P$, that is orthogonal to $\alpha^+$ at both of its endpoints, so that $\ell_{X_\alpha}(\beta^+) \prec \Col_X(\alpha) + 1$.

We show this as follows. We can decompose $P$ into two congruent right-angled hexagons. Let $h$ be one such hexagon. Then $\beta^+$ cuts $h$ into two right-angled pentagons (Figure \ref{fig:Pentagon_including_alpha}). Each pentagon has a boundary component that lies on $\alpha^+$. Choose the pentagon $Q$ so that the side of $Q$ that lies on $\alpha^+$ has length  $a \geq \ell(\alpha^+)/4$. 
  \begin{figure}[h!]
   \centering 
   \includegraphics{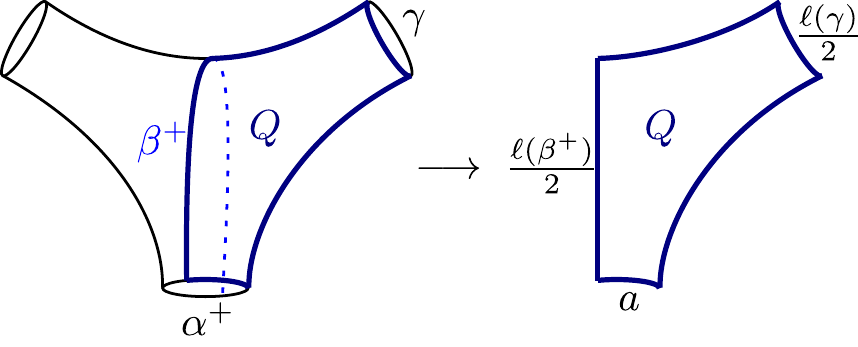}
   \caption{The pentagon $Q$ in the pair of pants $P$. The side lengths are labeled on the right.}
   \label{fig:Pentagon_including_alpha}
  \end{figure}

 So $Q$ has a boundary edge that lies on $\alpha^+$, and another boundary edge that lies on some boundary curve $\gamma$ of $P$. By, for example, \cite[Theorem 2.3.4]{Buser}, 
  \[
   \sinh(\ell(\beta^+)/2) = \frac{\cosh(\ell(\gamma)/2)}{\sinh(a)}
  \]

  Using the identity $\sinh(2a) = 2 \sinh(a)\cosh(a)$, we get 
  \[
   \sinh(\ell(\beta^+)/2) = \frac{2\cosh(a)\cosh(\ell(\gamma)/2)}{\sinh(2a)}
  \]
  Now, $\sinh^{-1}(x) = \log(\sqrt{1+x^2} + x)$, and so $\sinh^{-1}(cx) \leq \log c + \sinh^{-1}(x)$ for $c > 1$. Let $c =2 \cosh(a)\cosh(\ell(\gamma)/2)$ be the numerator above. Since $\cosh(x) > 1$ for all $x$, we have $c >  1$, so
  \begin{align*}
   \ell(\beta^+)& \leq 2\log\left( 2 \cosh(a)\cosh(\ell(\gamma)/2) \right) + 2\sinh^{-1}\left( \frac{1}{\sinh(2a)} \right )\\
   & \leq 2\log\big( 2 \cosh(a)\cosh(\ell(\gamma)/2) \big) + 2 \Col_X(\alpha)
  \end{align*}
  where we use that $a \geq \ell(\alpha^+)/4$ to conclude that $\sinh^{-1}\left (\frac{1}{\sinh(2a)}\right ) \leq \Col_X(\alpha)$.
  
  Since $\alpha$ and $\gamma$ belong to the shortest pants decomposition of $X$, the $\log$ term is bounded above in terms of $\chi(S)$. Thus,
  \[
   \ell(\beta^+) < C + 2 \Col_X(\alpha)
  \]
  where $C$ depends only on $\chi(S)$. Analogously, there is a pair of pants in $X_\alpha$ containing $\alpha^-$, so the above proof gives us an arc $\beta^-$ with both endpoints orthogonal to $\alpha^-$, for which $ \ell(\beta^-) < C + 2 \Col_X(\alpha)$. So we again get the orthogonal arc system for $\alpha$ we are looking for.
\end{proof}

We're now ready to show that a thick component of a length minimizer has a marking that is not very long with respect to $\mu$.
\begin{lem}
\label{lem:mu-short_marking}
 Let $Y$ be a thick component of $\pi(\mu)$, for $\mu \in \C_{fill}(S)$. Then there exists a marking $\M$ of $Y$ so that
 \[
  i(\mu, \gamma_i) \asymp \sys_Y(\mu)
 \]
for all $\gamma_i \in \M$.
\end{lem}
We call $\M$ a $\mu$-short marking of $Y$.
\begin{proof} 
 We can build our marking of $Y$ as follows.  As $\mu$ is filling, there exists an essential, non-peripheral simple closed curve $\gamma_0$ in $Y$ so that 
\[
 i(\gamma_0, \mu) = \sys_Y(\mu)
\]
This follows from the fact that, for each $L$, there are finitely many closed curves $\gamma$ with $i(\mu, \gamma) < L$ (by, for example, \cite{Glorieux}). We say that $\gamma_0$ is a \textbf{$Y$-systole} of $\mu$. 
 
 So let $\gamma_0$ be the $Y$-systole of $\mu$. Suppose we have essential, non-peripheral simple closed curves $\gamma_0, \gamma_1, \dots, \gamma_{2k}$, so that for each $i$, and each $j \neq 2i + 1$, 
\[
i(\gamma_{2i}, \gamma_j) = 0
\]
and $\gamma_{2i}$ and $\gamma_{2i+1}$ intersect minimally. In other words, the set $\{\gamma_{2i}\}_{i = 0,\dots, k}$ will eventually be part of the pants decomposition of $Y$, and $\gamma_{2i+1}$ is the transverse curve to $\gamma_{2i}$. Suppose for each $i$ we have $i(\mu,\gamma_i) \asymp \sys_Y(\mu)$. Then we form curves $\gamma_{2k+1}$ and $\gamma_{2k+2}$ as follows.

By Lemma \ref{lem:Shortest_arc_length_bound}, we can find an orthogonal arc system $\beta = \{\beta^+, \beta^-\}$ for $\gamma_{2k}$, with $\ell_{\pi(\mu)}(\beta) \prec \Col_{\pi(\mu)}(\gamma_{2k}) + 1$. But as $\gamma_{2k}$ lies in a thick component of $\pi(\mu)$, its collar width is also bounded above in terms of $\chi(S)$. Thus, 
 \[
  \ell_{\pi(\mu)}(\beta) \prec 1
 \]
for a constant depending only on $\chi(S)$. As the function $\delta$ from Proposition \ref{prop:Mixed_Collar_Orthogonal_Arcs} is decreasing, this implies that 
\[
 \delta(\ell_{\pi(\mu)}(\beta)) \succ 1
\]
So, by the mixed collar lemma for a curve and its orthogonal arc system (Proposition \ref{prop:Mixed_Collar_Orthogonal_Arcs}), $i(\beta, \mu)  \prec i(\gamma_{2k}, \mu)$ where the constant depends only on $\chi(S)$. Thus, 
\[
i(\beta, \mu) \asymp \sys_Y(\mu)
\]
by our assumptions on $\gamma_{2k}$, and the definition of systolic length.

It is possible that $\beta^+$ or $\beta^-$ intersect one of $\gamma_0, \dots, \gamma_{2k-1}$. In this case, we replace $\beta$ with a new collection $\beta'$ that does not, so that we still have $i(\beta, \mu) \asymp \sys_Y(\mu)$ (Figure \ref{fig:New_orth_arc_system}).

\begin{figure}[h!]
 \centering 
 \includegraphics{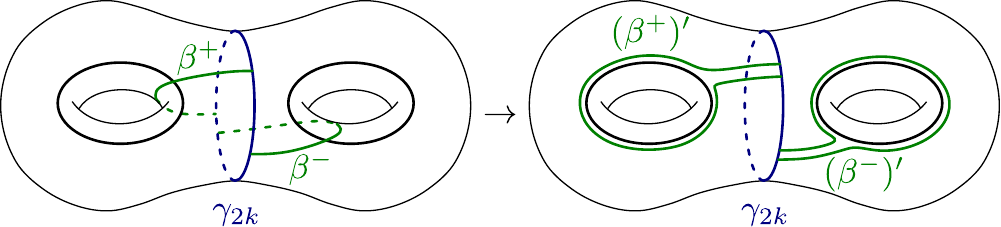}
 \caption{If $\beta^+$ or $\beta^-$ intersect any of $\gamma_0, \dots, \gamma_{2k-1}$, we do surgery on them to remove the intersection.}
 \label{fig:New_orth_arc_system}
\end{figure}

Since $\beta$ is an orthogonal arc system for $\gamma_{2k}$, there are endpoints $p, q$ on $\gamma_{2k}$ so that $\beta^+$ emerges from $\gamma_{2k}$ to the right out of $p$, and one of $\beta^+$ or $\beta^-$ emerge from $\gamma_{2k}$ to the left out of $q$. To form $(\beta^+)'$, we follow $\beta^+$ out of $p$. If we never intersect any of $\gamma_0, \dots, \gamma_{2k-1}$, then $(\beta^+)' = \beta^+$. Suppose we hit $\gamma_i$ for some $i < 2k$ at some point $x$. Then we follow $\gamma_i$ around back to $x$, and follow $\beta^+$ backwards from $x$ to $p$ to make $(\beta^+)'$. By the same argument as in \cite[Lemma 4.4]{MZ}, we have 
\begin{align*}
 i(\mu, (\beta^+)') & \leq i(\mu, \beta^+) + i(\mu, \gamma_i) \\
  & \prec \sys_Y(\mu)
\end{align*}
which is what we wanted. 

If $(\beta^+)'$ has both endpoints on the same side of $\gamma_{2k}$, then we do the same procedure starting at the point $q$ to make $(\beta^-)'$. Note that $(\beta^-)'$ might be made from segments of $\beta^+$. This gives us a new orthogonal arc system $\beta'$ with $i(\mu, \beta') \asymp \sys_Y(\mu)$. Moreover,  $(\beta^-)'$ must be disjoint from $(\beta^+)'$, as the two arcs are either the original arcs in $\beta$, or they are in different components of $S \setminus (\gamma_0 \cup \dots \cup \gamma_{2k})$.

So without loss of generality, we can assume that $\beta^+$ and $\beta^-$ are disjoint from $\gamma_0, \dots, \gamma_{2k-1}$. We can then use these arcs to build a curve $\gamma_{2k+1}$ that is transverse to $\gamma_{2k}$ and disjoint from $\gamma_0, \dots, \gamma_{2k-1}$ as follows.

There are two cases. First, suppose $\beta^+$ has an endpoints $p_1$ and $p_2$ on either side of $\gamma_{2k}$. Let $\sigma$ be a subarc of $\gamma_{2k}$ between $p_1$ and $p_2$.  Let $\gamma_{2k+1}$ be the geodesic in the free homotopy class of the concatenation $\beta \circ \sigma$. Again, by \cite[Lemma 4.4]{MZ},
 \begin{align*}
  i(\mu, \gamma_{2k+1}) & \leq i(\mu, \beta_1) + i(\mu, \sigma)\\
   & \prec \sys_Y(\mu)
 \end{align*}
 
 \begin{figure}[h!]
  \centering 
  \includegraphics{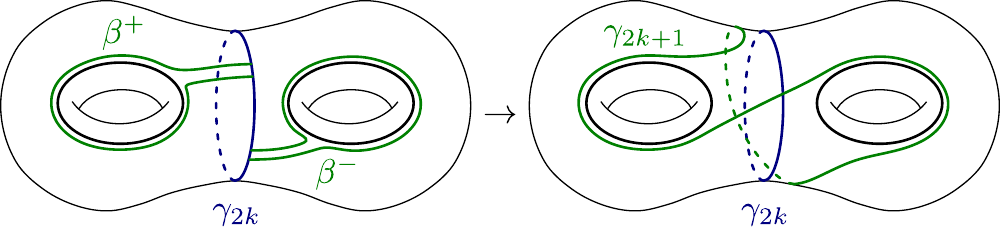}
  \caption{The case where the arc system has two arcs.}
  \label{fig:Arc_sys_to_closed_loop}
 \end{figure}

  Now, suppose $\beta^+$ has both endpoints, $p_1$ and $p_2$, on one side of $\gamma_{2k}$, and $\beta^-$ has both endpoints, $q_1$ and $q_2$, on the other side of $\gamma_{2k}$ (Figure \ref{fig:Arc_sys_to_closed_loop}). Up to relabeling the endpoints of $\beta^+$ and $\beta^-$, we can find disjoint subarcs $\sigma_1$ and $\sigma_2$ of $\gamma_{2k}$, so that $\sigma_1$ lies between $p_1$ and $q_1$, and $\sigma_2$ lies between $p_2$ and $q_2$. This is because $p_1, q_1$ and $p_2$ cut $\gamma_{2k}$ into three pieces, and $q_2$ can only lie in one of those pieces. Then we let $\gamma_{2k+1}$ be the closed geodesic in the free homotopy class of $\beta^+ \circ \sigma_1 \circ \beta^- \circ \sigma_2$. So we have 
  \begin{align*}
   i(\mu, \gamma_{2k+1}) & \leq i(\mu, \beta^+) + i(\mu, \sigma_1) + i(\mu, \beta^-) + i(\mu, \sigma_2)\\
   & \prec \sys_Y(\mu)
  \end{align*}

Now take the curves $\gamma_0, \dots, \gamma_{2k+1}$. These fill some subsurface $Y'$ of $Y$. If $Y'$ is homeomorphic to $Y$, then we are done. Otherwise, let $\gamma_{2k+2}$ be a boundary component of $Y'$ that is non-peripheral in $Y$. Since every curve in $\gamma_0, \dots, \gamma_{2k+1}$ intersects some other curve on this list, $\gamma_{2k+2}$ is not one of $\gamma_0, \dots, \gamma_{2k+1}$. We can homotope $\gamma_{2k+2}$ to a concatenation of subarcs of $\gamma_0, \dots, \gamma_{2k+1}$, that passes through each point on each subarc at most twice. Thus, 
\[
 i(\mu, \gamma_{2k+2}) \leq 2\sum_{i=0}^{2k+1} i(\mu, \gamma_i)
\]
As $k$ is bounded above in terms of $\chi(S)$, and $i(\mu, \gamma_i) \prec \sys_Y(\mu)$ by assumption, we again have 
\[
 i(\mu, \gamma_{2k+2}) \prec \sys_Y(\mu)
\]
Continuing on in this way, we get the marking of $Y$, as desired. (Note that $i(\mu, \gamma_i) \succ \sys_Y(\mu)$ by definition of the $Y$-systolic length of $\mu$.)
\end{proof}

\section{Another collar theorem}
\label{sec:Another_collar}
The usual collar lemma states that if $\alpha$ is any simple closed curve on a hyperbolic surface $X$, then $i(\mu, \alpha) \Col_X(\alpha) \leq \ell_X(\mu)$. Note that when $\mu$ is a closed curve, the same is true if $\alpha$ lies in a subsurface $Y$ of $X$, and we replace the length of $\mu$ in $X$ with just the length of $\mu$ restricted to $Y$. The following theorem is a mixed collar lemma in the case where $X = \pi(\mu)$. It replaces the length of $\mu$ in $Y$ with $i(\mu, \M)$ where $\M$ is any marking of $Y$. One should think of $i(\mu, \M)$ as a different way of measuring the length of $\mu$ restricted to $Y$.

\begin{theorem}
\label{thm:Another_Collar}
 Let $\mu \in \C_{fill}(S)$, and let $Y$ be a subsurface of $\pi(\mu)$ with $\ell_{\pi(\mu)}(\alpha) \leq c_b$ for each $\alpha$ in $\partial Y$, and $c_b$ the Bers constant. Then for each essential simple closed curve $\beta$ in $Y$ and marking $\M$ of $Y$,
 \[
  i(\mu, \beta) \Col_{\pi(\mu)}(\beta) \prec i(\mu, \M)
 \]
where the constant depends only on the topology of $S$.
\end{theorem}
This theorem holds for any subsurface $Y$ of $\pi(\mu)$ with hyperbolically short boundary components, which is more general than the case where $Y$ is a thick component. Moreover, it holds for any marking $\M$ of $Y$. In the following section, though, we will apply it to the case where $Y$ is a thick component and $\M$ is the $\mu$-short marking found in the previous section.
\begin{proof}
 We will prove this theorem assuming Lemmas \ref{lem:comparison_on_Yc} and \ref{lem:comparison_on_Yth} below. We will then prove these lemmas in Sections \ref{subsec:comparison_on_Yc} and \ref{subsec:comparison_on_Yth}, respectively.

For now, let $\gamma$ be a filling closed curve, rather than an arbitrary filling current. Let $\pi(\gamma)$ be its length minimizing metric, and let $Y$ be subsurface of $\pi(\gamma)$ so that $\ell_{\pi(\gamma)}(\alpha) < c_b$ for each boundary component $\alpha$ of $Y$. For each curve $\alpha$ on the boundary of $Y$, let $C(\alpha)$ be the regular neighborhood of $\alpha$ in $Y$ of width $\Col_{\pi(\gamma)}(\alpha)$, for 
\[
 \Col_{\pi(\gamma)}(\alpha) = \sinh^{-1}\left(\frac{1}{\sinh(\ell_{\pi(\gamma)}(\alpha)/2)} \right )
\]
the constant coming from the collar lemma. Let $C$ be the union of all such collar neighborhoods. We then let $Y^{bd} = Y \setminus C$. Note that if $Y$ were a thick component of $\pi(\gamma)$, then the diameter of $Y^{bd}$ would be bounded above by a universal constant depending only on $\chi(S)$. 

For any subset $Z$ of a hyperbolic surface $X$, we define $\ell_X(\gamma \cap Z)$ to be the length of $X$-geodesic representative of $\gamma$ intersected with $Z$, where the length is taken with respect to $X$. 

Let $\M$ be any marking of $Y$. We will first show
\begin{equation}
\label{eq:Bdd_length}
  \ell_{\pi(\gamma)}(\gamma \cap Y^{bd}) \prec i(\gamma, \M)
\end{equation}

To show this, we will construct a new metric $X$ that is isometric to $\pi(\gamma)$ outside of $Y^{bd}$, where each curve in $\M$ is of bounded length. We will then bound the length of $\gamma$ in $X$, and use that bound to estimate the length of $\gamma$ in $Y^{bd}$.
 
We build the metric $X$ on $S$ as follows: Let $\delta_1, \dots, \delta_k$ be the pants decomposition of $Y$ coming from the marking $\M$. Extend this to a pants decomposition 
\[
\p = \{\delta_1, \dots, \delta_k, \alpha_1, \dots, \alpha_n\}
\]
of $S$ that contains $\partial Y$. We define $X$ using Fenchel-Nielsen coordinates relative to $\p$. Choose the length and twist coordinates of $\alpha_1, \dots, \alpha_n$ to be the same as those of $\pi(\gamma)$. Next, choose the length coordinates of $\delta_1, \dots, \delta_k$ to all be 1, and their twist coordinates to all be 0. With this, we will have $\ell_X(\delta_i) \asymp 1$ for all $\delta_i \in \M$. That is, both the pants curves and the transverse curves in $\M$ will have uniformly bounded length in $X$.

Let $Y^c = \pi(\gamma) \setminus Y^{bd}$. Choose a homeomorphism $f: \pi(\gamma) \to X$, respecting the markings on $\pi(\gamma)$ and $X$, which is an isometry on $Y^c$. This is possible because the Fenchel-Nielsen coordinates on $\alpha_1, \dots, \alpha_n$ are the same for both $\pi(\gamma)$ and $X$, where $\alpha_1, \dots, \alpha_n$ includes the boundary curves of $Y$. Then on $X$, let $Z^{bd} = f(Y^{bd})$, and $Z^c = f(Y^c)$. This gives us the following decompositions of $\pi(\gamma)$ and $X$, where the isometric components are marked:
\[
\begin{matrix}
 \pi(\gamma) & = &Y^c                                   & \cup & Y^{bd}\\
            &    & \rotatebox[origin = c]{270}{$\cong$} &      &   \\
 X          & =  & Z^c                                  & \cup  & Z^{bd}
 \end{matrix}
\]

Since $\pi(\gamma)$ is the length minimizer of $\gamma$, we have 
\[
\ell_{\pi(\gamma)}(\gamma) < \ell_X(\gamma)
\]
Decomposing these two lengths gives the inequality 
\begin{align*}
 \ell_{\pi(\gamma)}(\gamma \cap Y^c) + \ell_{\pi(\gamma)}(\gamma \cap Y^{bd})  < \ell_X(\gamma \cap Z^c) + \ell_X(\gamma \cap Z^{bd}) 
\end{align*}
Now, by Lemma \ref{lem:comparison_on_Yc}, we have $  \ell_X(\gamma \cap Z^c)-\ell_{\pi(\gamma)}(\gamma \cap Y^c)   \prec i(\gamma, \M)$. Thus,
\[
 \ell_{\pi(\gamma)}(\gamma \cap Y^{bd})- \ell_X(\gamma \cap Z^{bd}) \prec i(\gamma, \M)
\]
But by Lemma \ref{lem:comparison_on_Yth}, we also have $ \ell_X(\gamma \cap Z^{bd}) \prec i(\gamma, \M)$. Therefore, 
\[
 \ell_{\pi(\gamma)}(\gamma \cap Y^{bd}) \prec i(\gamma, \M)
\]
where all constants depend only on $\chi(S)$. This proves Equation (\ref{eq:Bdd_length}).

Let $\beta$ be any non-peripheral, simple closed curve on $Y$. By the collar lemma, collars about disjoint simple closed curves are disjoint. As $Y^{bd}$ is formed by removing collars about the boundary curves of $Y$, we have that the $\Col_{\pi(\gamma)}(\beta)$-collar about $\beta$ lies entirely in $Y^{bd}$. Therefore, 
 \[
  i(\gamma, \beta) \Col_{\pi(\gamma)}(\beta) \leq \ell_{\pi(\gamma)}(\gamma \cap Y^{bd})
 \] 
Thus, by Equation \ref{eq:Bdd_length},
 \[
i(\gamma, \beta) \Col_{\pi(\gamma)}(\beta) \prec i(\gamma, \M)
 \]

Now, suppose $\mu$ is an arbitrary filling current. Then there is a sequence of filling closed geodesics $\gamma_n$, and constants $c_n > 0$ so that $c_n \gamma_n \to \mu$. The length minimizing projection $\pi$ is continuous and scale invariant. Thus,
 \[
  \lim_{n \to \infty} \pi(\gamma_n) = \pi(\mu)
 \]
By continuity, if a subsurface $Y$ of $S$ has boundary components shorter than the Bers constant in $\pi(\mu)$, then for all $n$ large enough, the same is true in $\pi(\gamma_n)$. Thus, for any essential, non-peripheral simple closed curve $\beta$ in $Y$,
 \[
i(c_n \gamma_n, \beta) \Col_{\pi(\gamma_n)}(\beta) \prec i(c_n\gamma, \M)
 \]
where we can scale both sides by $c_n$ as intersection number is linear. Taking limits, we get 
\[
 i(\mu, \beta) \Col_{\pi(\mu)}(\beta) \prec i(\mu, \M)
\]
as required.
\end{proof}

\subsection{Length outside of $Y^{bd}$}
\label{subsec:comparison_on_Yc}
We prove the first of the two lemmas needed for Theorem \ref{thm:Another_Collar}. Using the notation in the proof above, we show:
\begin{lem}
\label{lem:comparison_on_Yc}
On the complement of $Y^{bd}$, 
\[
 \ell_X(\gamma \cap Z^c) - \ell_{\pi(\gamma)}(\gamma \cap Y^c)\prec i(\gamma, \M)
\]
where the constant depends only on $\chi(S)$.
\end{lem}
Recall that $\ell_X(\gamma \cap Z^c)$ is the length of the $X$-geodesic representative of $\gamma$ intersected with $Z^c$, where the length is taken with respect to $X$. Similarly, $\ell_{\pi(\gamma)}(\gamma \cap Y^c)$ is the length of the $\pi(\gamma)$-geodesic representative of $\gamma$ intersected with $Y^c$, where the length is taken with respect to $\pi(\gamma)$.
\begin{proof}
 Let $\alpha$ be a curve in $\partial Z$ on $X$. Identify the universal covers of $X$ and $\pi(\gamma)$ with the upper half plane $\h \subset \mathbb C$, and take covering maps 
 \[
  \begin{array}{lll}
    \h & & \h\\
      \Big\downarrow p_1 & \text{ and } & \Big\downarrow p_2 \\
   X & & \pi(\gamma) 
  \end{array}
 \]
Suppose both covering maps send the imaginary axis to the geodesic representative of $\alpha$ (Figure \ref{fig:Complement_length_subsurface_lift}). We will thus refer to the imaginary axis as $\tilde \alpha$ from now on. As $\alpha$ is contained in $Z^c$, take the connected component $\tilde Z^c$ of $p_1^{-1}(Z^c)$ that contains $\tilde \alpha$. Since $Z^c$ is isometric to $Y^c$, we arrange it so that $\tilde Z^c$ is also a connected component of $p_2^{-1}(Y^c)$.

\begin{figure}[h!]
 \centering 
 \includegraphics{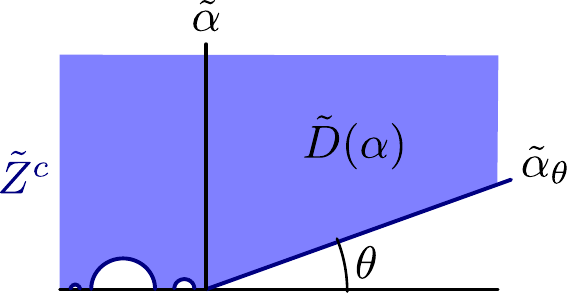}
 \caption{The shaded region is $\tilde Z^c$, and the region between $\tilde \alpha$ and $\tilde \alpha_\theta$ is $\tilde D(\alpha)$.}
 \label{fig:Complement_length_subsurface_lift}
\end{figure}

 The boundary of $\tilde Z^c$ contains a ray $\tilde \alpha_\theta$ coming out of the origin, where each point is at distance $\Col(\alpha)$ from $\tilde \alpha$. (We write $\Col(\alpha)$ to denote the collar width of $\alpha$ in both $\pi(\gamma)$ and $X$.) The ray $\tilde \alpha_\theta$ makes an angle $\theta$ with the real axis. We can compute that 
 \[
  \tan \frac \theta 2 = e^{-\Col(\alpha)}
 \]
 (See Lemma \ref{lem:Hyperbolic_geometry} in the Appendix.)
 
 If $D(\alpha)$ is the collar neighborhood of $\alpha$ in $X$, then it lifts to a set $\tilde D(\alpha)$ bounded on one side by $\tilde \alpha$, and on the other side by the curve $\tilde \alpha_\theta$.  Note that $\tilde D(\alpha)$ is also a lift of the collar neighborhood $C(\alpha)$ of $\alpha$ in $\pi(\gamma)$.
 
 Take a homeomorphism 
 \[
 g: X \to  \pi(\gamma)
 \]
 induced by the markings on the two points in $\T(S)$. Then $g$ induces a homeomorphism $\partial g: \partial \h \to \partial \h$. Identifying the set of geodesics in $\h$ with the set of unordered pairs of distinct points on $\partial \h$, we see that $\partial g$ also acts on the set of geodesics in $\h$. Choose the lift $\partial g$ so that it fixes $\tilde \alpha$. 
 
 Let $\tilde \beta$ be the geodesic in $\h$ with endpoints -1 and 1.  As $\alpha$ has the same length in both $X$ and $\pi(\gamma)$, it acts on $\h$ by the isometry $z \mapsto e^{\ell(\alpha)}z$ in both covers. Consider the set 
 \[
  \{\alpha^n \cdot \tilde \beta\}_{n \in \Z} = \{e^{n \ell(\alpha)} \tilde \beta\}_{n \in \Z}
 \]
 \begin{figure}[h!]
  \centering 
  \includegraphics{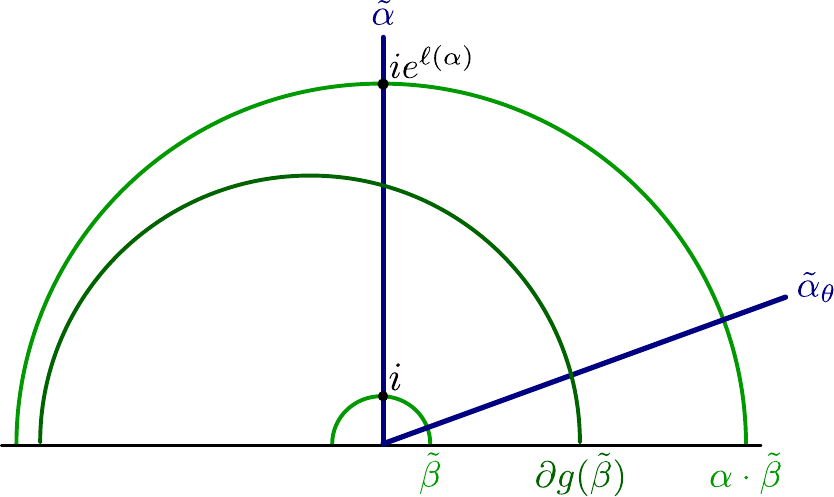}
  \caption{}
  \label{fig:Complement_length_beta_curves}
 \end{figure}

 Since the length and twist parameters of $\alpha$ are the same in both $\pi(X)$ and $X$, we have that $\partial g(\tilde \beta)$ is disjoint from $\alpha^n \tilde \beta$ for all $n$. Without loss of generality, suppose we chose $\partial g$ so that $\partial g(\tilde \beta)$ lies between $\tilde \beta$ and $\alpha \cdot \tilde \beta$ (Figure \ref{fig:Complement_length_beta_curves}). As $\partial g$ preserves the order of points on $\R$, this means that, for all $p>0$ with $e^{n \ell(\alpha)} \leq p < e^{(n+1) \ell(\alpha)}$ for some $n \in \Z$, we have $e^{n \ell(\alpha)} \leq \partial g(p) < e^{(n+2) \ell(\alpha)}$, and similarly for $p < 0$.
 
 \begin{figure}[h!]
  \centering 
  \includegraphics{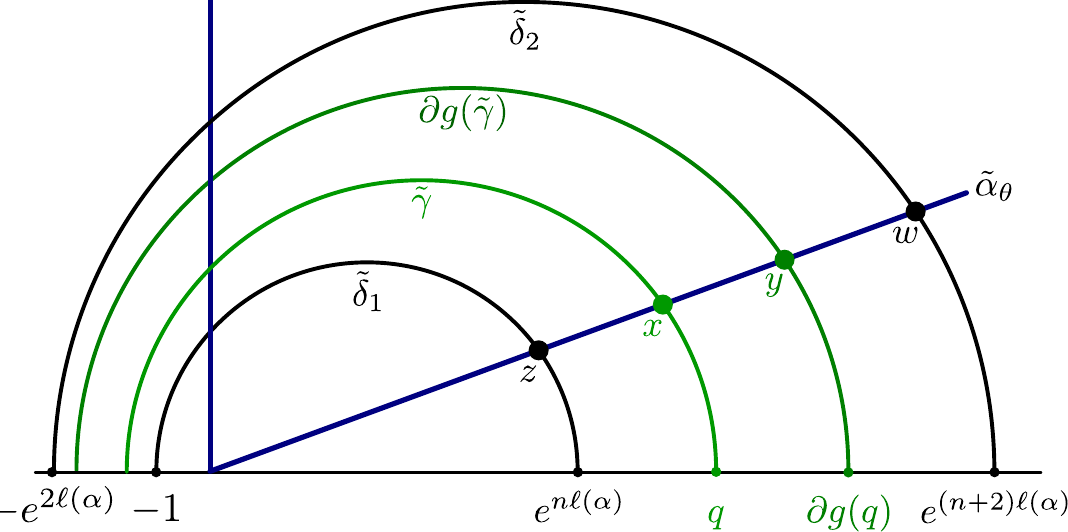}
  \caption{}
  \label{fig:Complement_length_gamma_crosses_alpha}
 \end{figure}

Let $\tilde \gamma$ be a lift of an $X$-geodesic representative of $\gamma$ to $\h$ that passes through $\tilde D(\alpha)$. Suppose first that $\tilde \gamma$ intersects $\tilde \alpha$ (Figure \ref{fig:Complement_length_gamma_crosses_alpha}). Then since $\partial g$ preserves intersections between geodesics, the geodesic $\partial g(\tilde \gamma)$ must also intersect $\tilde \alpha$. Then both $\tilde \gamma$ and $\partial g(\tilde \gamma)$ intersect $\tilde \alpha_\theta$ in exactly one point. Suppose $\tilde \gamma$ intersects $\tilde \alpha_\theta$ at a point $x$, and $\partial g(\tilde \gamma)$ intersects $\tilde \alpha_\theta$ at a point $y$. We will show that 
\[
d(x,y) \prec 1
\]
for a constant depending only on $\chi(S)$.

The action of $\alpha$ is an isometry that preserves $\tilde \alpha_\theta$. So we can act on $\tilde \gamma$ and $\partial g(\tilde \gamma)$ by $\alpha$ without changing the distance between $x$ and $y$. So up to composing with some power of $\alpha$, the endpoints of $\tilde \gamma$ on $\R$ are $p$ and $q$ with 
\[
 - e^{\ell(\alpha)} < p \leq -1, \text{ and } \ell^{n \ell(\alpha)} \leq q < \ell^{(n+1) \ell(\alpha)}
\]
Because $\partial g(\tilde \beta)$ lies between $\tilde \beta$ and $\alpha \cdot \tilde \beta$, we must have that 
\[
  - e^{2\ell(\alpha)} < g(p) \leq -1, \text{ and } \ell^{n \ell(\alpha)} \leq g(q) < \ell^{(n+2) \ell(\alpha)}
\]

Let $\tilde \delta_1$ be the geodesic with endpoints $-1$ and $\ell^{n \ell(\alpha)}$, and $\tilde \delta_2$ be the geodesic with endpoints $- e^{2\ell(\alpha)}$ and $\ell^{(n+2) \ell(\alpha)}$. Then both $\tilde \gamma$ and $\partial g(\tilde \gamma)$ lie between $\tilde \delta_1$ and $\tilde \delta_2$. Note that $\alpha^2 \cdot \tilde \delta_1 = \tilde \delta_2$. Again, the action of $\alpha$ preserves $\tilde \alpha_\theta$. So if $\tilde \delta_1$ intersects $\tilde \alpha_\theta$ in a point $z$ and $\tilde \delta_2$ intersects $\tilde \alpha_\theta$ in $w$, then $w= \alpha^2(z)$. That is, $w = e^{2 \ell(\alpha)} z$. From this, we can compute the length of the segment of $\tilde \alpha_\theta$ between $z$ and $w$, and conclude that
\[
  d(x, y) \leq 2 \ell(\alpha) \csc(\theta)
\]
We can compute that 
\[
 \sin(\theta) = \tanh\left (\frac{\ell(\alpha)}{2}\right)
\]
(see Lemma \ref{lem:Hyperbolic_geometry}) and so,
\[
  d(x, y) \leq 2 \frac{\ell(\alpha)}{\tanh\left (\frac{\ell(\alpha)}{2}\right)}
\]

Note that $x/ \tanh(x/2)$ is an increasing function, and it tends to 4 as $x$ goes to 0. So since $\alpha$ is a Bers-short curve, we see that 
\begin{equation}
\label{eq:dist_on_alpha_theta_for_intersecting}
 d(x, y) \prec 1
\end{equation}
for a constant depending only on $\chi(S)$.

We have that $\tilde \gamma$ enters and exits $\tilde Z^c$ through lifts of collar neighborhoods about curves in $\partial Z$. In fact, if $\tilde \gamma$ crosses $\tilde \alpha$ after it enters $\tilde Z^c$, then it must exit after intersecting another lift of a boundary curve of $Z$. Thus, Equation (\ref{eq:dist_on_alpha_theta_for_intersecting}) implies that 
\begin{equation}
\label{eq:Length_change_for_crossing}
 \ell(\tilde \gamma \cap \tilde Z^c) -  \ell(\partial g(\tilde \gamma) \cap \tilde Z^c) \prec 1
\end{equation}
where the constant is twice that of Equation (\ref{eq:dist_on_alpha_theta_for_intersecting}).

\begin{figure}[h!]
 \centering 
 \includegraphics{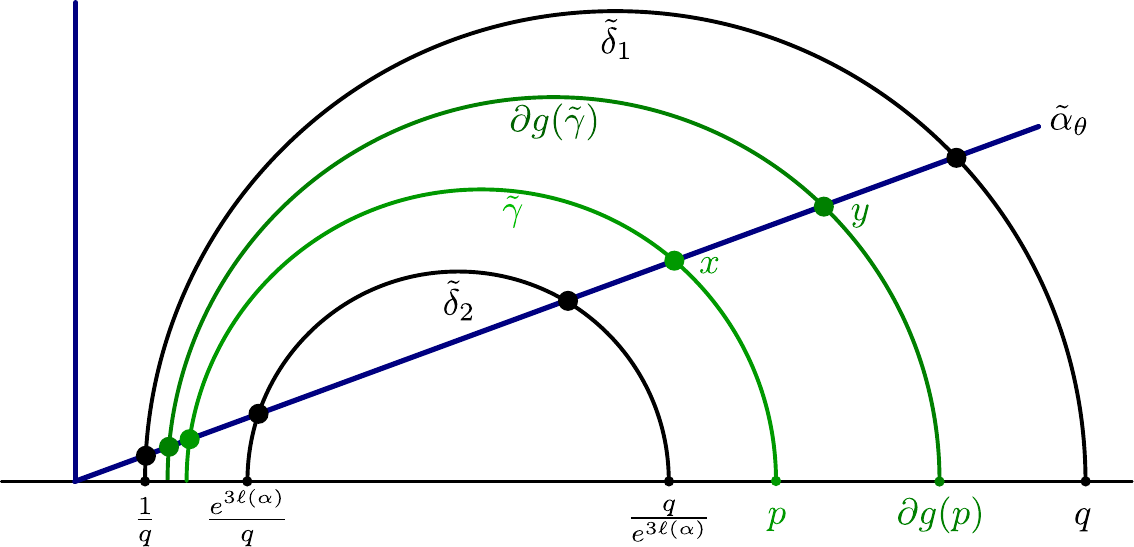}
 \caption{}
 \label{fig:Complement_length_gamma_disjoint_alpha}
\end{figure}

Now, suppose that $\tilde \gamma$ passes through $\tilde D(\alpha)$, but does not intersect $\tilde \alpha$ (Figure \ref{fig:Complement_length_gamma_disjoint_alpha}). Then $\tilde \gamma$ intersects $\tilde \alpha_\theta$ in two points: $x$ and $x'$. Label these points so that $|x'| < |x|$ in the Euclidean metric. The map $z \mapsto az$ is an isometry for all $a \in \R$ that fixes both $\tilde \alpha$ and $\tilde \alpha_\theta$. Thus, $\ell(\tilde \gamma \cap \tilde D(\alpha)) = \ell(a \tilde \gamma \cap \tilde D(\alpha))$ for all $a \in \R_{>0}$. So we can assume without loss of generality that the endpoints of $\tilde \gamma$ on $\R$ are $p$ and $1/p$ for some $p > 1$. Then there is some $n$ so that 
\[
 e^{n \ell(\alpha)} \leq p < e^{(n+1)\ell(\alpha)}
\]
for some $n \in \Z$. Thus, we have
\[
 e^{n \ell(\alpha)} \leq \partial g(p) < e^{(n+2)\ell(\alpha)}
\]
This time, let $\tilde \delta_1$ be the geodesic with endpoints 
\[q = e^{(n+2)\ell(\alpha)} \text{ and } \frac 1 q = e^{-(n+2)\ell(\alpha)}\]
and $\tilde \delta_2$ be the geodesic with endpoints 
\[\frac{q}{e^{3 \ell(\alpha)}} = e^{(n-1)\ell(\alpha)} \text{ and }\frac{e^{3 \ell(\alpha)}}{q} = e^{-(n-1) \ell(\alpha)}\]
Then both $\gamma$ and $\partial g(\gamma)$ lie between $\tilde \delta_1$ and $\tilde \delta_2$. 

\begin{figure}[h!]
 \centering 
 \includegraphics{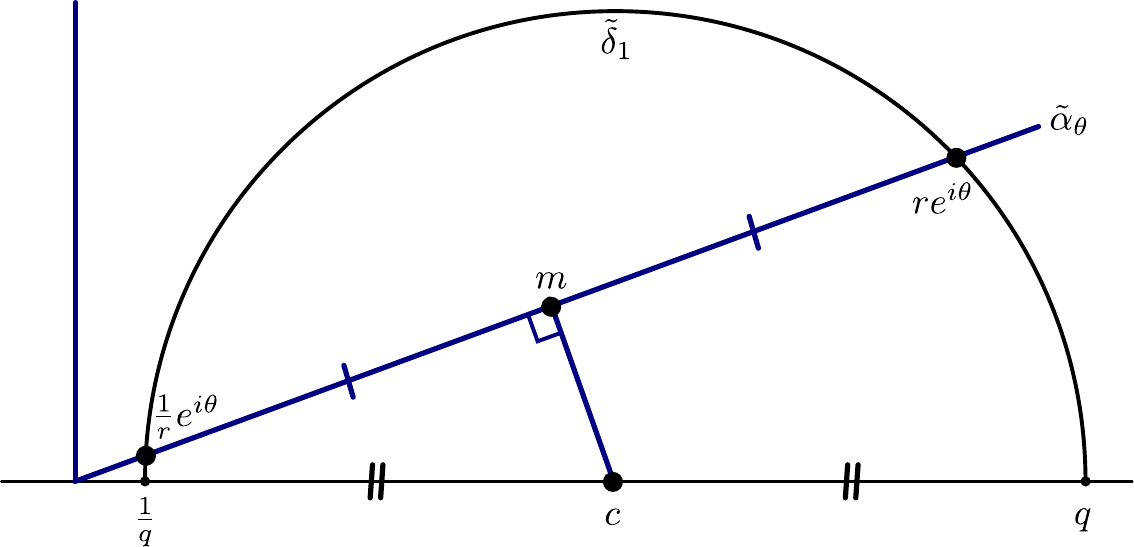}
 \caption{}
 \label{fig:Complement_length_delta_1_comp}
\end{figure}

Since $\tilde \delta_1$ lies above $\tilde \gamma$, it must intersect $\tilde \alpha_\theta$ (Figure \ref{fig:Complement_length_delta_1_comp}). In fact, since both $\tilde \alpha_\theta$ and $\delta_1$ are invariant under the isometry $z \mapsto \frac{1}{\bar z}$, it must intersect $\tilde \alpha_\theta$ at points $r e^{i \theta}$ and $\frac 1r e^{i\theta}$ for some $r > 0$. The midpoint along $\tilde \alpha_\theta$ of these two points is $m = \frac 12 (r + \frac 1r) e^{i \theta}$. If we draw a line perpendicular to $\tilde \alpha_\theta$ through $m$, it will intersect $\R$ in the Euclidean center $c$ of $\tilde \delta_1$. Thus, we see that 
\[
 c = \frac 12 (r + \frac 1r) \sec \theta
\]
On the other hand, the Euclidean center of the semicircle $\tilde \delta_1$ is $c = \frac 12 (q + \frac 1q)$. So we get that 
\begin{equation}
\label{eq:circle_equality}
 q+ \frac 1q = (r + \frac 1r) \sec \theta
\end{equation}

It is possible that $\tilde \delta_2$ does not intersect $\tilde \alpha_\theta$. In this case, there is some $x < 3$ for which the geodesic $\tilde \delta_2'$ with endpoints 
\[
 \frac{q}{e^{x \ell(\alpha)}} \text{ and } \frac{e^{x \ell(\alpha)}}{q} 
\]
is tangent to $\tilde \alpha_\theta$. In this case, by invariance under $z \mapsto \frac{1}{\bar z}$, we see that $\tilde \delta_2'$ is tangent to $\tilde \alpha_\theta$ at $e^{i \theta}$. The line perpendicular to $\tilde \alpha_{\theta}$ at $e^{i \theta}$ intersects $\R$ at the Euclidean center $c'$ of $\tilde \delta_2'$, so we see that $c' = \sec \theta$. In other words, 
\[
  \frac{q}{e^{x \ell(\alpha)}} + \frac{e^{x \ell(\alpha)}}{q} = \sec \theta
\]
Bounding the left-hand side from below by $e^{-x \ell(\alpha)}(q + \frac 1q)$, we deduce that 
\[
 e^{-x \ell(\alpha)} (r + \frac 1r) \sec \theta < \sec \theta
\]
Or, in other words, using that $x < 3$, 
\[
 r + \frac 1r < e^{3 \ell(\alpha)}
\]
In particular, $r < e^{3 \ell(\alpha)}$. The hyperbolic length of the segment of $\tilde \alpha_\theta$ between $re^{i \theta}$ and $\frac 1r e^{i \theta}$ is $2 \ln r \sin \theta$. Therefore, it is bounded above by $6 \ell(\alpha) \sin \theta$. We have previously computed that $\sin(\theta) = \tanh(\ell(\alpha))$. As $\ell(\alpha) \prec 1$, we have
\[
 \ell(\tilde \gamma \cap \tilde D(\alpha)) \prec 1
\]
In particular, if $\tilde \delta_2$ does not pass through $\tilde D(\alpha)$, then
\[
 \ell(\tilde \gamma \cap \tilde Z^c) -  \ell(\partial g(\tilde \gamma) \cap \tilde Z^c)\prec 1
\]
where we replace $\tilde D(\alpha)$ with $\tilde Z^c$ since the entire intersection between $\tilde \gamma$ and $\tilde Z^c$ is contained in $\tilde D(\alpha)$.

So suppose $\tilde \delta_2$ does intersect $\tilde \alpha_\theta$. Then because $\tilde \delta_2$ and $\tilde \alpha_\theta$ are invariant under $z \mapsto \frac{1}{\bar z}$, they must intersect at points $r' e^{i \theta}$ and $\frac 1r e^{i \theta}$ for some $r' > 0$. By the same argument as above, we see that 
\[
   \frac{q}{e^{3 \ell(\alpha)}} + \frac{e^{3 \ell(\alpha)}}{q} =(r' + \frac 1r') \sec \theta
\]
And since the left-hand side of this is bounded below by $e^{-3 \ell(\alpha)}(q + \frac 1q)$, we get that 
\[
 e^{-3 \ell(\alpha)}(r + \frac 1r)\sec \theta < (r' + \frac 1r') \sec \theta
\]
In particular, as $r + \frac 1r > r$ and $r' + \frac 1r' < 2 r'$, we see that $\frac{r}{r'} < 2e^{3 \ell(\alpha)}$. Thus, the length of the segment of $\tilde \alpha_\theta$ between $r' e^{i \theta}$ and $re^{i \theta}$ is at most $3\ln(2)\ell(\alpha) \sin \theta$, where again, $\ell(\alpha) \sin(\theta) \prec 1$. By symmetry, the same is true for the other pair of intersections of $\tilde \delta_1$ and $\tilde \delta_2$ with $\tilde \alpha_\theta$. Therefore, regardless of whether $\tilde \delta_2$ passes through $\tilde D(\alpha)$ or not,
\begin{equation}
\label{eq:length_change_in_C}
 \ell(\tilde \gamma \cap \tilde Z^c) - \ell(\partial g(\tilde \gamma) \cap \tilde Z^c) \prec 1
\end{equation}
where the constant depends only on $\chi(S)$. Again, we use that if $\tilde \gamma$ does not intersect $\tilde \alpha$, then $\tilde \gamma \cap \tilde D(\alpha) = \tilde \gamma \cap \tilde Z^c$, where $\tilde Z^c$ was a lift of both $Z^c$ and $Y^c$.

Now we are ready to show the statement of the lemma. Choose lifts $\tilde \gamma_1, \dots, \tilde \gamma_n$ of the $X$-geodesic representative of $\gamma$ to $\h$ that intersect $\tilde Z^c$, and so that $\tilde Z^c \cap (\tilde \gamma_1 \cup \dots \cup \tilde \gamma_n)$ is in 1-1 correspondence with $\gamma \cap (Y^c \cup C)$ in $X$. In particular, 
\[
 \ell_X(\gamma \cap Z^c) = \sum_{i=1}^n \ell(\tilde \gamma_i \cap \tilde Z^c)
\]
and, because $\tilde Z^c$ is also a lift of $Y^c \subset \pi(\gamma)$,
\[
 \ell_{\pi(\gamma)}(\gamma \cap Y^c) = \sum_{i=1}^n \ell(\partial g(\tilde \gamma_i) \cap \tilde Z^c)
\]

Note that each intersection of $\gamma$ with $\M$ corresponds to at most two lifts $\tilde \gamma_i$ and $\tilde \gamma_j$. Thus, the number of lifts is bounded above by $n \leq 2 i(\gamma, \M)$. So summing the inequalities in Equation (\ref{eq:Length_change_for_crossing}) and (\ref{eq:length_change_in_C}) (depending on whether $\tilde \gamma_i$ intersects a lift of $\partial Z$ or not) gives 
\[
 \ell_X(\gamma \cap Z^c) - \ell_{\pi(\gamma)}(\gamma \cap Y^c) \prec i(\gamma, \M) 
\]
\end{proof}

\subsection{Length on $Z^{bd}$}
\label{subsec:comparison_on_Yth}
Now that we have estimated the length of $\gamma$ outside of $Z^{bd}$, we prove the following:
\begin{lem}
\label{lem:comparison_on_Yth}
 We can estimate the length of the $X$-geodesic representative of $\gamma$ in $Z^{bd}$ by
 \[
  \ell_{X}(\gamma \cap Z^{bd}) \prec i(\gamma, \M)
 \]
 where the constant depends only on $\chi(S)$.
\end{lem}
\begin{proof}
We have that $\M$ cuts $\gamma \cap Z^{bd}$ into arcs $\gamma_1, \dots, \gamma_k$, for $k = i(\gamma, \M)$. We will show that each arc has length bounded above by a constant depending only on $\chi(S)$. 

Each arc $\gamma_i$ lies in some complementary region $R$ of $\M$ in $Z^{bd}$. First, suppose that $R$ is simply connected. The length of $\partial R$ is bounded above by $2 \sum \ell_X(\delta_i)$, where we sum over all $\delta_i \in \M$. In other words, it is bounded above by a constant depending only on $\chi(S)$. Then by the isoperimetric inequality, the diameter of $R$ is also bounded above in terms of the length of its boundary. As $\gamma_i$ is a geodesic arc joining two points on the boundary of $R$, this means that $\ell_X(\gamma_i) \prec 1$, for constants depending only on $\chi(S)$.

\begin{figure}[h!]
 \centering 
 \includegraphics{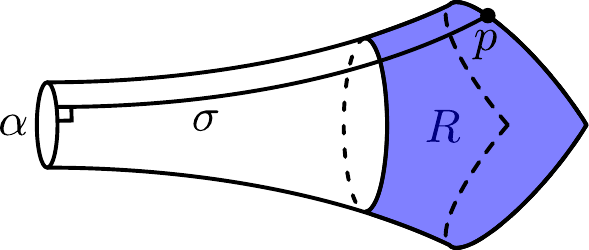}
 \caption{}
 \label{fig:Bounded_diameter_length_cyl_reg}
\end{figure}

Now suppose $R$ is not simply connected (Figure \ref{fig:Bounded_diameter_length_cyl_reg}). That means $R$ is a cylindrical region whose core curve is homotopic to some geodesic $\alpha$ on $\partial Z$. A priori, $\gamma_i$ could twist around $R$ multiple times. We must show that this does not occur. Let $p$ be a point on $\partial R$ that lies on some curve in $\M$. Drop a perpendicular $\sigma$ from $p$ to $\alpha$. We will choose $\sigma$ more carefully later. As $Z^{bd}$ has bounded diameter, we have that $\ell_X(\sigma \cap R) \prec 1$. So cutting $R$ along $\sigma$ gives a simply connected region of bounded diameter. 

We will show that we can choose a $\sigma$ that is disjoint from $\gamma_i$. Thus $\gamma_i$ will again be a geodesic arc crossing a simply connected region of bounded diameter.  This will give that $\ell_X(\gamma_i) \prec 1$ with constants depending only on $\chi(S)$.

\begin{figure}[h!]
 \centering 
 \includegraphics{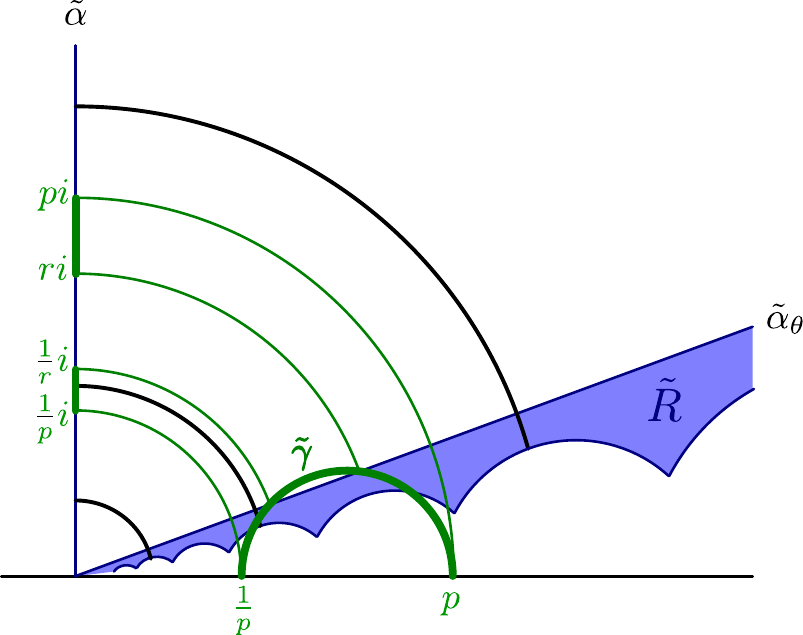}
 \caption{We bound the length of the intersection between $\tilde \gamma$ and $\tilde R$.}
 \label{fig:Bounded_diameter_length_intersection_estimate}
\end{figure}

We find $\sigma$ as follows. Consider the universal cover $\h$ of $X$, which we identify with the upper half plane (Figure \ref{fig:Complement_length_subsurface_lift}). Again, suppose $\alpha$ lifts to the imaginary axis, which we again denote $\tilde \alpha$. The collar neighborhood $D(\alpha)$ of $\alpha$ in $Z$ lifts to a region $\tilde D(\alpha)$ bounded on one side by $\tilde \alpha$, and on the other side by the ray $\tilde \alpha_\theta$ out of the origin making an angle $\theta$ with the real axis. Again, as the collar neighborhood has width $\Col_X(\alpha)$, we show in Lemma \ref{lem:Hyperbolic_geometry} that 
\[
 \tan \frac{\theta}{2} = e^{-\Col_X(\alpha)}
\]

Then $R$ lifts to a region $\tilde R$ bounded on one side by $\tilde \alpha_\theta$, and on the other by a piecewise geodesic curve where each point lies on the lift of some curve in $M_\gamma$ (Figure \ref{fig:Bounded_diameter_length_intersection_estimate}).

Let $\tilde \gamma$ be a lift of $\gamma$ so that a connected component of $\tilde \gamma \cap \tilde R$ projects down to $\gamma_i$. Let $\tilde \sigma$ be the set of all lifts of $\sigma$ to $\h$ (these are the black arcs in Figure \ref{fig:Bounded_diameter_length_intersection_estimate}). Then the number of intersections between $\gamma_i$ and $\sigma$ is bounded above by the number of lifts of $\sigma$ that intersect $\tilde \gamma \setminus \tilde D(\alpha)$. In fact, if $\tilde \gamma \cap \tilde R$ has two connected components, then we can ignore not just the intersects between $\tilde \gamma$ and $\tilde \sigma$ inside $\tilde D(\alpha)$, but also the intersects with the ``wrong'' connected component of $\tilde \gamma \cap \tilde R$.  

Consider the projection $P'$ of $\tilde \gamma \setminus \tilde D(\alpha)$ to $\tilde \alpha$. Then $P'$ is either one segment on the imaginary axis, or two, depending on whether $\tilde \gamma \cap \tilde R$ consists of one or two components, respectively. In Figure \ref{fig:Bounded_diameter_length_intersection_estimate}, $P'$ is the union of the two green segments on the imaginary axis. If $\tilde \gamma \cap \tilde R$ consists of two components, then let $\tilde \gamma_i$ be the component that is the lift of $\gamma_i$. Then let $P$ be the segment of $P'$ that contains the projection of $\tilde \gamma_i$. If $P'$ is a single segment, set $P = P'$.

If a lift of $\sigma$ intersects $\tilde \gamma_i$, then it intersects $P$. So we bound $\#P \cap \tilde \sigma$. We know that the distance between endpoints of lifts of $\sigma$ on $\tilde \alpha$ is exactly $\ell(\alpha)$ (where we suppress the dependence on the metric, as $\ell_{\pi(\gamma)}(\alpha) = \ell_X(\alpha)$). We will show that the length of $P$ is at most $\ell(\alpha)$. Then there will be a choice of $\sigma$ for which $\tilde \sigma$ is disjoint from the interior of $P$. Thus, we will be able to choose $\sigma$ disjoint from $\gamma_i$.

We first show that, without loss of generality, $\tilde \gamma$ is tangent to $\tilde \alpha_\theta$. As $\tilde \alpha$ and $\tilde \alpha_\theta$ are invariant under scaling, the length of $P$ will not change if we scale $\tilde \gamma$. So we can assume without loss of generality that the endpoints of $\tilde \gamma$ on $\R$ are at $p > 1$ and $\frac 1 p$. If $\tilde \gamma$ does not intersect $\tilde \alpha_\theta$, then $P = [\frac 1p i, pi]$, where $[ai, bi]$ denotes the set of points of the form $xi$ for $a \leq x \leq b$. If we increase $p$ until $\tilde \gamma$ is tangent to $\tilde \alpha_\theta$, then the length of $P$ will increase. So if $\tilde \gamma$ is disjoint from $\tilde \alpha_\theta$, we can assume it is tangent to $\tilde \alpha_\theta$.

On the other hand, suppose $\tilde \gamma$ intersects $\tilde \alpha_\theta$ at points $re^{i \theta}$ and, by symmetry, $\frac 1r e^{i \theta}$, for $r \geq 1$. Then $P' = [\frac 1p i, \frac 1r i] \cup [ri, pi]$, and, as the two segments are the same length, we can assume that
\[
 P = [ri, pi]
\]

Consider a family of curves with endpoints at $p$ and $q_t$ for $t \in [0,1]$, so that $q_0 = \frac 1p$, and the geodesic $\tilde \gamma_1$ with endpoints $q_1$ and $p$ is tangent to $\tilde \alpha_\theta$. We can choose $q_t$ to be continuously increasing in $t$. Let $\tilde \gamma_t$ be the geodesic with endpoints $q_t$ and $p$. Suppose it intersects $\tilde \alpha_\theta$ at $s_te^{i\theta}$ and $r_t e^{i \theta}$ with $s_t < r_t$. Since $q_t$ is increasing, for $t_1 < t_2$ we have $\tilde \gamma_{t_2}$ lies under $\tilde \gamma_{t_1}$. So we must have $r_{t_2} < r_{t_1}$. 

Let $P_t = [r_ti, pi]$. Then the length of $P_t$ increases as $t$ goes from 0 to 1. In particular, as $r_t$ is continuous in $t$, the length of $P_t$ must be maximized at $\tilde \gamma_1$, which is tangent to $\tilde \alpha_\theta$. Note that for all $t < 1$, $P_t$ is a connected component of the projection of $\tilde \gamma_t \setminus \tilde D(\alpha)$ to the imaginary axis. However, at $t = 1$, $P_1$ is only contained in this projection. Still, we only want an upper bound on the length of the projection, and so we can assume without loss of generality that $\tilde \gamma$ is tangent to $\tilde \alpha_\theta$.

We continue to assume that the endpoints of $\tilde \gamma$ on $\R$ are at $p$ and $1/p$. Since it is tangent to $\tilde \alpha_\theta$, it must meet $\tilde \alpha_\theta$ at $e^{i \theta}$.

We can then use Equation (\ref{eq:circle_equality}) in the proof of Lemma \ref{lem:comparison_on_Yc} with $r = 1$ to get
\[
 \frac{\frac 1p + p}{2} = \sec \theta
\]
Take $x>0$ so that $p = e^x$. Then we see 
\[
 \sec \theta = \cosh x
\]
But we show in Lemma \ref{lem:Hyperbolic_geometry} that $\sec(\theta) = \cosh\left (\frac{\ell(\alpha)}{2}\right)$. So as $x > 0$, this means 
\[
 x = \frac{\ell(\alpha)}{2}
\]

Now, $d(\frac 1p i, pi) = 2 \ln p = 2x$. In particular, the length of $P$ is 
\[
 d(\frac 1p i, pi) < \ell(\alpha)
\]

The distance between successive lifts of $\sigma$ is $\ell(\alpha)$. As we argued above, this means we can choose $\sigma$ that is disjoint from $\gamma_i$. Since $\sigma \cap R$ has length bounded above only in terms of $\chi(S)$, we see that $\gamma_i$ has length bounded above in terms of $\chi(S)$.

Since all of the arcs $\gamma_1, \dots, \gamma_k$ have length bounded in terms of $\chi(S)$, and $k = i (\gamma, \M)$, we get our result.
\end{proof}

\section{Thick-Thin Theorem}
\label{sec:Full_Thick_Thin}

To prove the main theorem, we will work with a filling current $\mu$, and its length minimizing metric $\pi(\mu)$. Given a thick component $Y$ of $\pi(\mu)$, we will consider its hyperbolically shortest marking $\Gamma$. That is, 
\[
 i(\pi(\mu), \Gamma) \asymp 1
\]
where the constant depends only on $\chi(S)$. The following lemma shows that the hyperbolically short marking $\Gamma$ is also $\mu$-short.
\begin{prop}
\label{lem:Hyp_short_marking_bounded_by_mu_systole}
 Let $\mu \in \C_{fill}(S)$, and let $Y$ be a thick component of its length minimizer $\pi(\mu)$. Let $\Gamma$ be the hyperbolically shortest marking on $Y$. Then 
 \[
  i(\mu, \Gamma) \asymp \sys_Y(\mu)
 \]
\end{prop}
\begin{proof}
Since $Y$ is a thick component of $\pi(\mu)$, let $\M$ be the $\mu$-short marking of $Y$ given by Lemma \ref{lem:mu-short_marking}. By Theorem \ref{thm:Another_Collar}, we have that for any $\alpha$ in $\Gamma$,
\[
 i(\mu, \alpha) \Col_{\pi(\mu)}(\alpha) \prec i(\mu, \M)
\]
But by Lemma \ref{lem:mu-short_marking}, $i(\mu, \M) \prec \sys_Y(\mu)$. Moreover, $Y$ is a thick component, and $\alpha$ is part of the hyperbolically shortest marking on $\pi(\mu)$. Thus, $\Col_{\pi(\mu)}(\alpha)$ is bounded below by a constant depending only on $\chi(S)$. Therefore, 
\[
 i(\mu, \alpha) \prec \sys_Y(\mu)
\]
for constants depending only on $\chi(S)$. Since the number of curves in $\Gamma$ also depends only on $\chi(S)$, summing over all the $\alpha \in \Gamma$ gives us the proposition.
\end{proof}

We are now ready to prove the main theorem, comparing the $\mu$- and $\pi(\mu)$- lengths of simple closed curves in the thick part of $\pi(\mu)$.
\begin{MainThm}
 Let $\pi(\mu)$ be the length minimizer of $\mu \in \C_{fill}(S)$. Let $Y$ be a thick component of $\pi(\mu)$. Then for any essential simple closed curve $\alpha$ in $Y$, we have 
 \[
\frac{i(\mu, \alpha)}{i(\pi(\mu), \alpha)} \asymp \sys_Y(\mu)
 \]
\end{MainThm}
\begin{proof}
Let $\Gamma$ be a hyperbolically shortest marking of $Y$. By Proposition \ref{prop:Thick-thin_without_collar_lemma}, we have 
 \[
  \frac{\sys_Y(\mu)^2} {i(\mu, \Gamma)} \prec \frac{i(\mu, \alpha)}{i(\pi(\mu), \alpha)} \prec i(\mu, \Gamma)
 \]
But by Proposition \ref{lem:Hyp_short_marking_bounded_by_mu_systole}, we have $i(\mu, \Gamma) \asymp \sys_Y(\mu)$. Therefore, the theorem follows immediately.
\end{proof}

\section{When curves are short}
\label{sec:When_curves_short}
We can use our collar theorems to characterize when curves in $\pi(\mu)$ are short for any $\mu \in \C_{fill}(S)$. First we prove
\begin{ThmShortCurve}
  For every $\epsilon > 0$, there are constants $N_1, N_2$ so that for any $\mu \in PC_{fill}(S)$, and any simple closed curve $\alpha$,
 \begin{enumerate}
  \item If $\ell_{\pi(\mu)}(\alpha) < \epsilon$, then for any simple closed curve $\beta$ with $i(\alpha, \beta) < 1$,
  \[
   i(\mu, \beta) > N_1 i(\mu, \alpha)
  \]
\item If for any simple closed curve $\beta$ with $i(\alpha, \beta) < 1$,
  \[
   i(\mu, \beta) > N_2 i(\mu, \alpha)
  \]
  then $\ell_{\pi(\mu)}(\alpha) < \epsilon$.
 \end{enumerate}
\end{ThmShortCurve}
This theorem is a direct consequence of Corollary \ref{cor:Short_curve_part1} and Lemma \ref{lem:Short_curve_guarantee} below.

First, we need an improvement of Lemma \ref{lem:mu-short_marking}, as follows:
\begin{lem}
\label{lem:Thm:short_curve_part1}
Let $\mu \in \C_{fill}(S)$. Then there is some $C>0$ depending only on $\chi(S)$ so that, for any simple closed curve $\alpha$ with $\ell_{\pi(\mu)}(\alpha)<C$, and for any simple closed curve $\beta$ that intersects $\alpha$, we have 
\[
 i(\mu, \alpha) \Col_{\pi(\mu)}(\alpha) \prec i(\mu, \beta)
\]
where all constants depend only on $\chi(S)$.
\end{lem}
\begin{proof}
 First, note that we can assume that $\beta$ intersects $\alpha$ minimally. If not, we can do surgery on $\beta$ to find a new curve $\beta'$ that does intersect $\alpha$ minimally, and for which $i(\mu, \beta') < i(\mu, \beta) + i(\mu, \alpha)$. Then, as $\alpha$ is relatively short, $\Col_{\pi(\mu)}(\alpha)$ is large enough, so that if the equation $ i(\mu, \alpha) \Col_{\pi(\mu)}(\alpha) \prec i(\mu, \beta')$ holds for $\beta'$, then it holds for $\beta$ with a different constant.
 
 Let $Y$ be the smallest union of thick subsurfaces of $\pi(\mu)$ that contains $\alpha$ and $\beta$. In other words, the lengths of the boundary components of $Y$ are all bounded above by $c_b$, and if $\gamma$ is an essential simple closed curve in $Y$ disjoint from $\alpha$ and $\beta$, then $\gamma$ has length greater than $c_b'$, where $c_b$ is the Bers constant for pants decompositions, and $c_b'$ is the Bers constant for individual curves.
 
 We then claim that we can find a marking $\M$ for $Y$ where $\alpha$ is a pants curve, and $\beta$ is its transverse curve, so that
 \[
  i(\mu, \M) \prec i(\mu, \alpha) + i(\mu, \beta)
 \]
In fact, the proof of this is the same as the proof of Lemma \ref{lem:mu-short_marking}. In terms of the proof of that lemma, let $\gamma_0 = \alpha$ and $\gamma_1 = \beta$. Then suppose we can find curves $\gamma_0, \dots, \gamma_{2i}$ for which $\gamma_0, \gamma_2, \dots, \gamma_{2i}$ are all disjoint (and part of a future pants decomposition of $Y$), and for all $j = 1, \dots, i-1$, $\gamma_{2j+1}$ is disjoint from $\gamma_{2k}$ for $k \neq j$, and intersects $\gamma_{2j}$ minimally. Moreover, suppose that
\begin{enumerate}
 \item $i(\mu, \gamma_i) \prec i(\mu, \alpha) + i(\mu, \beta)$
 \item the pants curves are not too short: $\ell_{\pi(\mu)}(\gamma_{2j}) > c_b'$ for all $j \geq 1$.
\end{enumerate}
Clearly this holds for $j = 0$. Then the proof of Lemma \ref{lem:mu-short_marking} produces curves $\gamma_{2j+1}$ and $\gamma_{2j+2}$ for which condition 1 holds. This is because the term $i(\mu, \alpha) + i(\mu, \beta)$ takes the the role of $\sys_Y(\mu)$ in the proof of the lemma, and otherwise, those are the only two conditions we use. If $\ell_{\pi(\mu)}(\gamma_{2j+2}) < c_b'$, then $\gamma_{2j+2}$ is a boundary curve of $Y$ (as it cannot be $\alpha$), so we no longer use it in the construction. So we can continue to build the marking $\M$ in this way until we are done.

Now, by Theorem \ref{thm:Another_Collar}, since $Y$ has boundary curves whose lengths are bounded by $c_b$, 
\[
 i(\mu, \alpha) \Col_{\pi(\mu)}(\alpha) \prec i(\mu, \M)
\]
But $i(\mu, \M) \prec i(\mu, \alpha) + i(\mu, \beta)$, and by assumption, $\Col_{\pi(\mu)}(\alpha)$ is large enough so that this implies 
\[
 i(\mu, \alpha) \Col_{\pi(\mu)}(\alpha) \prec i(\mu, \beta)
\]

\end{proof}

As a corollary, we get the first part of Theorem \ref{thm:Short_curves}.
\begin{cor}
\label{cor:Short_curve_part1}
 Let $\mu \in \PC_{fill}(S)$. Let $\alpha$ be any simple closed curve. Then for all $\epsilon >0$ small enough, there is an $N_1 > 0$ so that if $\ell_{\pi(\mu)}(\alpha) < \epsilon$, then for any $\beta$ that intersects $\alpha$, 
 \[
  i(\mu, \beta) > N_1 i(\mu, \alpha)
 \]
\end{cor}
Note that $N_1$ in this corollary is proportional to the collar width of $\alpha$.

Next, we prove the other direction of Theorem \ref{thm:Short_curves}. 
To prove that a curve $\alpha$ is short in $\pi(\mu)$, we need to show that the $\mu$-length of every transverse curve $\beta$ is much bigger than the $\mu$-length of $\alpha$.

\begin{lem}
\label{lem:Short_curve_guarantee}
 Let $\mu \in \C_{fill}(S)$. Let $\alpha$ be a simple closed curve. Then for every $\epsilon > 0$, there exists an $N_2 > 0$ so that, if for all curves $\beta$ with $i(\alpha,\beta) > 0$, 
 \[
  i(\mu, \beta) > N_2 i(\mu, \alpha)
 \]
 then $\ell_{\pi(\mu)}(\alpha) < \epsilon$. Note that $N$ depends only on $\epsilon$, and not on $\mu$ or $\alpha$.
\end{lem}
\begin{proof}
First, we reduce this lemma to the case where $\beta$ is an orthogonal arc system for $\alpha$. Let $\gamma$ be an orthogonal arc system for $\alpha$, and suppose 
\[
 i(\mu, \gamma) < (N-1) i(\mu, \alpha)
\]
By the argument in the proof of Lemma \ref{lem:mu-short_marking}, we can use $\gamma$ to construct a simple closed curve $\beta$ for which 
\[
 i(\mu, \beta) < i(\mu, \gamma) + i(\mu, \alpha)
\]
Thus, 
\[
 i(\mu, \beta) < N i(\mu, \alpha)
\]
In other words, if $i(\mu, \beta) > N i(\mu, \alpha)$ for all curves $\beta$ that intersect $\alpha$, then $i(\mu, \gamma) > (N-1) i(\mu, \alpha)$ for all orthogonal arc systems $\gamma$ for $\alpha$.

In Lemma \ref{lem:Shortest_arc_length_bound}, we show that if $\alpha$ is a simple closed curve, then there exists an orthogonal arc system $\gamma$ with 
\[
 \ell_{\pi(\mu)}(\gamma) \prec \Col_{\pi(\mu)}(\alpha) + 1
\]
where the constants depend only on $S$, and $\Col_{\pi(\mu)}(\alpha) = \sinh^{-1}\left (\frac{1}{\sinh(\ell(\alpha)/2)}\right)$ is the constant from the collar lemma. Note that as the length of $\alpha$ increases, the bound on the right decreases to 1. Thus, given $\epsilon > 0$, there is some $L > 1$ so that if $\ell(\gamma) > L$ for every orthogonal arc system $\gamma$, then $\ell(\alpha) < \epsilon$. 

By Proposition \ref{prop:Mixed_Collar_Orthogonal_Arcs}, there is a decreasing function $\delta : \R_+ \to \R_+$ so that for any orthogonal arc system $\gamma$ for $\alpha$,
 \[
  i(\mu, \alpha) > \delta(\ell_{\pi(\mu)}(\gamma)) i(\mu, \gamma)
 \]
Moreover, $\delta = \delta(x)$ is a decreasing function that goes to 0 as $x$ goes to infinity. In other words, there is some $\delta_0$ so that $\delta(x) < \delta_0$ implies $x > L$, for the constant $L$ above.

So let $N = 1/\delta_0$ and suppose 
\[
 i(\mu, \gamma) > \frac{1}{\delta_0} i(\mu, \alpha)
\]
for all orthogonal arc systems $\gamma$. Then we would have $i(\mu, \alpha) > \frac{\delta(\ell_{\pi(\mu)}(\gamma))}{\delta_0} i(\mu, \alpha)$, or in other words,
\[
 \delta(\ell_{\pi(\mu)}(\gamma)) < \delta_0
\]
This implies that for every orthogonal arc system $\gamma$, $\ell_{\pi(\mu)}(\gamma) > L$, which in turn implies that $\ell_{\pi(\mu)}(\alpha) < \epsilon$.
\end{proof}
Essentially, we show that if $i(\mu, \beta) > N i(\mu, \alpha)$, then every orthogonal arc system for $\alpha$ has very long hyperbolic length.

Once we know which curves are short in $\pi(\mu)$ using Lemma \ref{lem:Short_curve_guarantee}, we can prove that the rest of the surface is thick. We recall Theorem \ref{thm:Thick_part_guarantee}, that shows that if a subsurface has a marking where all curves have roughly equal $\mu$-length, then it is thick in $\pi(\mu)$.
\begin{ThmThickPart}
 Let $Y \subset \pi(\mu)$ be a subsurface so that $\ell_{\pi(\mu)}(\beta) < c_b$ for each boundary component $\beta$ of $Y$, where $c_b$ is the Bers constant. Suppose $\Gamma$ is a marking of $Y$ so that 
 \[
  i(\mu, \gamma) \asymp \sys_Y(\mu)
 \]
for all $\gamma \in \Gamma$. Then for each essential simple closed curve $\alpha$ in $Y$,
\[
 \ell_{\pi(\mu)}(\alpha) \succ 1
\]
where all constants depend only on $S$.
\end{ThmThickPart}
\begin{proof}
 Since all the boundary curves of $Y$ have length bounded above by the Bers constant, we can apply Theorem \ref{thm:Another_Collar}, which says that for any essential simple closed curve $\alpha$ in $Y$,
 \[
  i(\mu, \alpha) \Col_{\pi(\mu)}(\alpha) \prec i(\mu, \Gamma)
 \]
 for a constant depending only on $S$.
 
 Since $i(\mu, \Gamma) \asymp \sys_Y(\mu)$, and $i(\mu, \alpha) \geq \sys_Y(\mu)$, by definition, we have that 
 \[
  \Col_{\pi(\mu)}(\alpha) \prec 1
 \]
 But if the collar width of $\alpha$ is bounded above, then the length of $\alpha$ is bounded form below. Thus,
 \[
  \ell_{\pi(\mu)}(\alpha) \succ 1
 \]
\end{proof}

\section{Appendix: Hyperbolic geometry}
\label{sec:Appendix}
In Section \ref{sec:Mixed_length_comparisons}, we needed a few identities from hyperbolic geometry. These are standard computations in hyperbolic geometry, but we briefly prove them here for completeness. 

We return to the setup in Figure \ref{fig:Complement_length_subsurface_lift}. Let $\alpha$ be a simple closed geodesic on a complete hyperbolic surface $X$, and identify the universal cover of $X$ with the hyperbolic plane $\h$ with the upper half-plane model. Let $\tilde \alpha$ be a lift of $\alpha$. Without loss of generality, $\tilde \alpha$ is the imaginary axis. Let $\tilde \alpha_\theta$ be the ray out of the origin making an angle of $\theta$ with the positive real axis. Suppose the distance from each point on $\tilde \alpha_\theta$ to $\tilde \alpha$ is $\Col_X(\alpha)$, for 
\[
 \Col_X(\alpha) = \sinh^{-1}\left (\frac{1}{\sinh(\frac 12 \ell_X(\alpha))} \right )
\]
\begin{lem}
\label{lem:Hyperbolic_geometry}
 We have the following identities:
 \begin{enumerate}
  \item $\tan\left( \frac \theta 2 \right)= e^{-\Col_X(\alpha)}$
  \item $\sin(\theta) = \tanh\left (\frac{\ell(\alpha)}{2}\right)$
  \item $\sec(\theta) = \cosh\left (\frac{\ell(\alpha)}{2}\right)$
 \end{enumerate}
\end{lem}
\begin{proof}
 First, consider a point $p =re^{i \theta}$ on $\tilde \alpha_\theta$. The perpendicular from $p$ to $\tilde \alpha$ is the arc $\gamma : [\theta, \pi/2] \to \h$ with 
 \[
  \gamma(t) = re^{i t}
 \]
The length of $\gamma$ is given by 
\[
 \int_{t = \theta}^{\pi/2} \frac{1}{\sin t} dt = -\ln(\tan(\theta/2))
\]
But we know the distance from $p$ to $\tilde \alpha$ is $\Col_X(\alpha)$, and so 
\[
 \tan\left( \frac \theta 2 \right ) = e^{-\Col_X(\alpha)}
\]

Next, we compute $\sin(\theta)$. We have $\sin(\theta) = 2 \sin\left (\frac \theta 2 \right ) \cos \left (\frac \theta 2 \right )$, where the formula for $\tan(\theta/2)$ gives us
\[
 \sin\left( \frac \theta 2 \right ) = \frac{e^{-\Col_X(\alpha)}}{\sqrt{1 + e^{-2\Col_X(\alpha)}}}, \text{ and } \cos\left( \frac \theta 2 \right ) = \frac{1}{\sqrt{1 + e^{-2\Col_X(\alpha)}}}
\]
Thus,
\begin{align*}
 \sin \theta & = \frac{2e^{-\Col_X(\alpha)}}{1 + e^{-2\Col_X(\alpha)}} \\
  &= \frac{2}{e^{\Col_X(\alpha)} + e^{-\Col_X(\alpha)}}\\
  & = \frac{1}{\cosh(\Col_X(\alpha))}
\end{align*}

So we must expand $\cosh(\Col_X(\alpha))$. We use that $\cosh(\sinh^{-1}(x)) = \sqrt{1 +x^2}$, and get 
\begin{align*}
 \cosh(\Col_X(\alpha)) & = 
 \sqrt {
 1 + 
 \frac{1}
 {\sinh^2(\frac{\ell(\alpha)}{2})}
 } \\
 & = \frac{1}{\tanh\left (\frac{\ell(\alpha)}{2}\right)}
\end{align*}

Therefore, 
\[
 \sin(\theta) = \tanh\left (\frac{\ell(\alpha)}{2}\right)
\]

Lastly, we use that if $\sin(x) = a$, then $\cos(x) = \sqrt{1-a^2}$. Then, since $1-\tanh^2(x) = \frac{1}{\cosh^2(x)}$, we get 
\[
 \sec(\theta) = \cosh\left (\frac{\ell(\alpha)}{2}\right)
\]

\end{proof}

 \bibliographystyle{alpha}
  \bibliography{ProjectionProperties2}
\end{document}